\documentclass[11pt, a4paper]{amsart}
\calclayout

\usepackage{lmodern}		
\usepackage[T1]{fontenc}		
\usepackage{tikz-cd} 
\newcommand{\degree}{\ensuremath{^\circ}}	
\usepackage{amsmath}		
\allowdisplaybreaks
\usepackage{comment}
\usepackage{float}			
\usepackage{lastpage}		
\usepackage[colorlinks]{hyperref}
\usepackage{amsthm}			
\usepackage{hyperref}
\usepackage{caption}
\usepackage{mathrsfs}
\usepackage{amssymb}
\usepackage{amsthm}
\usepackage{enumerate}
\usepackage{mathtools}
\usepackage{amssymb}
\usepackage{dsfont}
\usepackage{esint}
\numberwithin{equation}{section}		
\newtheorem{theorem}{Theorem}[section]

\newtheorem{prop}[theorem]{Proposition}
\newtheorem{corollary}[theorem]{Corollary}
\newtheorem{lemma}[theorem]{Lemma}

\theoremstyle{definition}				
\newtheorem{remark}[theorem]{Remark}
\newtheorem{notation}[theorem]{Notation}

\newtheorem{definition}[theorem]{Definition}

\theoremstyle{theorem}					
\newcounter{mtheorem}
\newtheorem{mtheorem}[mtheorem]{Theorem}

\newtheorem{mcor}[mtheorem]{Corollary}

\usepackage{xfrac}

\newcommand{\N}{\ensuremath{\mathbb{N}}}

\newcommand{\Z}{\ensuremath{\mathbb{Z}}}

\newcommand{\R}{\ensuremath{\mathbb{R}}}
\newcommand{\W}{\ensuremath{\mathcal{W}}}

\newcommand{\OO}{\ensuremath{\mathcal{O}}}
\newcommand{\X}{\ensuremath{\mathfrak{X}}}
\newcommand{\UU}{\ensuremath{\mathfrak{U}}}

\newcommand{\C}{\ensuremath{\mathbb{C}}}
\newcommand{\CC}{\ensuremath{\mathscr{C}}}
\newcommand{\CCbar}{\ensuremath{\overline{\mathscr{C}}}}
\newcommand{\CCone}{\ensuremath{\mathcal{C}}}
\newcommand{\T}{\ensuremath{\mathbb{T}}}

\newcommand{\E}{\ensuremath{\mathcal{E}}}

\newcommand{\LL}{\ensuremath{\mathscr{L}}}

\newcommand{\p}{\ensuremath{\mathfrak{p}}}
\newcommand{\ttt}{\ensuremath{\mathfrak{t}}}
\newcommand{\gl}{\ensuremath{\mathfrak{gl}}}

\newcommand{\uu}{\ensuremath{\mathfrak{u}}}

\newcommand{\norm}[1]{\ensuremath{\left\lVert #1 \right\rVert}}

\newcommand{\Norm}[2]{\ensuremath{\left\lVert #1 \right\rVert_{#2}}}

\newcommand{\gen}[1]{\ensuremath{\left\langle #1 \right\rangle}}

\renewcommand{\leq}{\leqslant }
\renewcommand{\geq}{\geqslant }
\renewcommand{\phi}{\varphi}
\renewcommand{\rho}{\varrho}
\renewcommand{\P}{\textbf{P}}

\DeclareMathOperator{\U}{U}
\DeclareMathOperator{\SU}{SU}

\DeclareMathOperator{\re}{Re}
\DeclareMathOperator{\Lie}{Lie}

\DeclareMathOperator{\Span}{span}
\DeclareMathOperator{\dist}{dist}
\DeclareMathOperator{\length}{length}
\DeclareMathOperator{\diam}{diam}

\DeclareMathOperator{\tr}{Tr}

\DeclareMathOperator{\Rm}{Rm}
\DeclareMathOperator{\Aut}{Aut}
\DeclareMathOperator{\Isom}{Isom}
\DeclareMathOperator{\GL}{GL}

\DeclareMathOperator{\Id}{Id}

\DeclareMathOperator{\Ric}{Ric}

\DeclareMathOperator{\vol}{Vol}
\DeclareMathOperator{\dvol}{dvol}

\DeclareMathOperator{\Hess}{Hess}
\DeclareMathOperator{\Scal}{Scal}

\DeclareMathOperator{\diag}{diag}

\DeclareMathOperator{\eucl}{Eucl}

\DeclareMathOperator{\Scl}{Scl}

\DeclareMathOperator{\loc}{loc}
\newcommand{\del}{\partial}
\newcommand{\delbar}{\overline{\partial}}
\usepackage{breqn}
\usepackage{xstring}

\usepackage[
backend=biber,
maxnames=99,
style=numeric,
citestyle=numeric,
doi=false,
isbn=false,
]{biblatex}
\DeclareNameAlias{author}{family-given}
\DeclareNameAlias{editor}{family-given}

\DeclareFieldFormat{mrnumber}{\textbf{MR\space}\href{https://www.ams.org/mathscinet-getitem?mr=#1}{#1}}

\newcommand{\spliturl}[1]{%
	\StrBefore{#1}{,}[\urlfirstpart]%
	\StrBehind{#1}{,}[\urlsecondpart]%
	\href{https://www.ams.org/mathscinet-getitem?mr=\urlfirstpart}{MR\urlfirstpart}%
	\ifdefstring{\urlsecondpart}{}{}{, \href{https://zbmath.org/?q=an:\urlsecondpart}{Zbl \urlsecondpart}}%
}

\DeclareFieldFormat{url}{\spliturl{#1}}

\PassOptionsToPackage{hyphens}{url}
\AtEveryBibitem{%
	\clearfield{month}%
	\clearfield{day}%
	\clearfield{pages}
	\clearfield{pagetotal}
}

\title{A Liouville Theorem and $C^{\alpha}$-Estimate for Calabi-Yau Cones}
\author[Klemmensen]{Johan Jacoby Klemmensen}
\address{Department of Mathematics\\
	University of Münster\\
	48149 Münster, Germany}
\email{\href{mailto:j.klemmensen@uni-muenster.de}{j.klemmensen@uni-muenster.de}}

\begin{document}

\maketitle

\begin{abstract}
	Let $(\CC, \omega_\CC)$ be a Ricci-flat, simply connected, conical Kähler manifold. We establish a Liouville theorem for constant scalar curvature Kähler (cscK) metrics on $\CC$. The theorem asserts that any cscK metric $\omega$ satisfying the uniform bound $\frac{1}{C} \omega_\CC \leq \omega \leq C \omega_\CC$ for some $C\geq1$ is equal to $\omega_\CC$ up to a holomorphic automorphism that commutes with the scaling action of the cone structure.
	
	Next, we develop a $C^{0,\alpha}$-estimate for uniformly bounded Kähler metrics on a ball around the apex, using a Hölder-type seminorm inspired by Krylov. This estimate applies for small $\alpha > 0$ under the assumption of uniformly bounded scalar curvature. 
	
	As a corollary of this result, we show that such a Kähler metric $\omega$ is asymptotic to the Ricci-flat cone metric $\omega_\CC$, with polynomial decay rate $r^\alpha$ and for sufficiently small $\alpha > 0$.
\end{abstract}

\tableofcontents

\section{Introduction}

\subsection{Background and Motivation}
One of the first use cases of Liouville Theorems and scaling arguments to prove Schauder estimates goes back to the work of Simon \cite{simonSchauder1997} in his work on linear elliptic operators on $\R^m$. Later, similar results were developed for Kähler metrics satisfying the complex Monge-Ampère equation. In this context, the classical Liouville Theorem for Kähler metrics on $\C^m$ dates back to Riebesehl-Schulz \cite{riebesehlpriori1984}, which states that any Kähler metric with a constant determinant and uniformly equivalent to the Euclidean metric, is the pullback of the Euclidean metric by an element in $\GL(m,\C)$. Utilizing their results and rescaling methods, Chen-Wang provided a new proof \cite{chenalpha2015} of the $C^{2,\alpha}$-estimate for solutions of the complex Monge-Ampère equation on Euclidean space. A similar statement for Kähler metrics on $\C^m$ with singularities along a divisor was given in \cite{chenregularity2017} by Chen-Wang, with the Liouville Theorem stemming from a strengthening of the Liouville Theorems found in \cite{chenKahlerEinstein2014a} and \cite{chenlong2018}. To prove $C^{k,\alpha}$-estimates on product Kähler manifolds with collapsing fibers, more elaborate scaling arguments were presented in \cite{heinHigherorder2020} by Hein-Tosatti. The necessary Liouville Theorem is found in \cite{heinLiouville2019} by Hein, with a simplified proof by Li-Li-Zhang \cite{liMean2020}.\\

Tangentially to this, Hein-Sun proved \cite{heinCalabiYau2017} the following result: Given a complex projective variety such that the singularities are isolated and locally isomorphic to Calabi-Yau cones, then, if these cones are in addition smoothable and strongly regular, any Ricci-flat Kähler metric on the variety asymptotically approaches the given Ricci-flat Kähler cone metric. This work was later extended by Chiu-Székelyhidi in \cite{chiuHigher2023} to the case of a tangent cone that is not necessarily locally isomorphic to the germ of the singularity and with a non-smooth cross-section, thus in particular dropping the strongly regular but not the smoothable assumption.\\

This paper shows that, on any simply-connected Calabi-Yau cone with a smooth cross-section, polynomial convergence follows from uniformly bounded scalar curvature and uniform equivalence to the cone metric. For $\alpha>0$ small enough, then any such $\omega$ is asymptotic to $\omega_\CC$ at polynomial rate $r^{\alpha}$ and up to some holomorphic automorphism $\Psi$ commuting with scaling. Smoothability of the cone and Donaldson-Sun theory \cite{donaldsonGromov2017} are not needed for this. In a sense, these assumptions get replaced by the uniform equivalence, and we then prove an Evans-Krylov estimate, which ultimately implies polynomial convergence at the apex as a corollary.
The proof proceeds using the scaling methods presented in \cite{heinHigherorder2020}, where the proof relies on blowup arguments and applying Liouville theorems. The quintessential Liouville Theorem for this argument is Theorem \ref{LiouvilleTheoremM}, which describes the global behavior of constant scalar curvature Kähler (cscK) metrics on Calabi-Yau conical manifolds $\CC$ (see Definition \ref{CalabiYauConeDefinition}).

\subsection{Structure of the Paper}

We present the following Liouville Theorem for cscK metrics on Ricci-flat Kähler/Calabi-Yau cones $(\CC,\omega_\CC)$. For $\CC = \C^m\setminus\{0\}$, the proof presented is novel and independent of \cite{riebesehlpriori1984} and \cite{liMean2020}, where Ricci-flatness instead of cscK is also required. Let $\Aut_\Scl(\CC)$ denote the group of holomorphic automorphisms commuting with the action of the scaling vector field.

\begin{mtheorem}[Theorem \ref{LiouvilleTheorem}]\label{LiouvilleTheoremM}
	Let $(\CC,\omega_\CC)$ be a Calabi-Yau cone. Take any cscK metric $\omega$ on $\CC$ such that there exists some constant $C>0$ with
	\begin{equation*}
		\frac{1}{C} \omega_\CC \leq \omega \leq C \omega_\CC.
	\end{equation*}
	Then $\omega =  \Psi^*\omega_\CC$ for some cone automorphism $\Psi\in \Aut_\Scl(\CC)$.
\end{mtheorem}

The outline of the proof is as follows: By rescaling, the basic theory of pluriharmonic functions and a Harnack inequality in a neighborhood of the apex show that $\omega^m = \omega_\CC^m$. Section \ref{SectionLiouvilleTangentCones} uses heat kernel estimates to show that the asymptotic limits of $\omega$ at $o$ and $\infty$ are Riemannian cone metrics, henceforth called the tangent cones. The tangent cone metrics have associated Reeb fields, and the classification by Hein-Sun \cite{heinCalabiYau2017} of holomorphic vector fields commuting with $r\del_r$ restricts the choice of Reeb fields. This is presented in Section \ref{SectionEqualReebFields}. Section \ref{SectionEqualTangentCones} uses the work of Martelli-Sparks-Yau \cite{martelliSasaki2008} to conclude that the tangent cones have the same Reeb field as $\omega_\CC$. Applying the arguments of Bando-Mabuchi \cite{bandoUniqueness1987}, suitably generalized by Nitta-Sekiya \cite{nittaUniqueness2012}, proves in Section \ref{SectionEqualTangentCones} that the tangent cones and $\omega_\CC$ are related by an automorphism in $\Aut_\Scl(\CC)$. This is exactly the statement of Theorem \ref{LiouvilleTheoremM} applied to the tangent cones. Finally, Section \ref{SectionCombiningTangentConesToLiouville} reapplies the heat kernel estimates from Section \ref{SectionLiouvilleTangentCones} to $\omega$ to prove the Liouville Theorem in full generality. \\

In Section \ref{SectionCalpha}, a novel $C^{0,\alpha}$-type seminorm is presented and used to prove a $C^{0,\alpha}$-estimate of another Kähler metric at the apex $o$.

\begin{mtheorem}[Theorem \ref{TheoremHolderBound}]\label{mTheoremHolderBound}
	Let $(\CC,\omega_\CC)$ be a Calabi-Yau cone with cone metric $\omega_\CC$, and let $\omega$ be a Kähler metric on $B_3(o)\subset \CC$ such that
	\begin{equation}\label{ConditionUniformBoundm}
		\frac{1}{C}\omega_\CC \leq \omega\leq C\omega_\CC, \quad \Norm{\Scal(\omega)}{0, B_3(o)} \leq D,
	\end{equation}
	for some constants $C,D>0$. Then for any $\alpha = \alpha(\CC,\omega_\CC)>0$ small enough, there exists a constant $C'$ with
	\begin{equation*}
		[\omega]'_{\alpha,B_1(o),\Sigma_{3C}^2\times \Sigma_{\loc}^2}\leq C',
	\end{equation*}
	where $\alpha= \alpha(\CC,\omega_\CC)$ and $C'= C'(\CC,C, D, \alpha,\omega_\CC, \UU)$ are independent of $\omega$.
\end{mtheorem}

A similar estimate for the Laplacian acting on functions can be found in Proposition \ref{SchauderLinear}. $\UU$ is a choice of covering given in Section \ref{SectionHolderIntro}. The $L^\infty$-scalar curvature bound ensures that a blowup of $\omega$ at the apex is a cscK metric, allowing us to use Theorem \ref{LiouvilleTheoremM} in the proof of Theorem \ref{mTheoremHolderBound}.\\

The definition of $[\omega]'_{\alpha,B_1(o),\Sigma_{3C}^2\times \Sigma_{\loc}^2}$ can be found in Definition \ref{DefinitionHolderNorm}. It aims to measure the weighted distance between $\omega$ and the comparison set $\Sigma_{3C}^2 \times \Sigma_{\loc}^2$ in the $C^0$-norm. $\Sigma_{3C}^2$ contains pullbacks of the cone metric $\omega_\CC$ by holomorphic automorphisms commuting with scaling. $\Sigma_{\loc}^2$ contains notions of constant 2-form that are only locally defined, and $\UU$ is a set of open sets where elements in $\Sigma_{\loc}^2$ are defined.. If $\CC\cup \{o\} = \C^m$ and the comparison set consists of constant 2-forms, the seminorm is equivalent to the usual $C^{0,\alpha}$-seminorm (see \cite[Theorem 3.3.1]{krylovLectures1996}). To prove the theorem, we follow the ideas outlined in \cite{heinHigherorder2020} by first assuming that $[\omega_i]'_{\alpha,B_1(o),\Sigma_{3C}^2 \times \Sigma_{\loc}^2}$ is unbounded for some sequence $(\omega_i)$, and then blowing up the space such that $[\tilde \omega_i]'_{\alpha,B_{(\epsilon_i^{-1})}(o), \Sigma_{3C}^2 \times \Sigma_{\loc}^2} = 1$ for the blow-up sequence $(\tilde\omega_i)$ with $\epsilon_i\to 0$. We select points $x_i\in \CC \cup \{o\}$ and radii $\rho_i>0$ to maximize the seminorm for each $i$. Depending on whether $\rho_i\to 0, \rho_i \to \infty$, or $\rho_i$ remains bounded away from these values, we apply different Liouville theorems (including Theorem \ref{LiouvilleTheoremM}) to obtain a contradiction. The proof is found in Section \ref{SectionCalpha}.\\

An application of Theorem \ref{mTheoremHolderBound} gives the asymptotic behavior of such Kähler metrics near the apex. The corollary also proves that the tangent cone at $o$ of any such $\omega$ is unique.

\begin{mcor}[Corollary \ref{CorollaryOfCalpha}]\label{mCorollaryOfCalpha}
	Let $(\CC,\omega_\CC)$ be a Calabi-Yau cone with cone metric $\omega_\CC$, and let $\omega$ be a Kähler metric on $B_3(o)\subset \CC$ such that
	\begin{equation*}
		\frac{1}{C}\omega_\CC \leq \omega\leq C\omega_\CC, \quad \Norm{\Scal(\omega)}{0,B_3(o)} \leq D,
	\end{equation*}
	for some constants $C,D>0$. Then, for any $\alpha = \alpha(\CC,\omega_\CC)>0$ small enough, there exists an automorphism $\Psi\in \Aut_\Scl(\CC)$ with
	\begin{equation*}
		|\Psi^*\omega- \omega_\CC|_{\omega_\CC} \leq C' r^{\alpha}
	\end{equation*}
	for $r\leq 1$. $\alpha = \alpha(\CC,\omega_\CC)$ and $C'= C'(\CC,C, D, \alpha,\omega_\CC, \UU)$ are independent of $\omega$.
	
	If $\Scal(\omega) = 0$, then
	\begin{equation*}
		|\nabla_{\omega_\CC}^k (\Psi^*\omega) - \omega_\CC |_{\omega_\CC} \leq C_k' r^{\alpha - k}, \qquad k \in \mathbb{N}_0,
	\end{equation*}
	for $r\leq 1$ and $C_k' = C_k'(\CC, C, D, k, \alpha, \omega_\CC, \UU)$.
\end{mcor}

\vspace{5mm}

\addtocontents{toc}{\protect\setcounter{tocdepth}{0}}
\subsection*{Acknowledgments}

I am highly grateful to my advisor Hans-Joachim Hein for for his continued support. The project was funded by the Deutsche Forschungsgemeinschaft (DFG, German Research Foundation) – Project-ID 427320536 - SFB 1442, and under Germany's Excellence Strategy EXC 2044 390685587, Mathematics Münster: Dynamics–Geometry–Structure.

\addtocontents{toc}{\protect\setcounter{tocdepth}{2}}

\section{A Liouville Theorem for Calabi-Yau Cones}\label{SectionLiouville}

This section proves a Liouville theorem for Calabi-Yau cones by demonstrating that any cscK metric uniformly equivalent to the cone metric is equal to the cone metric up to a biholomorphism commuting with scaling. First, Calabi-Yau cones are defined as follows:

\begin{definition}\label{CalabiYauConeDefinition}\
	\begin{itemize}
		\item Let $(\CC,J, g_{\CC})$ be a complex $m$-dimensional Kähler manifold with complex structure $J$, metric $g_{\CC}$, and associated $(1,1)$-form (or Kähler form) $\omega_{\CC}$. We will often not distinguish between a metric and its Kähler form. The pair $(\CC,g_{\CC})$ (also written $(\CC,\omega_{\CC})$) is a Kähler cone if there exists a diffeomorphism $\CC \simeq \R_+\times L$ with $\R_+ \coloneqq (0,\infty)$ and for some closed and connected $(2m-1)$-dimensional manifold $L$, called the link. Furthermore, on $\R_+\times L$, $g_{\CC}$ takes the form $g_{\CC} = dr^2 + r^2h$ with $r$ the coordinate on $\R_+$ and $h$ a metric on the link $L$. The completion of the cone $\CC$ by including the apex $o$ is denoted by $\overline{\CC} \coloneqq \CC \cup \{o\}$.
		\item A Kähler cone $(\CC,\omega_{\CC})$ is a Calabi-Yau cone if $\CC$ is simply-connected and the associated Kähler metric $\omega_{\CC}$ is Ricci-flat.
		\item We denote by $\Aut_\Scl(\CC)$ the holomorphic automorphisms of $\CC$ commuting with scaling. Equivalently, $\Aut_\Scl(\CC)$ are the holomorphic automorphisms preserving the scaling vector field $r\del_r$ (and thereby also the Reeb field $\xi = Jr\del_r$). The action extends to $\overline{\CC}$ by fixing $o\in \overline{\CC}$.
		\item All constants in the estimates below may change from line to line.
	\end{itemize}
\end{definition}

\begin{notation}
	We always assume that the pair $(\CC,\omega_\CC)$ is Calabi-Yau.
\end{notation}

As $(\CC,\omega_{\CC})$ is a Kähler cone, the pair $(L,h)$ forms a Sasaki structure. Ricci-flatness implies that $(L,h)$ is Sasaki-Einstein with an Einstein constant of $2m-2$. The vector field $r\del_r$ is called the scaling vector field of $({\CC},\omega_{\CC})$ and the cone metric $\omega_{\CC}$ satisfies $\LL_{r\partial_r} \omega_{\CC} = 2\omega_{\CC}$. The Reeb field $\xi \coloneqq J(r\del_r)$ is the image of the scaling vector field $r\del_r$ under the complex structure $J$. Furthermore, the Reeb field $\xi$ restricts to $L$ and generates a torus action on this space. Using this setup, we prove the following theorem:

\begin{theorem}[Liouville Theorem for Calabi-Yau cones]\label{LiouvilleTheorem}
	Let $(\CC,\omega_\CC)$ be a Calabi-Yau cone, and take a cscK metric $\omega$ on $\CC$ such that there exists some constant $C>0$ with
	\begin{equation}\label{LiouvilleConditions}
		\frac{1}{C} \omega_\CC \leq \omega \leq C \omega_\CC.
	\end{equation}
	Then $\omega =  \Psi^*\omega_\CC$ for some cone automorphism $\Psi\in \Aut_\Scl(\CC)$.
\end{theorem}

\begin{notation}\
	\begin{itemize}
		\item Throughout the paper, all balls are measured with respect to $\omega_{\CC}$, i.e. for $p\in {\CC}$ and $r>0$, $B_r(p) = \{ x\in {\CC} \mid \dist_{\omega_{\CC}}(p,x)<r \}$ is the open ball with center $p$ and radius $r$ with respect to $\omega_{\CC}$.
		
		We extend the balls to measure distances from the apex. Let $o\in \overline{\CC}$ be the apex, and take $r>0$. $\dist_{\omega_{\CC}}$ extends from ${\CC}$ to a unique distance metric on $\overline{\CC}$, also denoted by $\dist_{\omega_{\CC}}$. Then
		\begin{equation*}
			B_r(o) \coloneqq \{x\in {\CC} \mid \dist_{\omega_{\CC}}(o,x) < r \}.
		\end{equation*}
		Notice that $o\notin B_r(o)$.
		\item The closure $\overline{B_r(p)}\subset \CCbar$ is with respect to $\dist_{\omega_\CC}$ extended to $\CCbar$. Hence, if $\dist_{\omega_\CC}(0,B_r(p)) = 0$, then $o\in \overline{B_r(p)}$.
		\item All norms considered in this paper are with respect to $\omega_{\CC}$ unless otherwise specified, i.e. $|\cdot| = |\cdot|_{\omega_{\CC}}$.
		\item When convenient, we identify any point $p = (r,x)\in \R_+ \times L$ under the identification $\CC \cong \R_+ \times L$.
	\end{itemize}
\end{notation}

\subsection{From Scalar Curvature to the Complex Monge-Ampère Equation}

This section transforms the cscK condition in the Liouville Theorem into a complex Monge-Ampère equation for $\omega$. In order to accomplish this, we need to review the Hölder norms on manifolds:

\begin{definition}[{\cite[Definition 3.1]{heinHigherorder2020}}]\label{DefHolderUsual}
	Let $(X,g)$ be a Riemannian manifold and $E\to X$ a vector bundle on $X$ with bundle metric $h$ and $h$-preserving connection $\nabla$. If $x,y\in X$ and if there is a unique minimal $g$-geodesic $\gamma$ joining $y$ to $x$, let $\P_{xy}$ denote $\nabla$-parallel transport on $E$ along $\gamma$. $\P_{xy}$ is undefined if there is no unique minimal geodesic. If $B_\lambda(p)$ is the $g$-geodesic ball of radius $\lambda$ at $p$, define:
	\begin{equation*}
		[f]_{\alpha, B_\lambda(p)} \coloneqq \sup \left\{\frac{|f(x) - \P_{xy}(f(y))|_{h(x)}}{d(x,y)^{\alpha}} \mid x\neq y, \P_{xy} \text{ is defined} \right\},
	\end{equation*}
	for all sections $f \in C^{0,\alpha}_\loc(B_{2\lambda}(x),E)$. Here we use the fact that any minimal geodesic in $B_\lambda(x)$ is contained in $B_{2\lambda}(x)$. If $\Norm{f}{0,B_\lambda(p)}$ is the usual $C^0$-norm with respect to $g$, the full $C^{0,\alpha}$-norm is:
	\begin{equation*}
		\Norm{f}{0,\alpha,B_\lambda(p)} \coloneqq \Norm{f}{0,B_\lambda(p)} + [f]_{\alpha, B_\lambda(p)}.
	\end{equation*}
	The $C^{k,\alpha}$-norm is:
	\begin{equation*}
		\Norm{f}{k,\alpha,B_\lambda(p)} \coloneqq \sum_{l=0}^k \Norm{\nabla^l f}{0,B_\lambda(p)} + [\nabla^k f]_{\alpha, B_\lambda(p)}.
	\end{equation*}
\end{definition}

We will only consider $E = \CC \times \C$ or $E = \bigwedge^l(\CC)$, where $h$ is induced from the metric $g$.

\begin{lemma}\label{LemmaGeodesics}
	For any two points $x,y\in\CC$, there exists a length-minimizing geodesic $\gamma$ with $x$ and $y$ as endpoints such that $\length_{\omega_\CC}(\gamma) = \dist_{\omega_\CC}(x,y)$. If $\CC = \C^m\setminus \{0\}$ and $\omega_\CC = \omega_{\eucl}$, the theorem is only true after adding the smooth point $0\in \C^m$.
\end{lemma}

\begin{proof}
	As $\Ric(\omega_\CC) = 0$, the link $L$ has Ricci curvature $\Ric(\omega_\CC|_L) = 2m-2$ as it is Sasaki-Einstein, implying that
	\begin{equation*}
		\diam_{\omega_\CC}(L) \leq \pi.
	\end{equation*}
	By Bonnet-Myers, if $\diam_{\omega_\CC}(L) = \pi$, then $L \cong S^{2m-1}$ and $(\CCbar,\omega_\CC) \cong (\C^m, \omega_{\eucl})$ by a rigidity theorem of Cheng \cite[Theorem 3.1]{chengEigenvalue1975}. In this case, the lemma is trivial after adding the apex $0\in \C^m$. For $\diam_{\omega_\CC}(L)<\pi$, the proof is given in \cite[Theorem 3.6.17]{buragoCourse2001}.
\end{proof}

\begin{remark}
	On any Ricci-flat Kähler cone $(\CC,\omega_\CC)$, any two points are connected by a minimizing geodesic (Lemma \ref{LemmaGeodesics}). For any points $p = (r,x), q = (s,y)\in \CC \cong \R_+ \times L$, the minimizing geodesic has the form
	\begin{equation*}
		\tilde \gamma(t) = (r(t), \gamma(t)),
	\end{equation*}
	where $\gamma(t)$ is a minimizing geodesic in $L$ between $x$ and $y$. This is unique away from the cut locus of $x$, hence any point $p\in \CC$ is connected to any $q\in \CC$ except for a set of Hausdorff dimension at most $2m-1$.
\end{remark}

The first step is to prove the following regularity theorem for the complex Monge-Ampère equation.

\begin{prop}\label{PropRegCMA}
	Let $\omega$ be a Kähler metric on a ball $B_6(0)\subset \C^m$. Assume that there exists constants $C,D>0$ such that
	\begin{equation*}
		\frac{1}{C}\omega_{\eucl} \leq \omega \leq C \omega_\eucl, \quad \Norm{\Scal(\omega)}{L^p(B_6(0))} \leq D,
	\end{equation*}
	for $p>2m$. Then
	\begin{equation*}
		\Norm{\omega}{1,\alpha, B_1(0)} \leq C_1
	\end{equation*}
	for $\alpha = 1-\frac{2m}{p}$ and constants $C_1 = C_1(m,C,D,p)$.
\end{prop}

\begin{proof}
	If $\omega^m = e^f \omega_{\eucl}^m$, the scalar curvature condition implies that
	\begin{equation}\label{Laplacianf}
		\Norm{\Delta_{\omega} f}{L^p(B_6(0))} \leq D.
	\end{equation}
	According to \cite[Theorem 4.18]{hanElliptic1997}, there exists a constant $\alpha = \alpha (m,p,C)$ such that
	\begin{equation*}
		\Norm{f}{0,\alpha,B_5(0)} \leq C_1
	\end{equation*}
	for $0<\alpha<1$. The $C^{0,\alpha}$-estimate for the complex Monge-Ampère equation \cite[Theorem 1.1]{chenalpha2015} (see also \cite{caffarelliInterior1989} and \cite{safonovClassical1989}) shows that
	\begin{equation}\label{omegaHolderEstimate}
		\Norm{\omega}{0,\alpha, B_4(0)} \leq C_2.
	\end{equation}
	Due to the uniform continuity of $\omega$ resulting from \eqref{omegaHolderEstimate}, \cite[Theorem 9.11]{gilbarg_Elliptic_2001} and \eqref{Laplacianf} imply that
	\begin{equation*}
		\Norm{f}{W^{2,p}(B_3(0))} \leq C_3.
	\end{equation*}
	Morrey's inequality \cite[Theorem 2.30]{aubinNonlinear1982}. shows that $\Norm{f}{1,\alpha, B_3(0)} \leq C_4$ for $\alpha = 1-\frac{2m}{p}>0$. By \cite[Lemma 2.1]{chenalpha2015}, there exists a potential $\phi \in C^{\infty}(B_2(0))$ such that $i\del\delbar \phi = \omega$ and $\Norm{\phi}{2,\alpha, B_2(0)} \leq C\Norm{\omega}{0,\alpha, B_3(0)}$. Differentiating the complex Monge-Ampère equation $\omega^m = e^f\omega_\eucl^m$ in direction $x_i$, $i=1,\dots,2m$, we obtain:
	\begin{equation*}
		m \omega^{m-1} \wedge i\del\delbar \frac{\partial \phi}{\partial x_i} = \frac{\partial e^f}{\partial x_i} \omega_{\eucl}^m,
	\end{equation*}
	so
	\begin{equation}\label{LaplacianCMABootstrap}
		m\Delta_{\omega} \frac{\partial \phi}{\partial x_i} = \frac{\partial e^f}{\partial x_i} \frac{\omega_\eucl^m}{\omega^m},
	\end{equation}
	with the right-hand side uniformly bounded in $C^{0,\alpha}(B_3(0))$ The Schauder estimates \cite[Theorem 9.19]{gilbarg_Elliptic_2001} imply that $\frac{\partial \phi}{\partial x_i} $ is bounded in $ C^{2,\alpha}(B_1(0))$ and hence $\omega$ is bounded in $ C^{1,\alpha}(B_1(0))$.
\end{proof}

\begin{corollary}\label{HolderBoundFunction1Alpha}
	Let $\omega$ be a Kähler metric on a ball $B_2(0)\subset \C^m$. Assume that there exists constants $C,D>0$ such that
	\begin{equation*}
		\frac{1}{C}\omega_{\eucl} \leq \omega \leq C \omega_\eucl, \quad \Norm{\Scal(\omega)}{0,B_2(0)} \leq D,
	\end{equation*}
	For any function $f\colon C^2(B_{\frac{3}{2}}(0))\to \R$ and $\alpha \in (0,1)$, there exists a constant $C' = C'(m,C,D,\alpha)>0$ such that
	\begin{equation*}
		\Norm{f}{1,\alpha,B_1(0)} \leq  C'(\Norm{\Delta_\omega f}{0,B_{\frac{3}{2}}(0)} + \Norm{f}{0,B_{\frac{3}{2}}(0)}).
	\end{equation*}
\end{corollary}

\begin{proof}
	After rescaling, Lemma \ref{PropRegCMA} shows that $\omega$ is uniformly bounded in $C^{1,\alpha}(B_{\frac{3}{2}}(0))$ for all $\alpha\in (0,1)$. Using the $W^{2,p}$-estimates for the Laplacian for $p\in (1,\infty)$, then
	\begin{equation*}
		\Norm{f}{W^{2,p}(B_{\frac{5}{4}}(0))} \leq C'(\Norm{\Delta_\omega f}{L^{p}(B_{\frac{3}{2}}(0))} + \Norm{f}{L^{p}(B_{\frac{3}{2}}(0)})) \leq C'(\Norm{\Delta_\omega f}{0,B_{\frac{3}{2}}(0)} + \Norm{f}{0,B_{\frac{3}{2}}(0)}).
	\end{equation*}
	Morrey's inequality \cite[Theorem 2.30]{aubinNonlinear1982} shows that
	\begin{equation*}
		\Norm{f}{1,1-\frac{2m}{p}, B_1(0)} \leq C'\Norm{f}{W^{2,p}(B_{\frac{5}{4}}(0))}
	\end{equation*}
	for $p>2m$, finishing the proof as $p\in (2m,\infty)$ can be chosen arbitrarily.
\end{proof}

\begin{prop}\label{LiouvilleUniformLaplace}
	Let $\omega$ be a Kähler metric on $\CC$ satisfying \eqref{LiouvilleConditions} and with scalar curvature $\Scal(\omega)$ bounded in $C^{0}(\CC\setminus B_R(o))$ for some radius $R>0$. Let $\phi\colon \CC \to \R$ be a bounded function such that
	\begin{equation*}
		\Delta_\omega \phi = A
	\end{equation*}
	for some constant $A\in \R$. Then $A=0$ and $\phi$ is a constant.\\
	If $A=0$ is given, the requirement $\Scal(\omega) $ bounded in $ C^{0}(\CC\setminus B_R(o))$ is not necessary.
\end{prop}

\begin{proof}
	
\end{proof}

\begin{proof}
	Assume that $A\neq 0$ and, without loss of generality, assume $A>0$ by changing the sign of $\phi$ if necessary. We will now show that $A\neq 0$ is not possible. Proposition \ref{PropRegCMA}, applied to balls $B_1(p)\subset \CC\setminus B_R(o)$, shows that $\omega$ is bounded in $ C^{1,\alpha}_\loc(\CC\setminus B_{2R}(o))$ for any $0<\alpha<1$. The Schauder estimates \cite[Theorem 9.19]{gilbarg_Elliptic_2001}, applied to $\phi$, imply that $\Norm{\phi}{C^{3,\alpha}_{\loc}(\CC\setminus B_{2R}(o))} \leq C$. The estimates are uniform away from the apex on balls of fixed radii. Therefore, by selecting points $p_i\in \CC$ with $\dist_{\omega_{\CC}}(o,p_i)\to \infty$ and increasing balls $B_{r_i}(p_i)$ with $r_i\to \infty$ not growing too fast compared to $\dist_{\omega_{\CC}}(o,p_i)$, we obtain a $C^{1,\alpha}_{\loc}$-sublimit $\omega|_{B_{r_i}(p_i)} \to \omega_\infty$ and a $C^{3,\alpha}_{\loc}$-sublimit $\phi|_{B_{r_i}(p_i)} \to \phi_\infty\in C_{\loc}^{3,\alpha}(\C^m)$ satisfying
	\begin{equation*}
		\frac{1}{C}\omega_{\eucl} \leq \omega_\infty \leq C \omega_{\eucl}, \quad \Norm{\omega_\infty}{0,\alpha,\C^m} \leq D, \quad \Delta_{\omega_\infty} \phi_\infty = A.
	\end{equation*}
	
	To obtain a contradiction with $A\neq 0$, consider
	\begin{equation}\label{TracePhi}
		\Delta_{\omega_\infty}\phi_\infty = i\tr_{\omega_\infty} (\del\delbar \phi_\infty) = i\tr (\omega_\infty^{-1}\del\delbar \phi_\infty) = i\tr (U\omega_\infty^{-1}U^*U (\del\delbar\phi_\infty) U^* ) = A,
	\end{equation}
	where $U = \SU(m)$ is a matrix such that $U \omega_\infty^{-1} U^*$ is diagonal at $0\in \C^m$. By the uniform bound \eqref{LiouvilleConditions}, the eigenvalues are bounded from below by $\frac{1}{C}$ and from above by $C$, i.e.
	\begin{equation*}
		iU\omega_\infty(0)^{-1} U ^* = (a_1, \dots, a_m ), \quad \frac{1}{C} \leq a_i \leq C \quad \text{for all } i=1,\dots,m.
	\end{equation*}
	Since $U$ defines a complex change of coordinates, $U (\del\delbar\phi_\infty) U^*$ represents the complex Hessian of $\phi_\infty$ in the new coordinates defined by $U$, denoted by $(z_1 = x_1+iy_1,\dots,z_m = x_m +iy_m)$. Let $\del_i = \del_{z_i}$ be the $\del$-operator in the direction of $z_i$, and let $\del_{\overline{i}}$ be the $\delbar$-operator in the direction of $\overline{z_i}$. Since $U\omega_\infty(0)^{-1} U^*$ is diagonal, \eqref{TracePhi} evaluates to
	\begin{equation*}
		\Delta_{\omega_\infty} \phi_\infty = i\tr (U\omega_\infty^{-1}U^*U (\del\delbar\phi_\infty) U^* ) = i\sum_{i=1}^m a_i \del_{i}\del_{\overline i}\phi_\infty = A>0.
	\end{equation*}
	For all $a_i$, $C\geq a_i \geq \frac{1}{C}$, so there is a direction $z_i$ such that $i\del_{i}\del_{\overline i}\phi_\infty = (\frac{\partial^2}{\partial x_i^2} +\frac{\partial^2}{\partial y_i^2}) \phi_\infty \geq \frac{A}{Cm}$. Therefore, either $\frac{\partial^2}{\partial x_i^2}\phi_\infty \geq \frac{A}{2Cm}$ or $\frac{\partial^2}{\partial y_i^2}\phi_\infty \geq \frac{A}{2Cm}$. Assuming that $\frac{\partial^2}{\partial x_i^2}\phi_\infty \geq \frac{A}{2Cm}$ and $\frac{\partial}{\partial x_i}\phi_\infty \geq 0$ (otherwise consider $-\frac{\partial}{\partial x_i}\phi_\infty$), the idea is to integrate along $x_i$ as long as $\frac{\partial^2}{\partial x_i^2}\phi_\infty \geq \frac{A}{4Cm}$. The uniform $C^{3,\alpha}_\loc$-estimate on $\phi_\infty$ guarantees that there always exists a uniform minimum length $l_{\min}$ for which this is true. Let $p_{\min}$ be the point reached after moving along $x_i$ for $l_{\min}$ units. Integrate along this path to find
	\begin{equation*}
		\phi_\infty(p_{\min}) - \phi_\infty(0) \geq \frac{A}{8Cm}l_{\min}^2.
	\end{equation*} 
	Starting from $p_{\min}$, the same procedure can be iterated as many times as needed to show that, for all $b>0$, there is a point $p\in \C^m$ such that
	\begin{equation*}
		\phi_\infty(p) - \phi_\infty(0) > b.
	\end{equation*}
	This contradicts $\phi_\infty$ being a bounded function and so $\Delta_{\omega} \phi = 0$. \\
	
	Assume now that $\Delta_{\omega} \phi = 0$ on $(\CC,\omega_\CC)$. As we will see in the proof of Theorem \ref{HeatKernelBounds} for the heat kernel, one can prove a Harnack inequality for the heat equation by adopting the standard proof for parabolic differential equations on domains $\Omega\subset \R^m$ to neighborhoods $B_r(o)\subset \CC$. The Harnack inequality only requires the uniform bound \eqref{LiouvilleConditions} on $\omega$. Extend $\phi$ constantly along $\R$ to the domain $\R\times \CC$ such that it solves the heat equation:
	\begin{equation*}
		(\del_t - \Delta_{\omega})\phi= 0,
	\end{equation*}
	since $\phi$ and $\omega$ are invariant in $t$ on the manifold $\R \times \CC$. Setting $t=0$ we obtain by \eqref{EqHolderHeat}:
	\begin{equation}\label{HarnackElliptic}
		|\phi(x)-\phi(y)|\leq E \left(\frac{\dist_{\omega_\CC}(x,y) }{r}\right)^\gamma\Norm{\phi}{L^\infty_{\omega_\CC}([-r^2,0]\times B_r(o))},
	\end{equation}
	for all $r<\infty$, some $\gamma,\delta>0$, all points $x,y\in B_{\delta r}(o)$, and some constant $E = E(\CC,\gamma,m, \delta)>0$. Taking $r\to\infty$, it follows that $\phi(x)-\phi(y) = 0$ for all $x,y\in \CC$ and so $\phi$ is a constant.
\end{proof}

\begin{corollary}\label{LemmaMetricRicciFlat}
	Any metric $\omega$ from Theorem \ref{LiouvilleTheorem} satisfies
	\begin{equation}\label{VolFormEquality}
		\omega^m = \omega_\CC^m
	\end{equation}
	after rescaling. In particular, $\omega$ is Ricci-flat.
\end{corollary}

\begin{proof}
	Take $\omega$ as in Theorem \ref{LiouvilleTheorem} and write $\omega^m = e^{F}\omega_\CC^m$. The uniform bound \eqref{LiouvilleConditions} implies that $F$ is bounded, and constant scalar curvature $\Scal(\omega)=A\in \R$ is equivalent to
	\begin{equation}\label{ScalarFlat}
		-\Delta_{\omega} F = A.
	\end{equation}
	By Proposition \ref{LiouvilleUniformLaplace}, the function $F$ is constant and $A=0$. Therefore, $\omega$ satisfies \eqref{VolFormEquality} after rescaling and is, in particular, Ricci-flat.
\end{proof}

By Corollary \ref{LemmaMetricRicciFlat} we assume for the remainder of the section that $\omega$ satisfies \eqref{VolFormEquality}.

\subsection{A Liouville Theorem for the Tangent Cones at $\infty$ and $o$}\label{SectionLiouvilleTangentCones}

\subsubsection{Cone Structure of the Asymptotic Limits}

Denote by $\Phi_\lambda \colon \CC\to \CC$ the rescaling $(r,x) \mapsto (\lambda r,x)\in \R_+\times L$. Take a Kähler metric $\omega$ satisfying the properties of Theorem \ref{LiouvilleTheorem}. Form the sequence of blow-down metrics
\begin{equation}\label{Rescaling}
	\omega_{\epsilon_i} = \Phi^*_{\epsilon_i}(\epsilon_i^{-2}\omega),
\end{equation}
for some sequence $\epsilon_i\to \infty$. Since the cone metric $\omega_\CC$ is invariant under rescaling, i.e. $\Phi^*_{\epsilon_i}(\epsilon_i^{-2}\omega_\CC) = \omega_\CC$, the conditions
\begin{equation*}
	\frac{1}{C} \omega_\CC \leq \omega_{\epsilon_i} \leq C \omega_\CC, \qquad \omega_{\epsilon_i}^m = \omega_\CC^m,
\end{equation*}
are preserved. The associated a priori estimate of the complex Monge-Ampère equation $\omega_{\epsilon_i}^m = \omega_\CC^m$ implies that there exists a $C^\infty_{\text{loc}}(\CC)$-sublimit $\omega_\infty$ satisfying \eqref{LiouvilleConditions} (regularity follows by bootstrapping \eqref{LaplacianCMABootstrap}). Taking the limit $\epsilon_i\to0$ gives an asymptotic limit $\omega_0$ at $o$, both $\omega_0$ and $\omega_\infty$ possibly depending on the sequence. Using heat kernel estimates, this section shows that $\omega_0$ and $\omega_\infty$ are (a priori, possibly different) Kähler cone metrics.\\

To proceed, we introduce the heat kernel $H_t$ and prove certain regularity statements for $H_t$.

\begin{definition}\label{DefHeatKernel}
	Let  $(M,g)$ be a Riemannian manifold with Laplacian $\Delta$. A nonnegative continuous function $H_t(x,y)\colon \R_+\times M \times M\to \R$ is a heat kernel of $(M,g)$ if the following holds: let $f\in C_0^\infty(M)$, then
	\begin{itemize}
		\item The function
		\begin{equation*}
			(P_t f)(x) \coloneqq \int_M H_t(x,y)f(y) \dvol_g(y)
		\end{equation*}
		solves the heat equation on $M$: 
		\begin{equation*}
			(\del_t -\Delta_g)(P_t f)(x) = 0.
		\end{equation*}
		\item As $t\to 0$, then
		\begin{equation*}
			(P_t f)(x) \to f(x)
		\end{equation*}
		in $C_0^\infty(M)$.
	\end{itemize}
\end{definition}

\begin{theorem}\label{HeatKernelProperties}
	\cite[Theorem 7.7, 7.13, and 7.20]{grigoryanHeat2009} For any Riemannian manifold $(M,g)$, there is a unique and smooth heat kernel $H_t(x,y)$ such that $H_t(\cdot,y)\in L^2(M)$ for every fixed $t>0,y\in M$, and satisfying:
	\begin{itemize}
		\item Symmetry: $H_t(x,y) = H_t(y,x)$.
		\item The semigroup identity:
		\begin{equation*}
			H_{t+s}(x,y) = \int_M H_t(x,z)H_s(z,y) \dvol_g(z).
		\end{equation*}
		\item For fixed $x\in M$, the function $H_t(x,y)$ satisfies the heat equation in $(t,y)$:
		\begin{equation}\label{HeatKernelHSolution}
			(\partial_t - \Delta_g) H_t(x,y) = 0.
		\end{equation}
		\item For any $x\in M$ and $t>0$:
		\begin{equation*}
			\int_M H_t(x,y) \dvol_g(y) \leq 1.
		\end{equation*}
	\end{itemize}
\end{theorem}

In order to prove regularity of the heat kernel of $\omega$, we need to control the curvature tensor.

\begin{lemma}\label{CurvatureBoundLemma}
	Let $\omega$ be a cscK metric on $\CC$ with a uniform bound $\frac{1}{C}\omega_\CC \leq \omega \leq C\omega_\CC$. For the curvature tensor $\Rm$ of $\omega$ and for any $m\in \N$, there exists a constant $C_m$ such that
	\begin{equation}\label{CurvatureBound}
		\sum_{k=0}^{m} R^{2+k} |\nabla^k_{\omega_\CC} \Rm|_{\omega_\CC}(p) \leq C_m,
	\end{equation}
	where $\nabla_{\omega_\CC}$ is the connection of $\omega_\CC$ and $R = \dist_{\omega_\CC}(o,p)$ is the distance to the apex.
\end{lemma}

\begin{proof}
	Pick a point $p\in \CC$ with $\dist_{\omega_\CC}(o,p) = R$. Pull back by $\Phi_{R}$ such that $\tilde p \coloneqq \Phi_{R}^*(p)$ with $\dist_{\omega_\CC}(o,\tilde p) = 1$. Assume that $B_\delta(\tilde p)$ is trivial for some small enough $\delta > 0$, and pick holomorphic normal coordinates $z_1 = x_1 + ix_2,\dots,z_m = x_{2m-1} + ix_{2m}$ with respect to $\omega_\CC$. Corollary \ref{LemmaMetricRicciFlat} implies that $\omega^m = \omega_{\CC}^m$ (potentially after rescaling). By \cite[Lemma 2.1]{chenalpha2015}, there exists a potential $\phi$ such that
	\begin{equation*}
		i\del\delbar \phi = \omega, \quad \Norm{\phi}{C^{2,\alpha}(B_{\delta}(p))} \leq C,
	\end{equation*}
	where $C$ depends only on $\omega_\CC$ and the uniform bound of $\omega$ due to the regularity of the complex Monge-Ampère equation \cite[Theorem 1.1]{chenalpha2015}. Proceeding as in Proposition \ref{PropRegCMA} by differentiating $\omega^m = \omega_\CC^m$ (specifically, as in \eqref{LaplacianCMABootstrap} but with a different background metric), we obtain the equation:
	\begin{equation*}
		m\Delta_{\omega} \frac{\partial \phi}{\partial x_i} = \tr_\omega \left(\frac{\partial}{\partial x_i} \omega_\CC^m \right).
	\end{equation*}
	The right-hand side is uniformly bounded in $C^{0,\alpha}(B_{\delta}(\tilde p))$. Bootstrapping this equation shows that $\Norm{\nabla^k_{\omega_\CC} \phi}{0,B_{\frac{\delta}{2}}(p)} \leq C_k$ for all $k\in \N$, with a constant depending only on $k$ and $C$ (cf. \cite[Theorem 6.17]{gilbarg_Elliptic_2001} for details on the estimates). Covering $\{r=1\}$ with finitely many such balls, it follows that $|\nabla^k_{\omega_\CC} \omega|(\tilde p) \leq C_k$ for any $\tilde p$ with $\dist_{\omega_\CC}(o,\tilde p) = 1$. Expressing $\Rm$ as a combination of second derivatives of $\omega$ and scaling back shows that $|\nabla^k_{\omega_\CC} \omega|(p) \leq C_k R^{-k}$.
\end{proof}

\begin{prop}\label{HeatKernelExistence}
	Let $H_t(x,y)$ be the heat kernel of $(\CC,\omega)$. Then, for any fixed $t > 0$ and $x \in \CC$, $\Norm{H_t(x,\cdot)}{L^\infty(\CC)}\leq Ct^{-\frac{1}{m}}$ with $C$ independent of $x \in \CC$. Furthermore, $|\nabla_\omega^k H_t(x,\cdot)|_\omega \leq C_k t^{-\frac{1}{m}} r^{-k}$.
\end{prop}

\begin{proof}
	First assume that $\dim_{\C} \CC \geq 3$. For the $L^\infty$-bound, fix $\alpha > 0$ and consider the truncated cone $\CC_\alpha \coloneqq (\alpha,\alpha^{-1})\times L$ with associated heat kernel $H^\alpha_t$. By \cite[Section 8.2.3]{grigoryanEstimates1999}, the heat kernel $H^\alpha_t$ extends to the closure $\CCbar_\alpha = [\alpha,\alpha^{-1}]\times L$ satisfying Dirichlet boundary conditions. Fix $x\in \CC_\alpha$ and set $E_t = \int_{\CC_\alpha} H^\alpha_t(x,y)^2 \dvol_\omega(y)$ for $t>0$ and $x\in \CC_\alpha$. Consider:
	\begin{equation*}
		\del_t E_t = \int_{\CC_\alpha} 2H^\alpha_t(x,y)\del_t H^\alpha_t(x,y) \dvol_{\omega}(y) = -2\int_{\CC_\alpha} |\nabla_y H^\alpha_t(x,y)|^2 \dvol_{\omega}(y)\leq -CE_t^{1+\frac{1}{m}},
	\end{equation*}
	by the Nash inequality (see Proposition \ref{NashInequality}). Recall that $m$ denotes the complex dimension of $\CC_\alpha$. ODE comparison implies that $E_t \leq Ct^{-\frac{1}{m}}$. By the semigroup property of $H^\alpha_t$:
	\begin{equation*}
		H_{t}^\alpha(x,y) = \int_{\CC_\alpha} H^\alpha_{\frac{t}{2}}(x,z) H^\alpha_{\frac{t}{2}}(z,y) \dvol_\omega(z) \leq \Norm{H^\alpha_{\frac{t}{2}}(x,\cdot)}{L^2_\omega(\CC)}\Norm{H^\alpha_{\frac{t}{2}}(\cdot,y)}{L^2_\omega(\CC)} \leq Ct^{-\frac{1}{m}}.
	\end{equation*}
	The constant $C$ only depends on the Nash inequality and is independent of $x$ and $y$. Thus, ${H_t}^\alpha \in L^\infty_{\omega}(\CC_\alpha \times \CC_\alpha)$ by the uniform bound on $\omega$. As $H^\alpha_t \to H_t$ in $C_\loc^{\infty}(\R_+ \times \CC)$ \cite[Exercise 7.40]{grigoryanHeat2009} and as the Nash constant is uniformly bounded for all $\alpha > 0$ and $t^{-\frac{1}{m}}$, we obtain $\Norm{H_t}{ L^\infty(\CC \times \CC)} \leq Ct^{\frac{-1}{m}}$ for all $t>0$. \\
	
	To show that $|\nabla^k_\omega H_t(x,\cdot)|_\omega \leq C t^{-\frac{1}{m}} r^{-k}$, fix $t > 0$ and $x \in \CC$. We combine Lemma \ref{CurvatureBoundLemma}, Proposition \ref{GradientEstimate}, and the $L^\infty$-estimate above as follows: identify $\CC \cong \R_+ \times L$, and let $\chi$ be a spacetime bump function on $\R \times \CC$ with the following properties: $\chi$ has support in $[\frac{1}{2},\frac{3}{2}]\times([\frac{1}{2},\frac{3}{2}]\times L)$, $\chi \equiv 1$ on $[\frac{3}{4},\frac{5}{4}]\times ([\frac{3}{4},\frac{5}{4}]\times L)$, and satisfies
	\begin{equation*}
		\chi^{-\frac{1}{2}} |\nabla \chi| + (\Delta \chi)^- + |\del_t \chi| \leq C
	\end{equation*}
	for some constant $C > 0$. Furthermore, define $\beta\colon [\frac{1}{2},\frac{3}{2}]\times L \to \R$ by $\beta(r,x) = r^2$ with $(r,x)\in \R_+\times L \cong \CC$. For $R = 1$, Proposition \ref{GradientEstimate} shows that $\Norm{\nabla^k_{\omega_\CC} H_t(x,\cdot)}{L_{\omega_\CC}^\infty([\frac{3}{4},\frac{5}{4}])} \leq C_k$. Pulling back by $\Phi_R$, the same argument shows that $|\nabla_{\omega_\CC}^k H_t(x,\cdot)| \leq C_k R^{-k}$ on $[\frac{3R}{4},\frac{5R}{4}]\times L$.
\end{proof}

It will be helpful to obtain local upper and lower Gaussian bounds on $H_t$. These were proved by Saloff-Coste in \cite[Theorem 6.1]{saloff-costeUniformly1992} for complete manifolds. The following proposition states and proves a Gaussian upper bound on $\CC$ by modifying his arguments, which are inspired by Moser's \cite{moserharnack1964, moserpointwise1971} proof of the parabolic Harnack inequality on $\R^n$. Corollary \ref{GaussianBoundsDerivative} extends the result to $\CCbar$.

\begin{prop}\label{HeatKernelBounds}
	For all $t > 0$ and $y, y' \in \CC$, there is the following upper bound for the heat kernel of $(\CC,\omega_\CC)$:
	\begin{equation}\label{GaussianBounds}
		H_{t}(y,y') \leq \frac{C}{t^m} \exp\left(- C_1 \frac{\dist_{\omega_\CC}^2(y,y')}{t}\right).
	\end{equation}
	The constants $C,C_1$ only depend on the uniform constant in \eqref{LiouvilleConditions}.
\end{prop}

\begin{proof}
	The proof in \cite{saloff-costeUniformly1992} starts with a complete Riemannian manifold $(M,h)$. On this space, Saloff-Coste studies the heat equation $(\del_t - \tilde\Delta)u = 0$ for a Riemannian metric $\tilde h$ (with Laplacian $\tilde\Delta$) uniformly equivalent to $h$. No curvature assumption is placed on $\tilde h$, only a lower bound on the Ricci curvature of $h$. In our setup, the manifold $(M,h)$ is the cone $(\CC,\omega_\CC)$, and $\omega$ is the uniformly equivalent metric $\tilde h$.
	
	The proof in \cite{saloff-costeUniformly1992} first shows the existence of Sobolev and weighted Poincaré inequalities \cite[Theorem 3.1 + 3.2]{saloff-costeUniformly1992} on balls $B_{r}(p)$. The proof of the weighted Poincaré inequality begins with the regular Poincaré inequality on smaller balls $B_s(q)\subset B_r(p)$. The Poincaré constant is uniformly bounded by \cite[p. 214]{busernote1982}. By choosing a Whitney covering of small balls \cite[Lemma 5.6]{jerisonPoincare1986} and applying the Poincaré inequality on each covering ball, the weighted Poincaré inequality is obtained. For further details, see \cite[Appendix]{saloff-costeOperateurs1991}. The suitable Sobolev inequality is proven in \cite[Theorem 10.3]{saloff-costeUniformly1992}.\\
	
	All of these arguments go through in our setting if $o \notin \overline{B_r(p)}$. We claim that they still hold if $o\in \overline{B_r(p)}$. For this, we first need the Sobolev inequality (see Proposition \ref{SobolevInequality} for the proof), which is as follows:
	\begin{equation*}
		\left(\int_{B_r(p)} |f|^{2q} \dvol_{\omega_\CC}\right)^{1/2q} \leq C \int_{B_r(p)} |\nabla_{\omega_\CC} f|^2 \dvol_{\omega_\CC},
	\end{equation*}
	for $f\in C_0^\infty(B_r(p))$ and such that the finite ball $B_r(p)$ may have closure in $\CCbar$ containing the apex $o\in \CCbar$, and $q = m/(m-1)$ (recall that $m$ is the complex dimension). This is a stronger version than \cite[Theorem 3.1]{saloff-costeUniformly1992}, as the integral on the right-hand side only contains $|\nabla f|^2$.\\
	
	For the weighted Poincaré inequality \cite[Theorem 3.2]{saloff-costeUniformly1992}, we know that the regular Poincaré inequality has a uniformly bounded constant on small balls around the link $\{r=1\}$. As the Poincaré inequality is scale-invariant, we may scale the distance and radius of the balls uniformly to get arbitrarily close to the apex with associated smaller radii. As the Whitney covering in \cite{saloff-costeOperateurs1991} constructs balls whose radii are proportional to the distance to the boundary, and considering $o\in \CCbar$ to be part of the boundary for this purpose, this exactly gives the necessary uniform bound on the Poincaré constants of the Whitney covering. The proof, therefore, goes through as in \cite{saloff-costeOperateurs1991}.\\
	
	The Harnack inequality \cite[Theorem 5.3]{saloff-costeUniformly1992} for positive solutions of the heat equation is proved using the methods in Moser \cite{moserpointwise1971}. For a test function $\phi\in C^\infty_0(B_r(p))$, Moser considers integrals of the form:
	\begin{equation}\label{MoserIntegral}
		\int_{B_{r}(p)} \gen{\nabla_{\omega_\CC} \phi, \nabla_{\omega_\CC} u^p}_{\omega_\CC} \dvol_{\omega_\CC},
	\end{equation}
	where $p\geq 1$ and $u$ is a weak solution of the heat equation. For balls such that $o \notin \overline{B_r(p)} \subset \CCbar$, nothing needs to be changed. If the closure of $B_r(p)$ in $\CCbar$ contains the apex, we replace $\phi$, which may not be in $C^\infty_0(B_r(p))$, with $\chi_i \phi$, where $(\chi_i)$ is a family of smooth bump functions on $\CC$ with supports away from $B_{\frac{1}{i}}(o)$ and such that $\chi_i \to 1$ pointwise on $B_r(p)$ as $i\to\infty$. By basic scaling, we can fix $\chi_i$ such that $|\nabla \chi_i| \leq \frac{C}{i}$ on $B_{\frac{2}{i}}(o)$ for some $C>0$ and $\nabla \chi_i = 0$ on $\CC\setminus B_{\frac{2}{i}}(o)$. Then:
	\begin{align*}
		&\int_{B_{r}(p)} \gen{\nabla_{\omega_\CC} (\chi_i\phi), \nabla_{\omega_\CC} u^p}_{\omega_\CC} \dvol_{\omega_\CC}\\
		=& \int_{B_{r}(p)} pu^{p-1}\chi_i\gen{\nabla_{\omega_\CC} \phi, \nabla_{\omega_\CC} u} \dvol_{\omega_\CC} + \int_{B_{r}(p)} pu^{p-1}\phi \gen{\nabla_{\omega_\CC} \chi_i, \nabla_{\omega_\CC} u} \dvol_{\omega_\CC}.
	\end{align*}
	Assume that $u \in L^\infty_{\omega_\CC}(B_r(p))$ and that both $u$ and its weak gradient $\nabla_{\omega_\CC}u$ have finite $L^2(B_r(p))$-norm. Use Cauchy-Schwarz, dominated convergence, and the bound above to show that the second integral on the right-hand side vanishes in the limit $i\to \infty$. Therefore, all arguments in Moser can be carried out for bump functions with support away from the apex. Afterwards, we take the limit $i\to \infty$ to include the apex.\\
	
	Let $Q = [0,\tau]\times B_R(p)$ for some $p\in\CCbar$, and define $Q_- = (\tau_1^-, \tau_2^-) \times B_{R'}(p'), Q_+ = (\tau^+, \tau) \times B_{R'}(p')$ for $R'<R$ and $0<\tau_1^- < \tau_2^- < \tau^+ < \tau$. The obtained Harnack inequality (see \cite[Theorem 1]{moserharnack1964} and \cite[Theorem 5.3]{saloff-costeUniformly1992}) now reads:
	\begin{equation}\label{HarnackOriginal}
		\sup_{Q_-} u \leq C \inf_{Q_+} u,
	\end{equation}
	for any nonnegative solution $u$ of the heat equation on $Q$, where $C= C(R,R',\tau_1^- , \tau_2^- , \tau^+ , \tau)$ and the constant of the uniform bound in \eqref{LiouvilleConditions}. If $u$ is positive, a chaining argument (see \cite[Proof of Theorem 2]{moserharnack1964}) shows that:
	\begin{equation}\label{EqLogBound}
		\log\left(\frac{u(t',y')}{u(t,y)}  \right) \leq C \left(\frac{\dist_{\omega_\CC}(y,y')^2}{t-t'} + 1\right),
	\end{equation}
	where $C$ only depends on the constant of the uniform bound in \eqref{LiouvilleConditions}. Furthermore, for any positive solution $u$ (see \cite[Corollary 5.3]{saloff-costeUniformly1992}):
	\begin{equation}\label{EqHolderHeat}
		|u(t',y')-u(t,y)|\leq C \left(\frac{\max\{\dist_{\omega_\CC}(y,y'),\sqrt{|t-t'|}\}}{r}\right)^\gamma\Norm{u}{L^\infty_{\omega_\CC}([s- r^2,s]\times B_r(p))},
	\end{equation}
	for $t,t'\in [s-\delta r^2,s]$, $y,y'\in B_{\delta r}(p)$, any $s\in \R_+$, and where $C,\gamma$ only depend on the constant in \eqref{LiouvilleConditions} and $0<\delta<1$.\\
	
	To finally obtain the Gaussian upper bounds in \cite[Theorem 6.1 and 6.3]{saloff-costeUniformly1992}, the proof in \cite[Theorem 6.3]{saloff-costeUniformly1992} only necessitates that the heat kernel generates the heat flow. Saloff-Coste combines this with simple results in complex analysis and the Harnack inequality to show the upper bound. Our bounds simplify compared to the work of Saloff-Coste as $\Ric(\omega_\CC) = 0$.
\end{proof}

\begin{corollary}\label{GaussianBoundsDerivative}
	The heat kernel $H_t(x,y)$ on $\CC$ has a unique continuous extension to $\CCbar$ satisfying the Gaussian upper bounds:
	\begin{equation}\label{GaussianBoundsDerivativeEq}
		|\nabla_{\omega,y}^k H_t(y,y')| \leq \frac{C_k}{t^m} \dist_{\omega_\CC}(o,y)^{-k}\exp\left(- C_1 \frac{\dist_{\omega_\CC}^2(y,y')}{t}\right),
	\end{equation}
	for all $y'\in \CCbar$ and $y\in \CC$. Furthermore, $\Norm{H_t(x,\cdot)}{L_\omega^1(\CC)} = 1$ for all $x\in \CCbar$.
\end{corollary}

\begin{proof}
	By Proposition \ref{HeatKernelExistence}, we know that $\Norm{H_t(x,\cdot)}{L_{\omega_\CC}^\infty(\CC)}\leq C$, with $C$ independent of $x$. Using the Hölder continuity property \eqref{EqHolderHeat}, $H_t(x,y)$ has a unique continuous extension to $\CCbar$, and $H_t(x,\cdot)\to H_t(o,\cdot)$ as $x\to o$ in $C^{0,\gamma}_\loc(\CCbar)$ by \eqref{EqHolderHeat}. By parabolic regularity, $H_t(o,\cdot)$ satisfies the heat equation \eqref{HeatKernelHSolution} and $H_t(x,\cdot)\to H_t(o,\cdot)$ in $C^\infty_\loc(\CC)$. As $H_t(x,\cdot)$ satisfies the Gaussian upper bounds of Theorem \ref{GaussianBounds} uniformly, so does the extension $H_t(o,\cdot)$. Finally, dominated convergence shows that $\Norm{H_t(o,\cdot)}{L^1_\omega(\CC)}\leq 1$. Hence, from now on, we assume that $H_t$ is a function $H_t \colon \R_+ \times \CCbar \times \CCbar \to \R_+$.\\
	
	The estimates on the derivatives follows by the same argument as in Proposition \ref{HeatKernelExistence}.\\
	
	Finally, the equality $\Norm{H_t(x,\cdot)}{L^1_\omega(\CC)} = 1$ for all $x\in \CCbar$ is a direct consequence of the nonnegativity of $H_t$ and the Gaussian bounds in \eqref{GaussianBoundsDerivativeEq}:
	\begin{align*}
		\del_t \Norm{H_t(x,\cdot)}{L^1_\omega(\CC)} &= \del_t \int_\CC H_t(x,y)\dvol_{\omega}(y) = \int_\CC \del_t H_t(x,y)\dvol_{\omega}(y) \\
		&= \int_\CC \Delta_{\omega} H_t(x,y)\dvol_{\omega}(y)\\
		&= \lim_{\alpha\to 0} \int_{\partial([\alpha,\alpha^{-1}]\times L)} \gen{\nabla_\omega H_t(x,y),\nu_\omega}\dvol_\omega(y) = 0,
	\end{align*}
	so $\Norm{H_t(x,\cdot)}{L^1_\omega(\CC)}$ is invariant in time. Finally, the property $(P_t f) \to f$ in $C^\infty_\loc(\CC)$ for all bump functions $f\in C^\infty_0(\CC)$ as $t\to 0^+$ shows that $\Norm{P_tf}{L^\infty_\omega(\CC)} \to \Norm{f}{L^\infty_\omega(\CC)}$. Young's inequality $\Norm{P_tf}{L^\infty_\omega(\CC)} \leq \Norm{H_t(x,\cdot)}{L^1_\omega(\CC)}\Norm{f}{L^\infty_\omega(\CC)}$ hence concludes that $\Norm{H_t(x,\cdot)}{L^1_\omega(\CC)}\to 1$ as $t\to 0^+$. For $x = o$, take a sequence $x_i \to o$ and use dominated convergence and the Gaussian upper bounds of \eqref{GaussianBoundsDerivativeEq} to conclude that the $L_\omega^1$-norm does not escape to infinity, thus $\Norm{H_t(o,\cdot)}{L^1_\omega(\CC)} = 1$.
\end{proof}

\begin{corollary}\label{GaussianLowerBounds}
	The heat kernel $H_t$ satisfies the Gaussian lower bound:
	\begin{equation*}
		H_t(y,y') \geq \frac{C}{t^m} \exp\left(- C_1\frac{\dist_{\omega_\CC}^2(y,y')}{t}\right),
	\end{equation*}
	for all $y,y'\in \CCbar$.
\end{corollary}

\begin{proof}
	By picking a ball $B_R(o)$ large enough, the Harnack inequality \eqref{HarnackOriginal} and  $\Norm{H_t(p,\cdot)}{L_\omega^1(\CC)} = 1$ for all $p\in \CC$ show that $H_t(p,\cdot)>0$ everywhere and for any $p\in \CCbar$. Pick $y,y'\in \CCbar$, $0 < t' = \frac{t}{2}$, and set $u(t,\cdot) = H_t(y',\cdot)$ The log-inequality \eqref{EqLogBound} states that
	\begin{equation}\label{EqLogBoundTransformed}
		u(t,y) \geq C \exp\left(-C_1\frac{\dist_{\omega_\CC}(y,y')^2}{t-t'}\right) u(t',y').
	\end{equation}
	Since \eqref{EqLogBoundTransformed} is invariant under a parabolic rescaling, rescale $(t,x) \mapsto (\tilde t, \tilde x) = (\frac{t}{t'}, \frac{x}{\sqrt{t'}})$  such that $t' \mapsto 1$, and let $\tilde u(\tilde t, \tilde x) = u(t,x)$. Under this rescaling, the Harnack inequality \eqref{HarnackOriginal} further shows that
	\begin{align*}
		\tilde u(2,\tilde y) &\geq  C \exp\left(-2C_1\dist_{\omega_\CC}(\tilde y,\tilde y')^2 \right) \tilde u(1,\tilde y') \\
		&\geq C \exp\left(-2C_1\dist_{\omega_\CC}(\tilde y,\tilde y')^2\right) \sup_{x\in B_R(\tilde y')} \tilde u\left(\frac{1}{2},x\right).
	\end{align*}
	The constant in the Harnack inequality only depends on $R>0$. Pick $R$ such that
	\begin{equation*}
		\int_{B_R(\tilde y')} \tilde u\left(\frac{1}{2},x\right) \, \dvol_{\tilde \omega}(x) \geq \frac{1}{2(t')^m} = \frac{2^{m-1}}{t^{m}}.
	\end{equation*}
	Here, $\tilde \omega = (t')^{-1}\Phi^*_{\sqrt{t'}} \omega$ is uniformly equivalent to $\omega_\CC$. Such an $R$ exists and is independent of $\tilde y'$ by the Gaussian estimates of Corollary \ref{GaussianBoundsDerivative} and $\Norm{H_{\frac{t'}{2}}(p,\cdot)}{L_\omega^1(\CC)} = 1$. The factor $\frac{1}{(t')^m}$ arises from the parabolic rescaling. Hence:
	\begin{align*}
		\tilde u(2,\tilde y) &\geq C \exp\left(-2 C_1 \dist_{\omega_\CC}(\tilde y,\tilde y')^2\right) \sup_{x\in B_R(\tilde y')} \tilde u\left(\frac{1}{2},x \right)\\
		& \geq C \exp\left(-2 C_1 \dist_{\omega_\CC}(\tilde y,\tilde y')^2\right) \frac{1}{\vol_{\tilde \omega}(B_R(\tilde y'))} \int_{B_R(\tilde y')} \tilde u\left(\frac{1}{2}, x\right) \, \dvol_{\tilde \omega}(x)\\
		& \geq \frac{C}{t^m} \exp\left(-2C_1\dist_{\omega_\CC}(\tilde y,\tilde y')^2\right),
	\end{align*}
	where $C$ absorbs the factor $\vol_{\tilde \omega}(B_R(\tilde y'))^{-1}$, which is uniformly bounded. Scale back and the result follows.
\end{proof}

If $\dim_{\C} \CC= 2$, the heat kernel extends smoothly over the apex.

\begin{corollary}\label{CorollaryUniformW51}
	$H_t$ extends smoothly over the apex when $\dim_{\C} \CC = 2$.
\end{corollary}

\begin{proof}
	If $\dim_{\C} \CC = 2$, then $\CCbar \cong \C^2 / \Gamma$ for a finite group $\Gamma\subset \SU(m)$. As $\CC$ is simply-connected, then $\Gamma = \{ \Id \}$. If $\omega^m = e^F \omega_\CC^m$, then $\Delta_{\omega_\eucl} F = 0$ as $\omega$ is Ricci-flat. Therefore, the Harnack inequality for harmonic functions on $\C^m$, Proposition \ref{PropRegCMA}, and bootstrapping imply that $\omega$ extends smoothly over $0$. Parabolic regularity \cite[Theorem 2.12]{krylovLectures2008} shows that $H_t$ extends to a smooth function on $\C^2$.
\end{proof}

\begin{corollary}\label{CorollaryEstimatesHeatKernelLimit}
	Let $H_{\epsilon_i,t}$ be the heat kernel of $\omega_{\epsilon_i}$, and denote by $H_{\infty,t}$ a $C^\infty_\loc(\CC)$ sublimit of $H_{\epsilon_i,t}$ as $\epsilon_i\to\infty$. Then $H_{\infty,t}$ satisfies:
	\begin{itemize}
		\item Gaussian upper and lower bounds of Corollaries \ref{GaussianBoundsDerivative} and \ref{GaussianLowerBounds}.
		\item Smooth extension over $o$ if $\dim_{\C} \CC = 2$.
		\item $\Norm{H_{\infty,t}(x,\cdot)}{L^1_{\omega_\infty}(\CC)} = 1$.   
	\end{itemize}
\end{corollary}

\begin{proof}
	As $H_{\infty,t}$ is the $C^\infty_\loc(\CC)$-limit of $H_{\epsilon_i,t}$, the Gaussian upper and lower bounds in Proposition \ref{GaussianBounds} hold uniformly and so pass along the sequence. The extension to $o$ follows by the Hölder estimate \eqref{EqHolderHeat}. Then $\Norm{H_{\infty,t}}{L^1_{\omega_\infty}(\CC)} = 1$ as $\Norm{H_{\epsilon_i,t}}{L^1_{\omega_{\epsilon_i}}(\CC)} = 1$ by dominated convergence and the Gaussian upper bounds.
\end{proof}

\begin{theorem}\label{TangentConesConeStructure}
	Define $f_{\infty,t}\colon \CCbar\times \CCbar \times \R_+ \to \R$ via
	\begin{equation}\label{Eq:f}
		H_{\infty,t}(x,y) = (4\pi t)^{-m} e^{-f_{\infty,t}(x,y)}.
	\end{equation}
	Then, the asymptotic limit $\omega_\infty$ is a cone metric with link given by the level set $\{f_{\infty,t}(o,\cdot) = c\}$ for any fixed $t>0$ and some constant $c$ depending on $t$.
\end{theorem}

\begin{proof}
	Define $f_t$ for the heat kernel of $\omega$ as in \eqref{Eq:f} and
	\begin{equation*}
		u(t,y) \coloneqq H_t(o,y)
	\end{equation*}
	as the heat kernel of $\omega$ with base point $o$. Following Perelman \cite{perelmanentropy2002}, define the entropy-like quantity:
	\begin{equation}\label{HeatKernelW}
		\W(t) \coloneqq \int_\CC \left(t|\nabla_\omega f_t(o,y)|_\omega^2 +  f_t(o,y) - 2m\right)u(t,y) \dvol_\omega(y). 
	\end{equation} 
	$\W(t)$ exists due to the Gaussian upper and lower bounds of Corollaries \ref{GaussianBoundsDerivative} and \ref{GaussianLowerBounds} for $\dim_{\C} \CC \geq 3$. If $\dim_{\C} = 2$, $H_t$ extends smoothly over $o$ and the integral also exists due to the Gaussian upper bounds at infinity.\\
	
	Using the Bochner formula and $\Norm{u(t,\cdot)}{L^1_\omega(\CC)} = 1$, \cite[proof of Theorem 16.8]{chowRicci2007} shows that
	\begin{align*}
		\frac{d}{dt}\W(t) =& \int_\CC (u(t,\cdot) \Delta_\omega W(t) + 2\gen{\nabla_\omega u(t,\cdot), \nabla_\omega W(t)}_\omega + W \Delta_\omega u(t,\cdot) )\dvol_\omega\\
		 &- \int_\CC 2t\left|\nabla_\omega^2 f_t(o,\cdot) - \frac{1}{2t}g \right|^2_\omega\dvol_\omega,
	\end{align*}
	where $\Ric(g) = 0$, $g$ is the metric tensor of $\omega$, and $W(t) = t(\Delta_\omega f_t(o,y) - |\nabla_\omega f_t(o,y)|_\omega^2) + f_t(o,y) - 2m$. The proof uses nothing but $u(t,\cdot)$ being a positive solution of the heat equation and $\Norm{u(t,\cdot)}{L^1_\omega(\CC)} = 1$. Restricting the integral to $\CC_\alpha = [\alpha,\alpha^{-1}]\times L$ and using integration by parts cancels terms inside the integral, leaving boundary terms. These boundary terms scale as $R^{-3}$ at the apex if $\dim_{\C} \geq 3$, so combining this with the Gaussian bounds \eqref{GaussianBounds} for the boundary at infinity shows that the integral of the boundary terms vanishes as $\alpha \to 0$. For $\dim_{\C} \CC = 2$, the heat kernel extends smoothly over $o$ and the boundary integral therefore also vanishes. Thus:
	\begin{equation}\label{WTimeDerivative}
		\frac{d}{dt}\W(t) = - \int_\CC 2t\left|\nabla_\omega^2 f_t(o,\cdot) - \frac{1}{2t}g \right|^2_\omega\dvol_\omega.
	\end{equation}
	
		Repeat the above procedure for the rescaled metric $\omega_\epsilon$ with associated heat kernel $H_{\epsilon,t}(x,y) = \Phi_{\epsilon_i}^* H_{\epsilon_i^2t}(x,y)$, where $\Phi_{\epsilon_i}(x,y) = (\epsilon_ix,\epsilon_iy)$. Since $x=o$ is fixed, $\W(t)$ is invariant under this rescaling in the following sense: let $\W_{\epsilon}(t)$ be the quantity \eqref{HeatKernelW} for $\omega_\epsilon$, then
	\begin{equation}\label{RescalingD}
		\W_{\epsilon}(t) = \W(\epsilon^{2}t).
	\end{equation}
	The integrals $\W_{\epsilon}(t)$ are uniformly bounded independent of $t>0$ due to the uniform Gaussian bounds on all $H_{\epsilon_i,t}$. Letting $\epsilon_i\to \infty$, the sequence $\omega_{\epsilon_i}\to \omega_\infty$ converges in $C_\loc^\infty(\CC)$, and the heat kernels $H_{\epsilon_i,t} \to H_{\infty,t}$ converge in $C_\loc^\infty(\CC)$. Since the integrand of $\W_{\epsilon}(t)$ is uniformly bounded by the Gaussian bounds, the limit $\W_{\epsilon}(t) \to \W_\infty(t)$ exists in $C^0_\loc((0,\infty))$ by dominated convergence. The monotonicity \eqref{RescalingD} implies that $\W_\infty(t)$ is constant in $t$. Hence $\partial_t \W_\infty(t) = 0$, and \eqref{WTimeDerivative} shows that $\nabla_{\omega_\infty}^2f_{\infty,t}(o,\cdot) = \frac{\omega_\infty}{2t}$, i.e.,
	\begin{equation}\label{GradientFlowEq}
		\LL_{\nabla_{\omega_\infty} f_{\infty,t}(o,\cdot)} \omega_\infty = \frac{\omega_\infty}{4t},
	\end{equation}
	where $\nabla_{\omega_\infty}$ is the connection with respect to $\omega_\infty$. \\
	
	We now show that $\omega_\infty$ is isometric to a conical metric. For the proof on complete manifolds, see \cite[Theorem 1]{tashiroComplete1965}. To show that $\omega_\infty$ is a conical metric, first assume $\nabla_{\omega_\infty} f_{\infty,t}(o,\cdot) \neq 0$ everywhere on $\CC$. Then all level sets are smooth, and the gradient flow of $f_{\infty,t}(o,\cdot)$ defines a local isometry near any level set $\{f_{\infty,t}(o,\cdot) = c\}$. Since $\LL_{\nabla_{\omega_\infty} f_{\infty,t}(o,\cdot)} f_{\infty,t}(o,\cdot) = |\nabla_{\omega_\infty} f_{\infty,t}(o,\cdot)|^2 >0$, $f_{\infty,t}(o,\cdot)$ is strictly increasing along its flow. If the flow stopped, $\LL_{\nabla_{\omega_\infty} f_{\infty,t}(o,\cdot)} f_{\infty,t}(o,\cdot) = 0$, which leads to a contradiction. Thus, $\omega_\infty$ is a conical metric.
	
	If $\nabla_{\omega_\infty} f_t(o,z) = 0$ at some point $z\in \CC$, then \eqref{GradientFlowEq} and the nondegeneracy of $\omega_\infty$ show that $z$ is an isolated critical point. Assume that $B_\epsilon(z)$ has no other critical points. By the Morse lemma, the level sets of $f_{\infty}$ are diffeomorphic to $S^{2m-1}$, and the previous argument shows that $\omega_\infty|_{B_\epsilon(z)}$ is isometric to a cone metric. As the metric extends smoothly over $z$, $\omega_\infty|_{B_\epsilon(z)}$ is flat. As $\omega_\infty$ is Ricci-flat, \cite[Theorem 5.2]{deturckRegularity1981} shows that $|\Rm_\omega|^2$ is analytic, and as the curvature vanishes on $B_\epsilon(z)$, $|\Rm_\omega|^2$ vanishes everywhere. We conclude that $\omega_\infty$ is a cone metric on $\CC$. 
\end{proof}

\begin{corollary}\label{TangentConesConeStructureCorollary}
	By letting $\epsilon_i \to 0$, the same proof as in Theorem \ref{TangentConesConeStructure} shows that $\omega_0$ is isometric to a conical metric.
\end{corollary}

By Theorem \ref{TangentConesConeStructure} and Corollary \ref{TangentConesConeStructureCorollary}, $\omega_\infty$ and $\omega_0$ are both isometric to, a priori different, conical metrics on $\CC$.

\subsection{The Tangent Cones and $\omega_{\mathscr{C}}$ have equal Reeb fields}\label{SectionEqualReebFields}

This section shows that the tangent cones $\omega_0$, $\omega_\infty$, and the original cone metric $\omega_\CC$ have equal Reeb fields (resp. scaling vector fields). All of the arguments in this section are done for $\omega_\infty$, but work equally well for $\omega_0$.

\subsubsection{Commuting Reeb Fields}

Let $V$ be the scaling vector field of $\omega_\infty$, and $r\del_r$ the scaling vector field of $\omega_\CC$. The first step is to show that $[V,r\del_r]=0$. The proof follows arguments outlined in \cite{heinCalabiYau2017}, wherein a classification of holomorphic vector fields commuting with $r\partial_r$ is presented. Employing the fact that $V$ is holomorphic and has linear growth due to the uniform bound \eqref{LiouvilleConditions}, i.e. $\frac{r}{C} \leq |V|_{\omega_\CC} \leq C r$ for some $C>0$, we modify the proofs to accommodate our situation and thereby show that $V$ is homogeneous and is given by this classification. Define
\begin{equation*}
	\psi(\cdot) \coloneqq g_\CC(V,\cdot) = \omega_\CC(J V,\cdot)\in \Omega^1(\CC)
\end{equation*}
to be the dual of $V$ with respect to the cone metric $\omega_\CC$.

\begin{lemma}
	$V$ is holomorphic for the complex structure $J$ of $\CC$, and the one-form $\psi$ dual to $V$ is harmonic with respect to $\omega_\CC$.
\end{lemma}

\begin{proof}
	Since $\omega_{\epsilon_i}$ are all Kähler with respect to $J$ and $\omega_\infty$ is the $C^\infty_{\loc}$--limit of $(\omega_{\epsilon_i})$, then $\omega_\infty$ is Kähler for $J$ and $V$ is holomorphic for the complex structure. Recall the Bochner formulas for vector fields $X\in \X(\CC)$ and one-forms $\alpha\in \Omega^1(\CC)$, cf. \cite[eq. (4.80)]{ballmannLectures2006}:
	\begin{equation*}
		\begin{aligned}
			2\delbar^*\delbar X &= \nabla_{\omega_\CC}^*\nabla_{\omega_\CC} X - \Ric_{\omega_\CC}(X),\\
			(dd^* + d^*d)\alpha &= \nabla_{\omega_\CC}^*\nabla_{\omega_\CC} \alpha + \Ric_{\omega_\CC}(\alpha).
		\end{aligned}
	\end{equation*}
	Since $\delbar V = 0$ and $\Ric_{\omega_\CC} = 0$, we have $\nabla_{\omega_\CC}^*\nabla_{\omega_\CC} V= 0$. Finally, $ \nabla_{\omega_\CC}^*\nabla_{\omega_\CC} \psi =0$ as $\psi$ is the dual of $V$, so the second equation implies that $(dd^* + d^*d)\psi = 0$ and thus $\psi$ is harmonic with respect to $\omega_\CC$.
\end{proof}

To use the classification by Hein-Sun \cite[Theorem 2.14]{heinCalabiYau2017} for $V$, it is necessary to show that the following lemma applies to $\psi$. Lemma \ref{HajoLemmaB.1} was originally formulated by Cheeger-Tian in \cite{cheegercone1994}, but we use the following formulation:

\begin{lemma}[{\cite[Lemma B.1]{heinCalabiYau2017}}] \label{HajoLemmaB.1} Let $C(Y)$ be a Riemannian cone of dimension $n\geq3$ such that $\Ric C(Y) \geq 0$. Let $\alpha$ be a homogeneous 1-form on $C(Y)$ with growth rate in $[0, 1]$. Then $(dd^* + d^*d)\alpha = 0$ holds if and only if, up to linear combination, either $\alpha = d(r^\mu\kappa)$, where $\kappa$ is a $\lambda$-eigenfunction on Y for some $\lambda\in [n - 1, 2n]$ and $\mu$ is chosen so that $r^\mu\phi$ is a harmonic function on $C(Y)$, or $\alpha = r^2\eta$, where $\LL_{r\partial_r}\eta = 0$ and, at $\{r = 1\}$, $\eta^\sharp$ is a Killing field on Y with $\Ric_Y \eta^\sharp = (n -2)\eta^\sharp$, or $\alpha = rdr$.
\end{lemma}

As $V$ is the scaling vector field of $\omega_\infty$, it has linear growth with respect to this metric. Furthermore, the uniform boundedness in \eqref{LiouvilleConditions} implies that $|V|_{\omega_\CC}$ is comparable to $r$, and hence this also holds for the dual $\psi$. Thus, to apply Lemma \ref{HajoLemmaB.1} to $\psi$, it is necessary to prove the following:

\begin{lemma}\label{HajoLemmaB.1Modified}
	The conclusion of Lemma \ref{HajoLemmaB.1} holds if "homogeneous 1-form $\alpha$ with growth rate in $[0,1]$" is replaced by "1-form $\alpha$ with upper and lower bound $\frac{r}{C} \leq |\alpha|_{\omega_\CC} \leq Cr$".
\end{lemma}

\begin{proof}
	The proof \cite[Lemma B.1]{heinCalabiYau2017} proceeds by decomposing the 1-form $\alpha$ into a sum of eigenfunctions $f(x)r^p\log^q r$, where $f$ is a one-form depending on the link (but could contain a $dr$-term), $p\in \N_0$ and $q\in \{0,1\}$. After solving $(dd^* + d^*d)f(x)r^p\log^q r = 0$ for each term separately, the uniform bound $\frac{r}{c} \leq |\alpha|_{\omega_\CC} \leq Cr$ implies that all terms except the linear growth term with no $\log$-factor vanish and the remaining term is homogeneous by construction. The argument is very similar to the proof of Lemma \ref{PropGrowthRateToHomogeneousFunction}. The rest of the proof follows as for Lemma \ref{HajoLemmaB.1}.
\end{proof}

Now, the following theorem severely restricts the form of the scaling vector field $V$ of $\omega_\infty$.

\begin{theorem}\label{DecompositionHoloFields}
	Let $(\CC,\omega_\CC)$ be a Calabi-Yau cone and $\p$ be the span of $r\del_r$ and the gradient fields of all $\xi$-invariant 2-homogeneous harmonic functions (see Lemma \ref{HajoLemmaB.1}) of $\omega_\CC$. Then the space of all holomorphic vector fields commuting with $r\del_r$ is equal to $\p\oplus J\p$. Furthermore, the scaling vector field $V$ of $\omega_\infty$ satisfies $[r\del_r,V] = 0$. Therefore, $V \in \p \oplus J\p$ with all elements in $J\p$ Killing fields of $\omega_\CC$.
\end{theorem}

\begin{proof}
	\cite[Theorem 2.14]{heinCalabiYau2017} shows that the space of all holomorphic vector fields commuting with $r\del_r$ is equal to $\p\oplus J\p$. Next, as $V$ is homogeneous in $r$ by Lemma \ref{HajoLemmaB.1Modified}, then
	\begin{equation*}
		[r\partial_r, V] = \LL_{r\partial_r} V = \mu V.
	\end{equation*}
	For any $(p,q)$-tensor $T$, $\LL_{r\partial_r} T = \mu T$ implies that $|T| = \OO(r^{\mu + p-q})$. $V$ is a $(1,0)$-tensor with linear growth, i.e. $\mu + 1 = 1$, thus $\mu=0$ and $[r\partial_r, V]=0$, and we conclude that $V\in \p \oplus J\p$.
\end{proof}

\subsubsection{Reeb fields are Killing fields}

The next step is to prove that $\xi$, the Reeb field of $\omega_\CC$, acts by isometries on the tangent cone $\omega_\infty$, and vice versa for the Reeb field $JV$ of $\omega_\infty$ on $\omega_\CC$. Since $JV$ is holomorphic and commutes with $r\del_r$, it is enough to show that $JV$ has no $\p$-component in the decomposition $\p\oplus J\p$ from Theorem \ref{DecompositionHoloFields} (recall that $J\p$ consists of Killing fields for $\omega_\CC$). This is accomplished by embedding the cone into $\C^N$ and applying the Jordan-Chevalley decomposition to $JV$.

Ornea-Verbitsky first proved the following lemma in \cite[Theorem 1.2]{orneaEmbeddings2007}. However, we use the formulation of van Coevering \cite[Theorem 3.1]{vancoeveringExamples2011} and slightly strengthen it to include an embedding of the holomorphic automorphism $\Aut_\Scl(\CC)$, which commutes with scaling.

\begin{lemma}\label{ConeEmbedding}
	Let $(\CC,\omega_\CC)$ be a Calabi-Yau cone with a torus $\T^k$ acting by holomorphic isometries. Assume furthermore that the Reeb vector field $\xi \in \Lie(\T^k)$ lie in the Lie algebra of $\T^k$ and all elements in $\T^k$ commute with $\Aut_{\Scl}(\CC)$. Then, for some $N\in\N$, there exists a holomorphic embedding $\Phi \colon \CC \hookrightarrow \C^N$ equivariant with respect to some homomorphism $\varphi\colon \Aut_\Scl(\CC) \hookrightarrow \GL(N,\C)$. Furthermore, there is a choice of hermitian inner product on $\C^N$ such that $\Isom(\omega_\CC)\cap \Aut_\Scl(\CC) \hookrightarrow \U(N)$ and $\T^k\hookrightarrow \U(1)^N$.
\end{lemma}

\begin{proof}
	The proof in \cite[Theorem 3.1]{vancoeveringExamples2011} provides a weight-space decomposition of functions on $\CC$ as follows: let $\T^k$ be the given torus acting by holomorphic isometries on $(\CC,\omega_\CC)$. Let $f$ be a holomorphic function on some open set $U\subset \CC$ with closure in $\CCbar$ containing $o$ and invariant under the action of $\T^k$. Let $\mathbf{a} = (a_1,\dots,a_k)\in \Z^k$ and $\mathbf{t} = (t_1,\dots,t_k)\in \T^k$. Define
	\begin{equation}\label{WeightSpaceDecomp}
		f_{\mathbf{a}}(z) \coloneqq \frac{1}{(2\pi)^k} \int_{\T^k} \prod_{j=1}^{k} t_j^{-1-a_j} f(t\cdot z) dt.
	\end{equation}
	By a change of variable, one easily sees that
	\begin{equation*}
		f_{\mathbf{a}}(\mathbf{t}\cdot z) = \mathbf{t}^{\mathbf{a}} f_{\mathbf{a}}(z) \coloneqq t_1^{a_1}\cdots t_k^{a_k}f_{\mathbf{a}}(z).
	\end{equation*}
	Then there is the decomposition
	\begin{equation*}
		f(z) = \sum_{(a_1,\dots,a_k)\in \Z^k} f_{(a_1,\dots,a_k)}(z),
	\end{equation*}
	with the right-hand side converging in $C_\loc^\infty(\CC)$. We call each $f_{\mathbf{a}}$ an $\mathbf{a}$-eigenfunction of $\T^k$. The proof of \cite[Theorem 3.1]{vancoeveringExamples2011} shows that for a finite set of functions $f_1,\dots, f_l$ defining an embedding around a small neighborhood of $o\in \CCbar$, we can pick finitely many components $(f_j)_{\mathbf{a}_j}$, $j=1,\dots,d$, and for different $\mathbf{a}_j = (a_{1,j},\dots,a_{k,j})$. All functions in $(f_j)_{\mathbf{a}_j}$ are holomorphic, and as $r\del_r$ is in the complexification $\Lie(\T^k)\otimes \C$, all $(f_j)_{\mathbf{a}_j}$ are homogeneous and hence determine an embedding by
	\begin{equation*}
		\tilde \Phi \colon \CC \hookrightarrow \C^{d}, \quad
		p \mapsto \left((f_1)_{\mathbf{a}_1}(p), \dots, (f_d)_{\mathbf{a}_d}(p) \right).
	\end{equation*}
	Next, we extend the embedding to some larger $\C^N$ for $N \geq d$ to define the homomorphism $\varphi\colon \Aut_\Scl(\CC) \to \GL(N,\C)$ as follows: let $V$ be the span of all functions $f\colon \CC \to \C$ such that $f = f_{\mathbf{a}}$ is an $\mathbf{a}$-eigenfunction of $\T^k$ for a sequence $\mathbf{a}$ coming from one of the embedding coordinates of $\tilde \Phi$. There are only finitely many different embedding functions $f_{\mathbf{a}}$, and so finitely many $\mathbf{a}$ for $\tilde \Phi$. Each function $f_{\mathbf{a}}$ is $d$-homogeneous for some $d$ depending on $\mathbf{a}$, and as $f_{\mathbf{a}}$ is furthermore holomorphic and hence harmonic, then
	\begin{equation*}
		\Delta_{\omega_\CC} f_{\mathbf{a}}= \frac{1}{r^2}(d(d-1) + d(2m-1) + \Delta_L) f_{\mathbf{a}} = 0.
	\end{equation*}
	Thus, $f_{\mathbf{a}}|_L$ is an eigenfunction of $\Delta_L$ with eigenvalue $-d(d-1) - d(2m-1)$. As the space of such eigenfunctions is finite-dimensional for each $d$, $V$ is finite-dimensional.
	
	Next, if $f_{\mathbf{a}} \in V$ is an $\mathbf{a}$-eigenfunction, $\Psi^*f_{\mathbf{a}}$ for $\Psi\in \Aut_{\Scl}(\CC)$ is also an $\mathbf{a}$-eigenfunction as $\Aut_\Scl(\CC)$ commutes with $\T^k$. Thus, $\Aut_\Scl(\CC)$ defines a natural linear action on $V$ by pullback. Define an embedding into the dual space by evaluation:
	\begin{align*}
		\CC &\hookrightarrow V^*\\ p &\mapsto (V \ni f \mapsto f(p)).
	\end{align*}
	The action of $\Aut_\Scl(\CC)$ on $V$ extends to the dual space $V^*$, preserves $\CC\subset V^*$, and is compatible with the action on $\CC$. More specifically, let $\varphi\colon \Aut_\Scl(\CC) \to \GL(V^*)$ denote the induced homomorphism; then, the action on $p\in V^*$ is
	\begin{equation*}
		(\varphi(\Psi)p)(f) = p(\Psi^*f),
	\end{equation*}
	and if $p\in \CC \subset V^*$, then
	\begin{equation*}
		(\varphi(\Psi)p)(f) = p(\Psi^*f) = f(\Psi(p)).
	\end{equation*}
	\sloppy Consider $V^* = \C^N$ for some $N$ to obtain the embedding $\Phi \colon \CC \to \C^N$ and homomorphism $\varphi \colon \Aut_\Scl(\CC) \to \GL(N,\C)$. $\varphi$ is injective as $\Phi$ is injective.
	
	Finally, $\Isom(\omega_\CC)\cap \Aut_\Scl(\CC)$ is a compact subgroup of $\GL(N,\CC)$, so averaging the standard Euclidean metric $g_{\text{Eucl}}$ on $\C^N$ over this group makes it invariant under $\Isom(\omega_\CC)\cap \Aut_\Scl(\CC)$, hence $\Isom(\omega_\CC)\cap \Aut_\Scl(\CC)\to\U(N)$ for this choice of metric. As the coordinates $(f_j)_{\mathbf{a}_j}$ are eigenvectors of the $\T^k$-action, then $\T^k\hookrightarrow \U(1)^N$.
\end{proof}

\begin{prop}\label{ReebFieldKilling}
	The Reeb vector field $JV$ of $\omega_\infty$ is a Killing field for $\omega_\CC$, and the Reeb vector field $\xi$ of $\omega_\CC$ is a Killing field for $\omega_\infty$.
\end{prop}

\begin{proof}
	By Lemma \ref{ConeEmbedding}, there is an embedding $\CC \hookrightarrow \C^N$ such that $J\p\hookrightarrow \uu(N)$ as $J\p$ consists of holomorphic Killing fields whose flows preserve the scaling vector field $r\del_r$ and fix $o$. Because $JV\in \p\oplus J\p$, the pushforward of $JV$ to $\C^N$ lies in $\gl(N,\C)$. Thus, we regard $J\p \subset \uu(N)$, $\p \subset i\uu(N)$, and $JV\in \gl(N,\C)$. By the Jordan-Chevalley decomposition of $\gl(N,\C)$, we find $D$ and $K$ such that $D$ is semisimple, $K$ is nilpotent, $[D,K] = 0$, and $JV = D+K$. As $JV$ has compact orbits, the same is true for $D$ and $K$ separately as $D$ has exponential flow and $K$ has polynomial flow. This means the flows cannot combine to form compact orbits if both flows do not already have compact orbits. As $D$ is diagonalizable, all eigenvalues are strictly imaginary; otherwise, write
	\begin{equation*}
		ADA^{-1} = \diag (a_1,\dots, a_N),
	\end{equation*}
	and assume that $\re(a_N) >0$ ($<0$ is the same proof by considering $t\to-\infty)$. If $\textbf{f} = (f_1,\dots,f_N)^t$ are the coordinate functions of the embedding $\CC \to \C^N$, then
	\begin{equation*}
		A \exp(tD) \textbf{f} = \diag (e^{a_1t} \dots, e^{a_Nt}) A\textbf{f}.
	\end{equation*}
	Therefore, for some $b_1,\dots, b_N$ not all zero and determined by $A$, the last coordinate of $A \exp(tD)\textbf{f}$ is $e^{a_Nt}(b_1 f_1 + \cdots + b_N f_N)$. As $t\to\infty$, then $|e^{a_Nt}(b_1 f_1 + \cdots + b_N f_N)| \to\infty$. As all $f_1,\dots,f_N$ are linearly independent, then some $f_i \to 0$ with $b_i\neq 0$, contradicting $\exp(tD)$ having compact orbits. We conclude that $D\in \uu(N)$. Furthermore, the flow of $K$ is polynomial and cannot only have compact orbits unless $K=0$. It follows that $JV \in \uu(N)$ and hence $JV\in J\p$, i.e. $JV$ is a Killing field for $\omega_\CC$.
	
	Finally, in all of the above arguments, we can interchange the roles of $(\omega_\CC,\xi)$ and $(\omega_\infty,JV)$, i.e. consider $\omega_\infty$ as the reference metric and apply the decomposition in Theorem \ref{DecompositionHoloFields} to $\xi$. This shows the second part of the proposition.
\end{proof}

\subsubsection{Deformations of the Sasakian Structure}

Recall that the Sasakian manifold associated with $(\CC,\omega_\CC)$ is the compact manifold $L$ equipped with a tuple $(g_L,\eta,\xi,\Phi)$ satisfying the following conditions: $g_L$ is the cone metric $g_\CC$ on $\CC$ pulled back to $\{r=1\}\cong L$, and $\xi$ is the associated Reeb field parallel to $\{r=1\}$. $\eta$ is the contact form defined by $\eta(\cdot) = \frac{1}{r^2} g_L(\xi,\cdot)$, and $\Phi$ is an endomorphism of $TL$ given by $\Phi(X) = J(X-\eta(X)\xi)$. Since $\xi$ and $JV$ commute and are Reeb fields associated with some Kähler cone metrics, the combined flow generates a torus $\T$ that acts by holomorphic isometries on $\omega_\CC$ and $\omega_\infty$. It is important to note that the torus $\T$ generated by $\xi$ and $JV$ preserves $(g_L,\eta,\xi,\Phi)$. As the associated $(1,1)$-form $\omega_\CC$ is a symplectic form, the action by $\T$ is symplectic on $(\CC,\omega_\CC)$ and generates a moment map
\begin{equation}\label{MomentMap}
	\mu \colon \CC \to \ttt^*,
\end{equation}
where $\ttt^*$ is the dual of the Lie algebra $\ttt$ of $\T$. By \cite[p. 18]{vancoeveringStability2013} and in terms of $\omega_\CC$ and $p=(r,x)\in \R_+ \times L$, the map is given by $\mu(p)(X) = r^2 (\omega_\CC)_x(\xi,X)$, where $X\in \ttt$ is regarded as the induced vector field on $\CC$ by the action $\T\times \CC\to \CC$. According to \cite[Theorem 1]{demoraesMoment1997}, the image of $\CC$ under $\mu$ is a strongly convex rational polyhedral cone $\CCone^*\subset \ttt^*$. Denote the interior of the dual cone by $\CCone_0\subset \ttt$. It is worth noticing that $\xi \in \CCone_0$ since $\eta(\xi) = 1$, but we do not know if $JV \in \CCone_0$ a priori. The goal of this section is to prove this. \\

In doing so, the transverse Kähler deformations must be defined. As the name suggests, they are deformations along the transverse structure obtained by the quotient of the Reeb field $\xi$ on $L$ and thus leave $\xi$ invariant. They have the property of leaving the cone $\CC_0$ invariant and are given in the following lemma: 

\begin{lemma}[{\cite[Lemma 7]{vancoeveringStability2013}}]\label{TransverseDeformationLemma}
	The space of all Sasakian structures with Reeb field $\xi$ and transverse complex structure $ J$ is an affine space modeled on $(C_b^\infty (L)/\R) \times (C_b^\infty(L)/\R)$ ($C_b^\infty$ denotes basic functions). Given a Sasakian structure $(L,g_L,\xi,\eta,\Phi)$, a new structure $(L,\tilde g_L,\xi,\tilde \eta,\tilde \Phi)$ with the same Reeb vector field $\xi$ is obtained by
	\begin{align*}
		\tilde\eta &= \eta + d^c\phi + d\psi, \\
		\tilde \Phi &= \Phi -\xi\otimes \tilde\eta\circ \Phi,\\
		\tilde g_L &= \frac{1}{2}d\tilde\eta \circ(\Id \otimes \tilde \Phi) + \tilde\eta \otimes \tilde\eta,\\
		\tilde \omega_L &= \omega_L + \frac{1}{2} dd^c \phi,
	\end{align*}
	where $\phi$ and $\psi$ are two basic smooth functions.
\end{lemma}

\begin{prop}\label{AllReebFieldsinOneCone}
	Let $(\CC,\omega_\CC)$ and $(\CC, \omega_\infty)$ be two Kähler cone metrics with commuting Reeb fields $\xi$ and $JV$. Let $\T$ denote the torus generated by the flows of $\xi$ and $JV$ with Lie algebra $\ttt$, and let $\CCone_0$ be the interior of the dual cone for $(\CC,\omega_\CC)$. Then $JV \in \CCone_0$.
\end{prop}

\begin{proof}
	First, embed $\omega_\CC$ and $\omega_\infty$ into $\C^N$ via the embedding from Lemma \ref{ConeEmbedding} with the torus-action $\T$ coming from $\xi$ and $JV$. On $\C^N$, the scaling vector fields $r\del_r$ and $V$ are given by
	\begin{equation}\label{DiagonalReebFields}
		r\del_r = \sum_{i=1}^N \left(a_i z_i \frac{\del}{\del z_i} + a_i \overline{z_i} \frac{\del}{\del \overline{z_i}}\right), \quad V = \sum_{i=1}^N \left(b_i z_i \frac{\del}{\del z_i} + b_i \overline{z_i} \frac{\del}{\del \overline{z_i}}\right),
	\end{equation}
	for $a_i,b_i\in \R$. As the functions $(f_j)$ are eigenfunctions of the torus-action $\T$, they are also eigenfunctions of the complexified torus $\T_\C$, which includes the action of the scaling vector fields $r\del_r$ and $V$. As the flow of the scaling vector fields transports any point towards $o$ on $\CCbar$ as $t\to-\infty$, it follows that $a_i,b_i>0$ for all $i=1,\dots,N$. We may therefore construct Kähler metrics $\tilde \omega_0$ and $\tilde \omega_1$ on $\C^N$ with the same Reeb vector fields $\xi_0 = \xi$ and $\xi_1 = JV$ and invariant under $\U(1)^N$:
	\begin{align*}
		\tilde \omega_0 &\coloneqq \frac{i}{2}\del\delbar (|z_1|^{2a_1^{-1}} + \cdots+ |z_N|^{2a_N^{-1}}), \quad \xi_0 = \xi,\\
		\tilde \omega_1 &\coloneqq \frac{i}{2}\del\delbar (|z_1|^{2b_1^{-1}} + \cdots+ |z_N|^{2b_N^{-1}}), \quad \xi_1 = JV.
	\end{align*}
	Restrict $\tilde \omega_0$ and $\tilde \omega_1$ to $\CC \subset \C^N$. As $\omega_\CC$ and $\tilde \omega_0$ (resp. $\omega_\infty$ and $\tilde \omega_1$) have the same Reeb field and transverse complex structure, they are transverse Kähler deformations of one another, and so have the same dual cone $\CCone_0$ \cite[p. 18]{vancoeveringStability2013}. By \cite[eq. 45]{vancoeveringStability2013}, to show that $JV \in \CCone_0$, it is enough to show that
	\begin{equation*}
		\tilde \omega_0(J(r\del_r),V) = \tilde g_0(r\del_r,V)>0,
	\end{equation*}
	where $\tilde g_0$ is the Riemannian metric associated with $\tilde\omega_0$. But this is shown by direct evaluation:
	\begin{equation*}
		\tilde g_0(r\del_r,V) = \sum_{i=1}^N a_i^{-1}b_i |z_i|^{2a_i^{-1}}>0.
	\end{equation*}
	The same argument shows that $\xi \in \CCone_0$.
\end{proof}

\subsection{The Tangent Cones equal $ \omega_{\mathscr{C}}$}\label{SectionEqualTangentCones}

The previous section proved that the dual moment cone $\CCone_0$ of $ \omega_\CC$ contains $\xi$ and $JV$, both of which are associated with Sasaki-Einstein metrics. Therefore, the results in \cite{martelliSasaki2008} will show that $\omega_\CC$ and $\omega_\infty$ are equal up to pullback by an element in $\Aut_\Scl(\CC)$. To prove this, we first set up the necessary notation.

\begin{lemma}\label{SameCalabiYauForm}
	There exists a holomorphic $(m,0)$-form $\Omega_\CC$ such that the Reeb fields $\xi$ and $JV$ of $\omega_\CC$ and $\omega_\infty$, respectively, satisfy $\LL_\xi \Omega_\CC = \LL_{JV} \Omega_\CC = i m \Omega_\CC$.
\end{lemma}

\begin{proof}
	As both $\omega_\CC$ and $\omega_\infty$ are Ricci-flat Kähler cone metrics on the Calabi-Yau cone $\CC$, there exist $\Omega_\CC$ and $\Omega_\infty$ such that $\LL_{JV} \Omega_\infty = i m\Omega_\infty$ and $\LL_\xi \Omega_\CC = im \Omega_\CC$ \cite[p. 624]{martelliSasaki2008}. These have the property
	\begin{equation*}
		\omega_\CC^m = i^m (-1)^{m(m-1)/2} \Omega_\CC \wedge \overline\Omega_\CC, \quad \omega_\infty^m = i^m (-1)^{m(m-1)/2} \Omega_\infty \wedge \overline\Omega_\infty.
	\end{equation*}
	Using the fact that $\Omega_\CC$ and $\Omega_\infty$ are both holomorphic sections of the line bundle $\bigwedge^{(m,0)}(\CC)$, there is a holomorphic function $h\colon \CC \to \CC$ such that
	\begin{equation*}
		\Omega_\infty = h \Omega_\CC.
	\end{equation*}
	By the volume condition $\omega_\infty^m = \omega_\CC^m$, we have
	\begin{equation*}
		i^{-m} (-1)^{-m(m-1)/2}\omega_\CC^m = \Omega_\CC\wedge\bar\Omega_\CC = \Omega_\infty \wedge \bar\Omega_\infty = |h|^2 \Omega_\CC\wedge\bar\Omega_\CC.
	\end{equation*}
	Since $h$ is a mapping $\CC \to S^1$ and is holomorphic, $h$ is a constant. The lemma follows.
\end{proof}

\begin{definition}[{\cite[p. 624]{martelliSasaki2008}}]\label{ReebFieldsConditions}
	Given the cone metric $(\CC, \omega_\CC)$ with a torus $\T$ acting by holomorphic isometries, define the set
	\begin{equation*}
		\Sigma_\CC = \{\tilde\xi \in \CCone_0 \, | \, \LL_{\tilde\xi} \Omega_\CC = im\Omega_\CC \}.
	\end{equation*}
\end{definition}

\begin{prop}[{\cite[p. 632]{martelliSasaki2008}}]
	$\Sigma_\CC$ is convex.
\end{prop}

\begin{proof}
	Let $\tilde\xi,\tilde\xi'\in\Sigma_\CC$. Then there exists $Y\in \ttt$ such that $\tilde\xi' = \tilde\xi + Y$ and
	\begin{equation*}
		\LL_Y \Omega_\CC = 0.
	\end{equation*}
	The space of all $Y\in \ttt$ for which $\LL_Y \Omega_\CC = 0$ forms a vector space in $\ttt$. Thus, $\Sigma_\CC$ is an affine space intersected with the convex cone $\CCone_0$ and hence also convex.
\end{proof}

The following theorem is crucial in showing that $\xi = JV$.

\begin{theorem}[{\cite[p. 635 and p. 638]{martelliSasaki2008}}]\label{SparksUniqueness}
	Define the volume function $\vol\colon \Sigma_\CC \to \R$ via
	\begin{equation*}
		\vol(\tilde \xi) = \int_{\tilde r\leq 1}\frac{\tilde\omega^m}{m!}.
	\end{equation*}
	for any Kähler cone metric $\tilde \omega = \frac{i}{2}\del\delbar \tilde r^2$ with Reeb field $\tilde \xi$. This is well-defined and strictly convex, hence has a unique critical point. Furthermore, $\tilde \xi\in \Sigma_\CC$ is a critical point if and only if there is a Ricci-flat Kähler cone metric with Reeb vector field $\tilde\xi$.
\end{theorem}

\begin{corollary}\label{Liouvilleinfinity}
	The Reeb vector field $JV$ for the tangent cone at infinity $ \omega_\infty$ is equal to $\xi$. Therefore, there exists an automorphism $\Psi\in \Aut_\Scl(\CC)$ such that\ $\Psi^* \omega_\infty =  \omega_\CC$.
\end{corollary}

\begin{proof}
	Proposition \ref{AllReebFieldsinOneCone} and Lemma \ref{SameCalabiYauForm} show that $\xi,JV\in \Sigma_\CC$. Theorem \ref{SparksUniqueness} therefore proves equality as they are both associated with Ricci-flat Kähler cone metrics. As both $ \omega_\CC$ and $ \omega_\infty$ are Ricci-flat Kähler cones with the same scaling vector field and transversal holomorphic structure, the Bando-Mabuchi argument \cite[Theorem A]{bandoUniqueness1987}, suitably generalized by Nitta-Sekiya \cite[Theorem A]{nittaUniqueness2012}, proves the corollary.
\end{proof}

As all of the arguments above also apply to $\epsilon_i\to0$ in \eqref{Rescaling}, i.e. for the tangent cone $ \omega_0$ at $o$, we immediately show Corollary \ref{Liouvilleinfinity} for $ \omega_0$:

\begin{corollary}\label{LiouvilleZero}
	Corollary \ref{Liouvilleinfinity} also applies to $ \omega_0$, the tangent cone at o.
\end{corollary}

\subsection{Combining Tangent Cones to Liouville Theorem}\label{SectionCombiningTangentConesToLiouville}

The last section proved that the tangent cones at $o$ resp. $\infty$ of $\omega$ equal $\Psi_0^*\omega_\CC$ resp. $\Psi_\infty^* \omega_\CC$ for automorphisms $\Psi_0, \Psi_\infty\in \Aut_\Scl(\CC)$. Given this, an analysis of \eqref{HeatKernelW} yields the proof of the Liouville Theorem (Theorem \ref{LiouvilleTheorem}).

\begin{proof}[Proof of Theorem \ref{LiouvilleTheorem}]
	Corollaries \ref{Liouvilleinfinity} and \ref{LiouvilleZero} proved that the asymptotic limits $\omega_0$ and $\omega_\infty$ are conical metrics and that there exists an isomorphism $\tilde{\Psi} \in \Aut_\Scl(\CC)$ such that $\tilde{\Psi}^* \omega_\infty = \omega_0$. By the proof of Theorem \ref{TangentConesConeStructure}, $\W_{\epsilon}(t) = \W(\epsilon_i^2 t)$ is nonincreasing in time (recall that $\W_{\epsilon}(t)$ is \eqref{HeatKernelW} for $\omega_\epsilon$), with $\W_{0}(t) = \text{constant} = \lim_{t \to 0} \W(t)$ and $\W_{\infty}(t) = \text{constant} = \lim_{t \to \infty} \W(t)$. Set $\tilde{r} = \dist_{\omega_\infty}(o, \cdot)$ and choose $c$ in $f_{\infty,t} = \frac{\tilde{r}^2}{4t} +c$ by Theorem \ref{TangentConesConeStructure} such that
	\begin{equation}\label{FractionVolEuclidean}
		\int_\CC H_{\omega_\infty}(o, \cdot) \dvol_{\omega_\infty} = \frac{1}{(4\pi t)^{m}} \int_\CC e^{-f} \dvol_{\omega_\infty} = \frac{1}{(4\pi t)^{m}} \int_\CC e^{-\frac{\tilde{r}^2}{4t} - c} \dvol_{\omega_\infty} = 1.
	\end{equation}
	When $(\CC, \omega_\infty) \cong (\mathbb{R}^{2m} \setminus \{0\}, \omega_\eucl)$, we have $c = 0$. By writing $\dvol_{\omega_\infty} = \tilde{r}^{2m-1} d\tilde{r} \wedge \dvol_{\omega_\infty|_L}$ and comparing with the Euclidean integral, we see that 
	\[
	c = -\log\left(\frac{\vol_{\omega_\infty}(L)}{\vol_{\eucl}(S^{2m-1})}\right).
	\]
	The expression for $\W_\infty(t)$ becomes:
	\begin{align*}
		\W_\infty(t) &= \int_\CC \left(t |\nabla_\omega f_{\infty,t}(o, y)|_\omega^2 + f_{\infty,t}(o, y) - 2m \right) u(t, y) \dvol_{\omega_\infty}(y) \\
		&= \frac{1}{(4\pi t)^{m}} \int_\CC \left( \frac{\tilde{r}^2}{4t} + \frac{\tilde{r}^2}{4t} - 2m - c \right) e^{-\frac{\tilde{r}^2}{4t} - c} \dvol_{\omega_\infty} \\
		&= \frac{1}{(4\pi t)^{m}} \int_\CC \left( \frac{\tilde{r}^2}{2t} - 2m \right) e^{-\frac{\tilde{r}^2}{4t} - c} \dvol_{\omega_\infty} - \frac{c}{(4\pi t)^{m}} \int_\CC e^{-\frac{\tilde{r}^2}{4t} - c} \dvol_{\omega_\infty} \\
		&= 0 + \log\left(\frac{\vol_{\omega_\infty}(L)}{\vol_{\eucl}(S^{2m-1})}\right).
	\end{align*}
	Thus,
	\[
	\lim_{t \to 0} \W(t) = \lim_{t \to \infty} \W(t) = \log\left(\frac{\vol_{\omega_\infty}(L)}{\vol_{\eucl}(S^{2m-1})}\right),
	\]
	since $\omega_\infty = \tilde{\Psi}^* \omega_0$, and therefore they have the same volume for their link $L$. Since $\W(t)$ is nonincreasing, it follows that 
	\[
	\W(t) = \log\left(\frac{\vol_{\omega_\infty}(L)}{\vol_{\eucl}(S^{2m-1})}\right)
	\]
	for all $t$. The arguments in the proof of Theorem \ref{TangentConesConeStructure} conclude that $(\CC, \omega)$ is a Riemannian cone. The remaining arguments in this section imply the existence of an automorphism $\Psi \in \Aut_\Scl(\CC)$ such that $\Psi^* \omega = \omega_\CC$, concluding the proof of the Liouville Theorem.
\end{proof}

\vspace{10mm}

\section{$C^{0,\alpha}$-Type Estimate for Kähler Metrics in a Neighborhood of $o$}\label{SectionCalpha}

Given the Liouville Theorem for Calabi-Yau cones, we prove a $C^{0,\alpha}$-type estimate for Kähler metrics on $B_3(o)$ that are uniformly equivalent to a Calabi-Yau cone metric and have scalar curvature controlled in $L^\infty(B_3(o))$. This requires constructing a novel seminorm analogous to the Hölder-like seminorm defined by Krylov \cite[Theorem 3.3.1]{krylovLectures1996}. The seminorm by Krylov is uniformly equivalent to the standard Hölder norm in $\R^{2m}$ and computes a weighted distance to a specified class of objects, which for $\R^{2m}$ is the set of constants. However, this set does not exist on $\CC$ unless $\CC \cup \{o\} = \R^{2m}$. Furthermore, we want to show that any Kähler metric $\omega$ as above is close to $\omega_\CC$ around the apex (see Corollary \ref{CorollaryOfCalpha}). In order to prove this, we need to prove a Hölder estimate for the Laplacian in a neighborhood of the apex $o$. The method of proof is very similar to the estimate for the complex Monge-Ampère equation in Theorem \ref{TheoremHolderBound} and hence also serves as a model for the proof of the nonlinear estimate. The approach is inspired by \cite{simonSchauder1997} and \cite{heinHigherorder2020}. In the following sections, all constants may change from line to line.

\subsection{$C^{0,\alpha}$-Type Seminorm}\label{SectionHolderIntro}

The estimates for the Laplacian or complex Monge-Ampère equation rely on comparing the functions or 2-forms to a specified set of comparison objects. We define the comparison set and the Hölder-like seminorm, and then discuss how it relates to the usual $C^{0,\alpha}$-seminorm from Definition \ref{DefHolderUsual}.

\begin{definition}[Coordinate Balls]
	Let $p\in \{r=1\}\subset \CC$ be a point, and consider a holomorphic coordinate ball $B_\rho(p)$ around $p$. Cover $\{r=1\}$ with finitely many such balls $B_{\rho_i}(p_i)$ centered at points $p_i \in \{r=1\}$ for $i=1,\dots,n$. Define the rescaling map $\Phi_{\epsilon}(r,x) = (\epsilon r,x)$. Consider the family 
	\[
	\UU \coloneqq \{ B_{\epsilon \rho_i} (\Phi_\epsilon(p_i)) \mid \epsilon>0, \, i=1,\dots,n \},
	\]
	consisting of all rescalings of the balls $B_{\rho_i}(p_i)$ with coordinates inherited by scaling. $\UU$ is then called a family of coordinate balls, and any $U\in \UU$ is a coordinate ball of $\UU$.
\end{definition}

\begin{definition}[Comparison set]\label{DefinitionComparisonSet}
	Let $(\CC,\omega_{\CC})$ be a conical Calabi-Yau manifold $\CC\cong \R_+\times L$. Define the following:
	\begin{enumerate}
		\item For $C \geq 1$, define 
		\begin{equation*}
			\Sigma_{C}^2 \coloneqq \{\Psi^*(\omega_\CC) \mid \Psi\in \Aut_\Scl(\CC),\; \frac{1}{C} \omega_\CC \leq \Psi^*\omega_\CC \leq C \omega_\CC \}
		\end{equation*}
		as the set of pullbacks of $\omega_\CC$ by automorphisms of $\CC$ commuting with scaling and bounded uniformly above and below.
		\item Let $\UU$ be a family of coordinate balls. Let $\Sigma_{\loc}^2$ be the collection of $(1,1)$-forms on any coordinate ball $U$ that is constant in the associated coordinates. Notice that any element in $\Sigma_{\loc}^2$ is only defined on a coordinate ball and not globally. 
		\item For any element $\pi\in \Sigma_\loc^2$ and set $V\subset {\CC}$, define the infinity indicator function
		\begin{equation*}
			\mathds{1}_V(\pi)\coloneqq \begin{cases}
				1 & \text{if $\pi$ is defined on $V$,}\\
				\infty & \text{otherwise}.
			\end{cases}
		\end{equation*}
		More elaborately, $\mathds{1}_V(\pi)=1$ for $\pi\in \Sigma_\loc^2\setminus\{0\}$ if and only if $V$ is contained in the coordinate ball $U$ where $\pi$ is defined, otherwise $\mathds{1}_V(\pi) = \infty$. We assume that $0\in \Sigma_{\loc}^2$ is globally defined and set $\mathds{1}_V(0)=1$ for any set $V \subset \CC$.
	\end{enumerate}
\end{definition}

We define the new $C^{0,\alpha}$-type seminorm using the given comparison set $\Sigma_\loc^2$ and any other set $\Pi$ of globally defined 2-forms on $\CC$

\begin{definition}[${C^{0,\alpha}}'$-seminorm]\label{DefinitionHolderNorm}
	Let $\omega\in \Omega^2(U) $ be an open set $V\subset \CCbar$. Pick a family of coordinate balls $\UU$ as in Definition \ref{DefinitionComparisonSet}. Let $f\colon (0,\infty)\to(0,\infty)$ be a positive function. Define the Hölder-type $C^{0,\alpha}$-seminorm, namely the ${C^{0,\alpha}}'$-seminorm, on $V$ by
	\begin{equation}\label{Holder'}
		\begin{aligned}
			[\omega]_{\alpha,V,f, \Pi\times \Sigma_\loc^2}' = &\sup_{\substack{\rho,\nu\in (0,\infty)\\x\in V}} \rho^{-\alpha} f(\dist_{\omega_{\CC}}(o,(B_\rho(x) \cap V)\setminus B_\nu(x)))\\
			&\times\inf_{(\pi,\eta)\in \Pi \times \Sigma_\loc^2} \mathds{1}_{(B_\rho(x)\cap V)\setminus B_\nu(x)}(\eta) \Norm{\omega-\pi - \eta}{0,(B_\rho(x)\cap V)\setminus B_\nu(x)}.
		\end{aligned}
	\end{equation}
	If $f \equiv 1$, we omit the function in the seminorm and write $[\omega]_{\alpha,V,1, \Pi \times\Sigma_\loc^2}' = 	[\omega]_{\alpha,V, \Pi \times\Sigma_\loc^2}'$.
	The full ${C^{0,\alpha}}'$-norm is
	\begin{equation*}
		\Norm{\omega}{0,\alpha, V,f,\Pi \times \Sigma_\loc^2}' = \Norm{\omega}{0, V} + 	[\omega]_{\alpha,V,f,\Pi \times \Sigma_\loc^2}'.
	\end{equation*}
\end{definition}

\begin{remark}\
	\begin{enumerate}
		\item We consider $\Pi = \Sigma_{3C}^2$ in Theorem \ref{TheoremHolderBound}. In the proofs, we also consider $\Pi = \emptyset$ or $\Pi = \rho^{-\alpha}(\Sigma_{3C}^2 - \pi)$ for some $\pi\in \Sigma_{3C}^2$ and $\rho>0$. In any of these case, the infimum in \eqref{Holder'} is realized and replaced by a minimum.
		\item On $\R^m$ and for $\zeta\in \Omega^k(V)$, the seminorm
		\begin{equation}\label{HolderSemiEuclideanAlt}
			[\zeta]_{\alpha,V,\Sigma_{\text{const}}^k}' = \sup_{\substack{\rho\in (0,\lambda]\\ x\in V}} \rho^{-\alpha} \inf_{\pi\in 
				\Sigma_{\text{const}}^k} \Norm{\zeta-\pi}{0,B_\rho(x)\cap V}
		\end{equation}
		is uniformly equivalent to the usual $C^{0,\alpha}$-seminorm on $\R^m$ (see
		\cite[Theorem 3.3.1]{krylovLectures1996}) if $\Sigma_{\text{const}}^k$ is the set of constant $k$-forms on $\R^m$.
		\item All seminorms in Definition \ref{DefinitionHolderNorm} depend on the family $\UU$ of coordinate balls
		\item The function $f$ is an arbitrary weight function that will be used to ensure that the seminorm a priori exists close to the apex.
		\item The factor $\mathds{1}_{(B_\rho(x)\cap V)\setminus B_\nu(x)}(\pi)$ ensures that if $\pi\in \Sigma_\loc^2$ is a constant 2-form in a coordinate ball $U$, then $\pi$ is actually defined on the set $(B_\rho(x)\cap V)\setminus B_\nu(x)$. Recall that $0\in \Sigma_{\loc}^2$ is globally defined and $\mathds{1}_{(B_\rho(x)\cap V)\setminus B_\nu(x)}(0) = 1$ always.
		\item In this chapter, all balls, radii, and norms are with respect to $\omega_{\CC}$.
	\end{enumerate}
\end{remark}

\subsection{Linear $C^{0,\alpha}$-Estimate}\label{SectionLinear}

The main proposition of the linear case is the following: A $C^{0,\alpha}$-estimate for functions on $B_3(o)$ with a sufficiently regular Laplacian.

\begin{prop}\label{SchauderLinear}
	Let $(\CC, \omega_{\CC})$ be a conical Calabi-Yau manifold. Let $\omega$ be a Kähler metric on $B_3(o)$ such that
	\begin{equation}\label{CorLinearHolderMetricCond}
		\frac{1}{C}\omega_{\CC} \leq \omega \leq C \omega_\CC, \quad \Norm{\Scal(\omega)}{0,B_3(o)} \leq D.
	\end{equation}
	Assume that $\varphi\colon B_3(o)\to \R$ is a smooth function satisfying
	\begin{equation}\label{PhiBound}
		\Norm{\varphi}{0, B_3(o)}\leq C_1, \quad \Norm{\Delta_{\omega} \varphi }{0, B_3(o)} \leq C_2,
	\end{equation}
	\sloppy for $C_1,C_2>0$. Then for all $\alpha \in (0,1)$ there exists a constant $C_3 = C_3({\CC},\omega_{\CC},C, D,C_1,C_2,\alpha)$ such that
	\begin{equation*}
		[\varphi]_{\alpha, B_1(o)} \leq C_3.
	\end{equation*}
\end{prop}

\begin{remark}
	In Proposition \ref{SchauderLinear}, \eqref{CorLinearHolderMetricCond} can be replaced by the assumption that $\omega = \omega_{C(L)}$ is itself a Kähler cone metric, where the manifold $C(L)$ does not necessarily support a conical Calabi-Yau metric. The proof is the same, except we no longer choose a sequence $(\omega_i)$ for $\omega$, but instead keep the metric fixed. However, the constant $C_3$ cannot be chosen uniformly for any choice of Kähler cone metric. Furthermore, the proof now only holds for $\alpha>0$ small enough.
\end{remark}

We will actually prove Proposition \ref{SchauderLinear} via the primed seminorm defined in Definition \ref{DefinitionHolderNorm}, modified to support functions. For any function $\varphi$ on $V$, define
\begin{equation}\label{Holder'Function}
	\begin{aligned}
		[\varphi]_{\alpha,V,f, \C}' = \sup_{\substack{\rho,\nu\in (0,\infty)\\x\in V}} \rho^{-\alpha} &f(\dist_{\omega_{\CC}}(o,(B_\rho(x) \cap V)\setminus B_\nu(x)))\\
		 &\times \min_{\zeta\in \C} \Norm{\varphi-\zeta}{0,(B_\rho(x)\cap V)\setminus B_\nu(x)}.
	\end{aligned}
\end{equation}
The function $\mathds{1}_{(B_\rho(x)\cap V)\setminus B_\nu(x)}$ is not necessary as all constant functions are defined globally. The full norm is
\begin{equation*}
	\Norm{\varphi}{0,\alpha, V,f, \C}' = \Norm{\varphi}{0, V} + [\varphi]_{\alpha,V,f, \C}'.
\end{equation*}
By the following lemma, this is equivalent to the usual seminorm:

\begin{lemma}\label{PropAlternativeHolderEquivConstant}
	On any open $V\subset \CC$ and $\varphi\in C^{0,\alpha}(V)$, then
	\begin{equation*}
		\Norm{\varphi}{0,\alpha,V} \leq 2\Norm{\varphi}{0,\alpha,V,\C}',
	\end{equation*}
	assuming that the right-hand side exists. Conversely, if $\varphi\in C^{0,\alpha}(V)$:
	\begin{equation*}
		\Norm{\varphi}{0,\alpha,V,\C}' \leq \Norm{\varphi}{0,\alpha,V}.
	\end{equation*}
\end{lemma}

\begin{proof}
	Take two points $x,y\in V$ and let $\rho \coloneqq \dist_{\omega_\CC}(x,y)$. Then
	\begin{align*}
		\Norm{\varphi}{0,\alpha,V} &= \Norm{\varphi}{0,V} + \sup_{x\neq y\in V}\frac{|\varphi(x) - \varphi(y)|}{\rho^\alpha}\\ 
		&= \Norm{\varphi}{0,V} + \sup_{x\neq y\in V} \min_{\zeta\in \C} \rho^{-\alpha} \left( |\varphi(x)-\zeta+\zeta-\varphi(y)|\right)\\
		&\leq\Norm{\varphi}{0,V} + \sup_{x\neq y\in V} \min_{\zeta\in \C}\rho^{-\alpha}\left(|\varphi(x)-\zeta|+|\zeta-\varphi(y)|\right)\\
		&\leq \Norm{\varphi}{0,V} + 2 \sup_{\rho>0, x\in V} \min_{\zeta \in \C}\rho^{-\alpha} \Norm{\varphi - \zeta}{0,B_\rho(x)\cap V}\\
		&\leq 2 \Norm{\varphi}{0,\alpha,V,\C}'.
	\end{align*}

	 Conversely:
	\begin{align*}
		\Norm{\varphi}{0,\alpha,V,\C}' =& \Norm{\varphi}{0,V} + [\varphi]_{\alpha,V,\C}'\\
		=& \Norm{\varphi}{0,V} + \sup_{\rho>0, x\in V} \min_{\zeta \in \C} \rho^{-\alpha} \Norm{\varphi - \zeta}{0,B_\rho(x)\cap V}\\
		=& \Norm{\varphi}{0,V} + \sup_{\rho>0, x\in V} \min_{\zeta \in \C} \rho^{-\alpha} \sup_{y\in B_\rho(x)\cap V} |\varphi(y) - \zeta|\\
		\leq& \Norm{\varphi}{0,\alpha,V},
	\end{align*}
	where for every $x\in V$ we choose $\zeta\in \C$ such that $\varphi(x) = \zeta$.
\end{proof}

Next, the condition \eqref{CorLinearHolderMetricCond} on the metric in Proposition \ref{SchauderLinear} shows that any such sequence $(\omega_i)$ has a subconvergent limit when blowing up.

\begin{lemma}\label{LemmaScalarCurvControlSublimit}
	If $(\CC,\omega_\CC)$ is Calabi-Yau and $(\omega_i)$ is a sequence of Kähler metrics on $B_3(o)$ satisfying the estimates
	\begin{equation}\label{UnifBoundScalarControl}
		\frac{1}{C}\omega_{\CC} \leq \omega_i \leq C \omega_\CC, \quad \Norm{ \Scal(\omega_i)}{0,B_3(o)} \leq D,
	\end{equation}
	then the blowup $\epsilon_i^{-2} \Phi^*_{\epsilon_i} \omega_i$ subconverges to $\Psi^*\omega_{\CC}$ in $C^{1,\beta}_\loc(\CC)$ for all $\epsilon_i\to 0$, $\beta<1$, and for some $\Psi\in\Aut_{\Scl}(\CC)$.
\end{lemma}
\begin{proof}
	In the proof, we pass to subsequences several times without explicitly mentioning it. Let $\tilde\omega_i \coloneqq \epsilon_i^{-2} \Phi^*_{\epsilon_i}\omega_i$. Writing $\omega_i^m = e^{F_i}\omega_\CC^m$, the scalar curvature condition in \eqref{UnifBoundScalarControl} is equivalent to
	\begin{equation}\label{LaplacianF}
		\Norm{ \Delta_{\omega_i} F_i}{0,B_3(o)} \leq D.
	\end{equation}
	Defining the rescaling $\tilde F_i \coloneqq \Phi_{\epsilon_i}^* F_i$, then
	\begin{equation}\label{LaplacianFRescaled}
		\Norm{\Delta_{\tilde\omega_i} \tilde F_i}{0,B_{3(\epsilon_i)^{-1}}(o)} \leq D \epsilon_i^{2}.
	\end{equation}
	Proposition \ref{PropRegCMA} shows that $\tilde\omega_i $ is uniformly bounded in $C^{1,\beta}_{\loc}(\CC)$ with uniform bounds for all $i$ and all $0<\beta<1$. Hence, as $i\to \infty$, we obtain a $C^{1,\beta}_\loc(\CC)$-sublimit $\tilde\omega_i \to \omega_\infty$ for all $\beta<1$, satisfying the equation
	\begin{equation*}
		\Scal(\omega_\infty) = -\Delta_{\omega_\infty} \tilde F_\infty = 0
	\end{equation*}
	weakly and with $\omega_\infty^m = e^{\tilde F_\infty}\omega_\CC^m$. Since $\omega_\infty$ and $\tilde F_\infty$ are both in $C^{1,\beta}_\loc(\CC)$, the Schauder estimates \cite[Theorem 9.19]{gilbarg_Elliptic_2001} show that $\tilde F_\infty$ is bounded in $C^{2,\beta}_\loc(\CC)$. Using \eqref{LaplacianCMABootstrap} shows that $\omega_\infty$ is bounded in $C^{2,\beta}_\loc(\CC)$. Bootstrapping finally implies that $\omega_\infty \in C^{\infty}_\loc(\CC)$. We conclude by the Liouville Theorem (Theorem \ref{LiouvilleTheorem}) that $\omega_\infty = \Psi^*\omega_\CC$ for some automorphism $\Psi\in \Aut_\Scl(\CC)$.
\end{proof}


\subsection{Proof of Proposition \ref{SchauderLinear}}

\subsubsection{Picking Weight-Function to make Seminorm Well-Defined}
	
	As $[\varphi]_{\alpha, B_1(o)}$ is equivalent to $[\varphi]_{\alpha, B_1(o), \C}'$ by Lemma \ref{PropAlternativeHolderEquivConstant}, we give the proof for the ${C^{0,\alpha}}'$-seminorm. 
	
	A priori, the regular Hölder seminorm $[\varphi]_{\alpha, B_2(o)}$ does not exist in a neighborhood of $o$. However, we can control the growth of $[\varphi]_{\alpha,B_2(o)\setminus B_\nu(o)}$ as $\nu\to 0$ as follows: Take a point $p\in B_2(o)$ and $\rho = \dist_{\omega_\CC}(o,p)$. Apply a rescaling such that $\Phi^{-1}_\rho(p) = \tilde p$ with $\dist_{\omega_\CC}(o,\tilde p) = 1$. Define $\tilde \varphi \coloneqq \Phi^*_\rho\varphi$. Corollary \ref{HolderBoundFunction1Alpha} provides an $A>0$ such that $[\tilde\varphi]_{\alpha, B_{\frac{1}{2}}(\tilde p)} < A$ for any $\varphi$ satisfying the requirements of Proposition \ref{SchauderLinear}. Scaling back, we see:
	\begin{equation*}
		[\varphi]_{\alpha, B_{\frac{\rho}{2}}(\tilde p)} < A\rho^{-\alpha}.
	\end{equation*}
	Therefore:
	\begin{equation*}
		[\varphi]_{\alpha, B_{\frac{\rho}{2}}(p)} < A\rho^{-\alpha},
	\end{equation*}
	for any $p\in B_2(o)\setminus B_\rho(o)$. The seminorm on $B_2(o)\setminus B_\rho(o)$ is therefore controlled on radii smaller than $\frac{\rho}{2}$. For radii greater than this, take $x,y\in B_2(o)$ with $\dist_{\omega_\CC}(x,y) \geq \frac{\rho}{2}$. Then
	\begin{equation*}
		\frac{|\varphi(x)-\varphi(y)|}{\dist_{\omega_\CC}(x,y)^{\alpha}} \leq 2\left(\frac{\rho}{2}\right)^{-\alpha} \Norm{\varphi}{0, B_3(o)}\leq 4\rho^{-\alpha}C_1.
	\end{equation*}
	We conclude that, for any $\delta>0$:
	\begin{equation*}
		[\varphi]_{\alpha, B_2(o)\setminus B_\delta(o)} \leq A \delta^{-\alpha},
	\end{equation*}
	and so by Lemma \ref{PropAlternativeHolderEquivConstant}:
	\begin{equation}\label{SeminormGeneralBlowupLinear}
		[\varphi]_{\alpha, B_2(o)\setminus B_\delta(o),\C}' \leq A \delta^{-\alpha}.
	\end{equation}
	
	This leads us to define the following function to control the growth of $[\varphi]_{\alpha, B_2(o)\setminus B_\delta(o),\C}'$ as $\delta\to 0$: For any $\delta>0$, define the function $f_\delta \colon [0,2] \to [0,1]$ by
	\begin{equation*}
		f_\delta(r) = \begin{cases}
			\delta^{-\frac{1+\alpha}{2}}r^{\frac{1+\alpha}{2}} & 0 \leq r \leq \delta,\\
			1 &  r\in (\delta,1),\\
			2-r & r \in [1,2].
		\end{cases}	
		\end{equation*}
	The case $r \in [1,2]$ is to keep the seminorm away from the outer boundary. Using this seminorm, it is guaranteed that $[\varphi]_{\alpha, B_2(o), f_{\delta},\C}'$ always exists for any $\delta>0$. As $\alpha<1$, notice that
	\begin{equation*}
		\alpha < \frac{1+\alpha}{2} < 1.
	\end{equation*}
	This will be important later for the application of Proposition \ref{PropGrowthRateToHomogeneousFunction}.
	\\

	Proceeding by contradiction, assume that there exists a sequence of functions $(\varphi_i)$ and metrics $(\omega_i)$ satisfying the assumptions of the proposition and with
	\begin{equation}\label{PhiSeminormToInfty}
		[\varphi_i]_{\alpha, B_2(o), f_{i^{-1}}, \C}' \to \infty
	\end{equation}
	for $\alpha < 1$. As $i\to \infty$, the weight function on $B_1(o)$ becomes identically $1$.\\
	
	For a sequence $\epsilon_i\to 0$, form the rescaled functions $\tilde \varphi_i \coloneqq \Phi_{\epsilon_i}^*\varphi_i$ and metrics $\tilde \omega_i \coloneqq \epsilon_i^{-2}\Phi^*_{\epsilon_i} \omega_i$, now defined on $B_{3(\epsilon_i)^{-1}}(o)$. The $C^0$-norm of $\tilde\varphi_i$ is preserved, but the Hölder seminorm changes:
	\begin{equation}\label{ScalingXi}
		[\tilde\varphi_i]_{\alpha,B_{2(\epsilon_i)^{-1}}(o),  \Phi^*_{\epsilon_i} f_{i^{-1}}, \C}' = \epsilon_i^{\alpha} [\varphi_i ]_{\alpha, B_2(o), f_{i^{-1}}, \C}'.
	\end{equation}
	The Laplacian \eqref{PhiBound} of $\varphi$ scales by
	\begin{equation}\label{PhiLaplacianScaled}
		\Norm{\Delta_{\tilde\omega_i}\tilde \varphi}{0,B_{3(\epsilon_i)^{-1}}(o)} \leq C_2 \epsilon_i^{2}.
	\end{equation}
	Fix the sequence $(\epsilon_i)$ such that
	\begin{equation}\label{PhiSeminormEqual1}
		[\tilde\varphi_i]_{\alpha,B_{2(\epsilon_i)^{-1}}(o),  \Phi^*_{\epsilon_i} f_{i^{-1}}, \C}' = 1
	\end{equation}
	for all $i$. The condition $[\varphi_i]'_{\alpha,B_{2}(o), f_{i^{-1}},\C} \to \infty$ as $i \to \infty$ implies that $\epsilon_i \to 0$. Therefore, the metrics $(\tilde\omega_i)$ and functions $(\tilde \varphi_i)$ are defined on increasingly larger balls as $i \to \infty$, and by Lemma \ref{LemmaScalarCurvControlSublimit}, there exists a $C^{1,\beta}_\loc(\CC)$-sublimit $\omega_\infty = \Psi^*\omega_\CC$ for some $\Psi \in \Aut_\Scl(\CC)$ and for all $\beta < 1$. Furthermore, by Corollary \ref{HolderBoundFunction1Alpha}, we also obtain a limiting function $\varphi_\infty \in C^{1,\beta}_\loc(\CC)$. The next sections utilize this limiting function and metric to obtain a contradiction with \eqref{EqualToKCone}.\\
	
	\subsubsection{Changing Weight-Function and Picking Maximizing Point and Distance.}
	
	The previous choice of weight function was picked to ensure that the seminorm \eqref{PhiSeminormToInfty} exists for any $i>0$. For the proof, it will be convenient to change the function slightly, and also pick points $x_i\in \CC$, radii $\rho_i, \nu_i>0$, and element $\varphi_i\in \C$ realizing the seminorm. We do this as follows: \\
	
	For every $i$, pick a point $x_i \in \overline{B_{2(\epsilon_i)^{-1}}(o)}$, radii $\rho_i, \nu_i > 0$, and an element $\zeta_i \in \C$ realizing at least half of the seminorm $[\tilde \varphi_i]'_{\alpha,B_{2(\epsilon_i)^{-1}}(o),\Phi_{\epsilon_i}^*(f_{i^{-1}}),\C}$ as follows: Define 
	\begin{equation*}
		\delta_i \coloneqq \dist_{\omega_{{\CC}}}(o, B_{\rho_i}(x_i)\setminus B_{\nu_i}(x_i)),
	\end{equation*} 
	then
	\begin{align}\label{DataMaximizingLinear}
	\frac{1}{2} = \frac{1}{2}	[\tilde \varphi_i]'_{\alpha,B_{2(\epsilon_i)^{-1}}(o),\Phi_{\epsilon_i}^*(f_{i^{-1}}),\C} \leq (\rho_i)^{-\alpha}\Phi_{\epsilon_i}^*(f_{i^{-1}})(\delta_i) \Norm{\tilde\varphi_i-\zeta_i}{0,(B_{\rho_i}(x_i)\setminus B_{\nu_i}(x_i))\cap B_{2(\epsilon_i)^{-1}}(o)}.
	\end{align}
	$\zeta_i$ is assumed to be the minimizer on $B_{\rho_i}(x_i)\setminus B_{\nu_i}(x_i)$. A priori, it could be that $\Phi_{\epsilon_i}^*(f_{i^{-1}})(\delta_i)\to 0$ as $i\to\infty$. For convenience in the proof, we redefine $\Phi_{\epsilon_i}^*(f_{i^{-1}})$  to $\tilde f_{i}$ such that this does not occur. This can happen in two places: if $\delta_i< (i\epsilon_i)^{-1}$, or if $\delta_i \in [\epsilon_i^{-1},2\epsilon_i^{-1}]$:\\
	
	If $\delta_i \geq (i\epsilon_i)^{-1}$ already, then define $\tilde f_i \coloneqq \Phi_{\epsilon_i}^*(f_{i^{-1}})$; otherwise if $\delta_i < (i\epsilon_i)^{-1}$:
	\begin{equation*}
		\tilde f_i(r) \coloneqq \begin{cases}
			(\delta_i)^{-\frac{1+\alpha}{2}}r^{\frac{1+\alpha}{2}}, & 0 \leq r \leq \delta_i,\\
			1, &  r \in (\delta_i, \epsilon_i^{-1}),\\
			2-\epsilon_ir, & r \in [\epsilon_i^{-1},2\epsilon_i^{-1}].
		\end{cases}
	\end{equation*}
	For the weight functions, we have the crucial property:
	\begin{equation*}
		\frac{\tilde f_i(r)}{\tilde f_i(\delta_i)} \leq \frac{\Phi_{\epsilon_i}^*(f_{i^{-1}})(r)}{\Phi_{\epsilon_i}^*(f_{i^{-1}})(\delta_i)}
	\end{equation*}
	for any $r \in (0,2\epsilon_i^{-1})$. Thus, for any other $x_i',\rho_i',\nu_i'$ with $\delta_i' \coloneqq \dist_{\omega_\CC}(o, B_{\rho_i'}(x_i')\setminus B_{\nu_i'}(x_i'))$, then 
	\begin{align*}
		&\tilde f_i (\delta_i) \rho_i^{-\alpha} \Norm{\tilde\varphi_i-\zeta_i}{0,(B_{\rho_i}(x_i)\setminus B_{\nu_i}(x_i))\cap B_{2(\epsilon_i)^{-1}}(o)}\\ 
		=& \tilde f_i (\delta_i') \frac{\tilde f_i (\delta_i)}{\tilde f_i (\delta_i')} \rho_i^{-\alpha} \Norm{\tilde\varphi_i-\zeta_i}{0,(B_{\rho_i}(x_i)\setminus B_{\nu_i}(x_i))\cap B_{2(\epsilon_i)^{-1}}(o)}\\
		\geq& \frac{1}{2} \tilde f_i (\delta_i') \frac{\tilde f_i (\delta_i)}{\tilde f_i (\delta_i')} \frac{\Phi_{\epsilon_i}^* f_{i^{-1}} (\delta_i')} {\Phi_{\epsilon_i}^* f_{i^{-1}} (\delta_i)} 
		(\rho_i')^{-\alpha}  \min_{\zeta \in \C}\Norm{\tilde\varphi_i-\zeta}{0,(B_{\rho_i'}(x_i')\setminus B_{\nu_i'}(x_i'))\cap B_{2(\epsilon_i)^{-1}}(o)}\\
		\geq&\frac{1}{2} \tilde f_i (\delta_i')  (\rho_i')^{-\alpha}\min_{\zeta \in \C} \Norm{\tilde\varphi_i-\zeta}{0,(B_{\rho_i'}(x_i')\setminus B_{\nu_i'}(x_i'))\cap B_{2(\epsilon_i)^{-1}}(o)},
	\end{align*}
	with the penultimate inequality following from \eqref{DataMaximizingLinear} and $x_i,\rho_i,\nu_i, \zeta_i$ realizing at least half of the seminorm in \eqref{DataMaximizingLinear}. Thus, $x_i,\rho_i,\nu_i, \zeta_i$ also realizes at least half of $[\tilde\varphi_i]_{\alpha,B_{2(\epsilon_i)^{-1}}(o),\tilde f_i,\C}'$:
	\begin{equation*}
		\frac{1}{2} [\tilde\varphi_i]_{\alpha,B_{2(\epsilon_i)^{-1}}(o),\tilde f_i,\C}' \leq \rho_i^{-\alpha} \tilde f_i(\delta_i) \Norm{\tilde\varphi_i-\zeta_i}{0,(B_{\rho_i}(x_i)\setminus B_{\nu_i}(x_i))\cap B_{2(\epsilon_i)^{-1}}(o)}.
	\end{equation*}
	
	For the case $\delta_i \in [\epsilon_i^{-1}, 2\epsilon_i^{-1}]$, then as $\Norm{\tilde \varphi_i}{0, B_{3(\epsilon_i)^{-1}}(o)} \leq C_1$ and the Laplacian of $\tilde\varphi$ is bounded by \eqref{PhiLaplacianScaled}, Corollary \ref{HolderBoundFunction1Alpha} and Lemma \ref{PropAlternativeHolderEquivConstant} imply that $[\tilde\varphi_i]_{\alpha,B_{2(\epsilon_i)^{-1}}(o) \setminus B_1(o),\C}' \leq \frac{1}{c}$ for some constant $c>0$, weight function identically $1$, and on $B_{2(\epsilon_i)^{-1}}(o) \setminus B_1(o)$. Therefore,
	\begin{equation*}
		\tilde f_i(\delta_i) \geq c,
	\end{equation*}
	i.e. $\tilde f_i(\delta_i)$ is bounded from below. This also further shows that $\delta_i \leq \frac{2-c}{\epsilon_i}$. Assuming that $\rho_i$ stays bounded as $i\to\infty$, the distance $\dist_{\omega_\CC}(\partial (B_{2(\epsilon_i)^{-1}}(o)) , B_{\rho_i}(x_i)\setminus B_{\nu_i}(x_i))\to \infty$ as $i\to\infty$, i.e. our maximizing ball gets infinitely far away from the outer boundary in the limit.\\

	Finally, as $\tilde f_i \geq \Phi_{\epsilon_i}^*(f_{i^{-1}})$, then $[\tilde\varphi_i]_{\alpha,B_{2(\epsilon_i)^{-1}}(o),\tilde f_i,\C}' \geq 	[\tilde \varphi_i]'_{\alpha,B_{2(\epsilon_i)^{-1}}(o),\Phi_{\epsilon_i}^*(f_{i^{-1}}),\C}$. If necessary and without changing notation, make $\epsilon_i\to 0$ even faster such that
	\begin{equation}\label{EqualToKConeLinear}
		[\tilde\varphi_i]_{\alpha,B_{2(\epsilon_i)^{-1}}(o),\tilde f_i,\C}' = 1
	\end{equation}
	for any $i \in \N$. This rescaling preserves the discussion above, in particular the lower bound $\tilde f_i(\delta_i) \geq c$.
	
	\subsubsection{Setting up Contradiction and Obtaining Limiting Function and Metric}
	
	We prove Proposition \ref{SchauderLinear} by showing the following claim:\\
	
	\textbf{Claim:} As $i\to\infty$:
	\begin{equation}\label{PhiSeminormTo0}
		[\tilde\varphi_i]_{\alpha,B_{2(\epsilon_i)^{-1}}(o), \tilde f_i, \C}' \to 0.
	\end{equation}
	
	If the claim is true, we obtain a contradiction with \eqref{EqualToKConeLinear} and hence also \eqref{PhiSeminormToInfty}, proving the theorem.\\
	
	\textbf{Proof:} The first step is to show that there exists a $C^{1,\beta}_\loc(\CC)$-sublimit $\tilde\varphi_i \to \varphi_\infty$ for all $\beta<1$. Due to the convergence $\tilde\omega_i \to \omega_\infty$ in $C^{1,\beta}_{\loc}(\CC)$ (see Lemma \ref{LemmaScalarCurvControlSublimit}), the bounds coming from \eqref{PhiBound}, and rescaling:
	\begin{equation*}
		\Norm{ \Delta_{\tilde\omega_i}\tilde \varphi_i}{0, B_{3(\epsilon_i)^{-1}}(o)} \leq D\epsilon_i^2,
	\end{equation*}
	so Corollary \ref{HolderBoundFunction1Alpha} shows that we obtain a $C^{1,\beta}_\loc(\CC)$-sublimit $\varphi_\infty$. The limit satisfies the equation
	\begin{equation}\label{LaplacePhiLimitIsZero}
		\Delta_{\omega_\infty} \varphi_\infty = 0
	\end{equation}
	weakly. Elliptic regularity implies that $\varphi_\infty \in C_\loc^\infty(\CC)$ and satisfies \eqref{LaplacePhiLimitIsZero} strongly with $L^\infty$-bound coming from \eqref{PhiBound}, so $\varphi_\infty \in \C$ by Proposition \ref{PropGrowthRateToHomogeneousFunction}.\\
	
	We divide the proof of \eqref{PhiSeminormTo0} into three cases:

	\subsubsection{\textbf{Case 1:} $\rho_i\to\infty$}
	
	Since $(\tilde\varphi_i)$ is uniformly bounded by assumption and $0\in \C$, we have
	\begin{equation*}
		\frac{1}{2}[\tilde\varphi_i]_{\alpha,B_{2(\epsilon_i)^{-1}}(o),  \tilde f_i, \C}' \leq \tilde f_i(\delta_i) \rho_i^{-\alpha} \min_{\zeta \in \C}\Norm{\tilde \varphi_i-\zeta}{B_{\rho_i}(x_i)\setminus B_{\nu_i}(x_i)} \leq C_1\rho_i^{-\alpha} \to 0
	\end{equation*}
	as $i\to\infty$, which is as stated in the claim.

	\subsubsection{\textbf{Case 2:} $\rho_i$ is bounded away from $0$ and $\infty$, i.e., there exist constants $b,B>0$ such that $b<\rho_i <B$.}
	
	\begin{enumerate}[\text{Subcase 2.}1)]
		\item $0 < a < \delta_i < A < \infty$ for some $a,A>0$: The convergence $\tilde\varphi_i\to \varphi_\infty \in \C$ in $C^{1,\beta}(K)$ for a compact set $K$ containing all $B_{\rho_i}(x_i)\setminus  B_{\nu_i}(x_i)$ for $i$ sufficiently large implies that
		\begin{equation*}
			\rho_i^{-\alpha} \min_{\zeta\in\C} \Norm{\tilde\varphi_i - \zeta}{0,B_{\rho_i}(x_i)\cap B_{\nu_i}(x_i) } \leq 2b^{-\alpha}\Norm{\tilde\varphi_i - \varphi_\infty}{0,K} \to 0.
		\end{equation*}
		
		\item $\delta_i \to 0$: In this case, $B_{\rho_i}(x_i)\setminus B_{\nu_i}(x_i)$ leaves every compact set in ${\CC}$, and we cannot use the convergence $\tilde\varphi_i\to \varphi_\infty$ in $C^{1,\beta}_\loc(\CC)$ directly. For this part, we assume $\rho_i=1$ at all times to ease notation. Note that $\tilde f_i(\delta_i) = 1$, but this does not play a role here.
		
		Since $\tilde\varphi_i \to \varphi_\infty \in \C$ in $C_\loc^{1,\beta}(\CC)$, for any small but fixed $\mu>0$ and $i$ sufficiently large, we have
		\begin{equation}\label{EqOmega_iConvMuLinear}
			\Norm{\tilde\varphi_i - \varphi_\infty}{0,[\mu, 2] \times L}\leq \mu.
		\end{equation}
		By the Hölder bound \eqref{EqualToKConeLinear} and choosing $x_i = o$, there exists $\psi_i \in \C$ such that
		\begin{equation*}
			\tilde f_i(\delta_i)\Norm{\tilde\varphi_i - \psi_i}{0,[\delta_i,2\mu] \times L} \leq (2\mu)^\alpha.
		\end{equation*}
		The triangle inequality now implies that
		\begin{align*}
			\tilde f_i(\delta_i)\Norm{\varphi_\infty - \psi_i}{0,[\mu,2\mu] \times L} &=\tilde f_i(\delta_i) \Norm{\varphi_\infty - \tilde\varphi_i + \tilde\varphi_i - \psi_i}{0,[\mu,2\mu]\times L}\\ 
            &\leq \tilde f_i(\delta_i)\Norm{\varphi_\infty - \tilde\varphi_i }{0, [\mu,2\mu]\times L} + \tilde f_i(\delta_i)\Norm{ \tilde\varphi_i - \psi_i}{0,[\mu,2\mu]\times L}\\
			&\leq\mu + (2\mu)^\alpha \\
			&\leq 3\mu^\alpha.
		\end{align*}
		As $\varphi_\infty, \psi_i\in \C$, then
		\begin{equation*}
			\tilde f_i(\delta_i) \Norm{\varphi_\infty - \psi_i}{0,(0,2] \times L} \leq  3\mu^\alpha.
		\end{equation*}
		So as $\zeta_i\in \C$ is the minimizer on $B_{\rho_i}(o)\setminus B_{\nu_i}(o)\subset [\delta_i,2]\times L$:
		\begin{align*}
				\tilde f_i(\delta_i)\Norm{\tilde\varphi_i - \zeta_i}{0,B_{\rho_i}(x_i)\setminus B_{\nu_i}(x_i)} &\leq \tilde f_i(\delta_i)\Norm{\tilde\varphi_i - \varphi_\infty}{0,[\delta_i,2] \times L} \\
				&\leq \tilde f_i(\delta_i)(\Norm{\tilde\varphi_i - \varphi_\infty}{0,[\delta_i,2\mu] \times L} + \Norm{\tilde\varphi_i - \varphi_\infty}{0,[2\mu,2] \times L}) \\
				&\leq\tilde f_i(\delta_i) \Norm{\tilde\varphi_i -\psi_i+\psi_i - \varphi_\infty}{0,[\delta_i,2\mu] \times L} + \mu\\ 
				&\leq \tilde f_i(\delta_i)\Norm{\tilde\varphi_i -\psi_i}{0,[\delta_i,2\mu] \times L} + \tilde f_i(\delta_i)\Norm{\psi_i - \varphi_\infty}{0,[\delta_i,2\mu] \times L} + \mu\\
			&\leq 2\mu^\alpha +  3\mu^\alpha + \mu,
		\end{align*}
		as stated by the claim as $\mu$ can be chosen arbitrarily small for $i$ big enough.
		\item $\delta_i \to\infty$: Pick a covering $\UU$ of coordinate balls. As $\dist_{\omega_{\CC}}(x_i,o)\to\infty$, and as the coordinate balls $\UU$ have increasing size moving away from the apex, $B_{\rho_i}(x_i)$ is always contained in a coordinate ball $U$ for $i$ sufficiently large. Fix the associated holomorphic coordinates on $U$, and apply translations such that $x_i=0$ is always the center. We already know that $\dist(x_i, \partial B_{2(\epsilon_i)^{-1}}(o))\to\infty$, so for every $i$, choose an increasing sequence of balls $B_{r_i}(x_i)\subset B_{2(\epsilon_i)^{-1}}(o)$ such that $r_i\to\infty$ and $\frac{r_i}{\epsilon_i^{-1}} \to 0$. This ensures that $B_{r_i}(x_i)$ is always contained in a coordinate ball for $i$ large enough.
		
		By the Liouville Theorem (Theorem \ref{LiouvilleTheorem}) for cscK metrics and the Liouville Theorem for bounded harmonic functions on $\C^m$ \cite[Corollary 3.12]{gilbarg_Elliptic_2001}, $\tilde\omega_i|_{B_{r_i}(x_i)}\to A^*\omega_\eucl$ in $C_\loc^{1,\beta}(\C^m)$ for some $A\in \GL(m,\C)$, and $\tilde\varphi_i|_{B_{r_i}(x_i)}$ converges to a constant $\zeta_{\text{const}}\in \C$ in $C_\loc^{1,\beta}(\C^m)$. Hence,
		\begin{equation*}
			\rho_i^{-\alpha}\min_{\zeta\in \C} \Norm{\tilde\varphi_i - \zeta}{0,B_{\rho_i}(x_i)\setminus B_{\nu_i}(x_i)} \leq b^{-\alpha}\Norm{\tilde\varphi_i -\zeta_{\text{const}}}{0,B_{\rho_i}(x_i)} \to 0, \quad i\to\infty,
		\end{equation*}
		and the result follows.
	\end{enumerate}
	
	\subsubsection{\textbf{Case 3:} $\rho_i\to 0$ }
	
	In this case, we rescale ${\CC}$ at $o$ such that if $\hat x_i$ and $\hat \rho_i$ are the rescaled coordinates of $x_i$ and $\rho_i$, respectively, then $\hat \rho_i = 1$ for all $i$. This rescaling decreases the ${C^{0,\alpha}}'$-seminorm of $\tilde \varphi_i$, but after multiplying by $\rho_i^{-\alpha}$, the seminorm equals 1 in the rescaled coordinates. The $C^0$-norm now blows up, but by subtracting a constant, we can locally control the uniform norm. More precisely, we define:
	\begin{enumerate}
		\item $\hat x_i \coloneqq \Phi_{\rho_i}^{-1}(x_i)$,
		\item $\rho_i^{-2}\Phi_{\rho_i}^*(\omega_\CC) = \omega_\CC$,
		\item $\hat \omega_i \coloneqq \rho_i^{-2}\Phi_{\rho_i}^* \tilde\omega_i$,
		\item $\hat \varphi_i \coloneqq  \rho_i^{-\alpha}\Phi_{\rho_i}^* \tilde \varphi_i$,
		\item $\hat \rho_i \coloneqq \frac{\rho_i}{\rho_i} =1, \hat \nu_i \coloneqq \frac{\nu_i}{\rho_i}< 1, \hat \delta_i \coloneqq \frac{\delta_i}{\rho_i}$,
		\item $\hat f_i \coloneqq \Phi_{\rho_i}^*(\tilde f_i)$.
	\end{enumerate}
	$\omega_\CC$ is invariant under this rescaling as it is a cone metric. As stated above, the maximizing distance is $\hat \rho_i = 1$ and
	\begin{equation}\label{HatXiHolderEqual1}
		[\hat \varphi_i]_{\alpha, B_{2(\epsilon_i\rho_i)^{-1}}(o), \hat f_i,\C}' = 1.
	\end{equation}
	as $\hat \varphi_i \coloneqq  \rho_i^{-\alpha}\Phi_{\rho_i}^* \tilde \varphi_i$. However, a priori,
	\begin{equation*}
		\Norm{ \hat \varphi_i}{0,B_{2(\epsilon_i\rho_i)^{-1}}(o)}\to\infty
	\end{equation*}
	as $i\to\infty$. By construction of $\tilde f_i$, we have
	\begin{equation}\label{fBoundInfinityLinear}
		\hat f_i(\hat \delta_i) \geq c.
	\end{equation}
	Pick a fixed but arbitrary $p\in \{r=1\}$. By defining $\check \varphi_i \coloneqq \hat \varphi_i - \hat \varphi_i(p)$, we ensure that $\check{\varphi}_i(p) = 0$, and it also satisfies the estimate \eqref{HatXiHolderEqual1}:
	\begin{equation}\label{CheckXiHolderEqual1}
		[\check \varphi_i]_{\alpha, B_{2(\epsilon_i\rho_i)^{-1}}(o), \hat f_i,\C}' = 1.
	\end{equation}
	
	We again divide Case 3 into three subcases:
	
	\begin{enumerate}[\text{Subcase 3.}1)]
		\item $0<a < \hat \delta_i < A<\infty$ for some $a,A>0$: 
		We first show that $\check\varphi_i $ is uniformly bounded in $C^{0,\alpha}_\loc(B_{2(\epsilon_i\rho_i)^{-1}}(o))$: Let $\zeta^{R,\nu}\in \C$ be the constant realizing the minimum for $[\check \varphi_i]_{\alpha, B_{2(\epsilon_i\rho_i)^{-1}}(o), \hat f_i,\C}'$ on $B_R(o)\setminus B_\nu(o)$. Then \eqref{CheckXiHolderEqual1} implies that
		\begin{equation}\label{GrowthEstimateXi}
			\begin{aligned}
			\hat f_i(\nu)\Norm{\check \varphi_i}{0,B_R(o)\setminus B_{\nu}(o)} &\leq \hat f_i(\nu)\Norm{\check \varphi_i - \zeta^{R,\nu}}{0,B_R(o)\setminus B_{\nu}(o)} + \hat f_i(\nu)\Norm{\zeta^{R,\nu}}{0,B_R(o)\setminus B_{\nu}(o)}\\
			&\leq R^\alpha + \hat f_i(\nu)|\zeta^{R,\nu}(p)| \\
			&\leq R^\alpha + \hat f_i(\nu)|\zeta^{R,\nu}(p) - \check \varphi_i(p)| + \hat f_i(\nu)|\check \varphi_i(p)|\\
			&\leq 2R^\alpha
		\end{aligned}
		\end{equation}
		as $\check \varphi_i(p) = 0$ by construction. Notice that $\Norm{\zeta^{R,\nu}}{0,B_R(o)\setminus B_{\nu}(o)} = |\zeta^{R,\nu}(p)|$ as $\zeta^{R,\nu}$ is a constant. As $\hat f_i(\hat\delta_i) = 1$ and $\hat\delta_i < A<\infty$ for some $A>0$, the decay condition of $\hat f_i$ and \eqref{GrowthEstimateXi} show that
		\begin{equation}\label{C0LocBoundXi}
			|\check \varphi_i| \leq C \begin{cases}
				r^{-\frac{1+\alpha}{2}} & r \leq 1, \\ r^{\alpha} & r > 1.
			\end{cases}
		\end{equation}
		Due to the $C^{0}_\loc(B_{2(\epsilon_i\rho_i)^{-1}}(o))$-bound \eqref{C0LocBoundXi} of $\check \varphi_i$, $\epsilon_i \to 0$, and
		\begin{equation*}
			\Norm{\Delta_{\hat \omega_i} \check \varphi_i}{0, B_{3(\epsilon_i\rho_i)^{-1}}(o)} \leq C_2\epsilon_i^{2}\rho_i^{2-\alpha},
		\end{equation*}
		we obtain a $C^{1,\beta}_{\loc}(\CC)$-limit $\check\varphi_\infty$ and a $C^{1,\beta}_{\loc}(\CC)$-limit $\hat\omega_\infty = \Psi^*\omega_\CC$ for all $\beta<1$ satisfying
		\begin{equation}\label{EqDeltaXiInfty}
			\Delta_{\hat \omega_\infty} \check \varphi_\infty = 0
		\end{equation}
		weakly, and hence strongly by elliptic regularity. $\check \varphi_\infty$ also has the growth rate \eqref{C0LocBoundXi} as $\hat \delta_i < A$, so by Proposition \ref{PropGrowthRateToHomogeneousFunction}, Ricci-flatness of $\hat\omega_\infty$, and $\alpha< \frac{1+\alpha}{2}< 1$, we conclude that $\check\varphi_\infty\in \C$. As $\hat\rho_i = 1$, then
		\begin{equation*}
			1 = [\check \varphi_i]_{\alpha, B_{2(\epsilon_i\rho_i)^{-1}}(o), \hat f_i,\C} \leq 2 \Norm{\check \varphi_i- \check\varphi_\infty}{0,K} \to 0, \quad i\to \infty,
		\end{equation*}
		for some compact set $K$ containing all $B_1(\hat x_i)\setminus B_{\hat \nu_i}(\hat x_i)$, showing the claim in this case.
		\item $\hat\delta_i = \dist_{\omega_{{\CC}}}(B_1(\hat x_i)\setminus B_{\hat \nu_i}(\hat x_i),o)\to 0$: Subcase 3.1 showed that $\check \varphi_i \to \check \varphi_\infty$ in $C^{1,\beta}_\loc(\CC)$. Given this, the proof is exactly as in subcase 2.2.
		\item $\dist_{\omega_{\CC}}(\hat x_i,o)\to \infty$: As in subcase 2.2, assume that $B_1(\hat x_i)$ is always contained in a coordinate ball for $i$ great enough, and apply a translation in the associated holomorphic coordinates such that $\hat x_i=0$. Fix a ball $B_{r_i}(\hat x_i)$ for any $i$ such that $r_i\to\infty$ and $B_{r_i}(\hat x_i)$ is always contained in a coordinate ball. 
		
		Redefine $\check\varphi_i \coloneqq \hat\varphi_i - \hat\varphi_i(\hat x_i)$, then the estimates of Subcase 3.1 show that
		\begin{equation*}
			\left|\check\varphi_i |_{B_{r_i}(\hat x_i)} \right| \leq Cr^{\alpha}
		\end{equation*}
		for $0\leq r \leq r_i$. The estimate of the Laplacian also holds:
		\begin{equation*}
			\Norm{\Delta_{\hat \omega_i} \check \varphi_i}{0, B_{3(\epsilon_i\rho_i)^{-1}}(o)} \leq C_2\epsilon_i^{2}\rho_i^{2-\alpha}.
		\end{equation*}
		Therefore, we obtain $C^{1,\beta}_{\loc}(\C^m)$-sublimits $\check\varphi_\infty, \hat \omega_\infty$ for $\beta<1$ satisfying $\check \varphi_\infty = \OO(r^\alpha)$ and $\Delta_{\hat \omega_\infty} \check\varphi_\infty = 0$ weakly. By the Liouville Theorem (Theorem \ref{LiouvilleTheorem}), then $\hat\omega_\infty = A^* \omega_{\eucl}$ for some matrix $A\in \GL(m,\C)$, and hence the usual Liouville Theorem for harmonic functions with sublinear growth (or see Proposition \ref{PropGrowthRateToHomogeneousFunction}) implies that $\check\varphi_\infty \in \C$. This again contradicts $\hat \rho_i = 1$ and \eqref{CheckXiHolderEqual1} by the arguments in subcase 2.3.
	\end{enumerate}


\subsection{Nonlinear $C^{0,\alpha}$-Estimate: Preliminaries}\label{SectionNonlinear}

After proving the linear Schauder estimate in Proposition \ref{SchauderLinear}, we are now in a position to prove an analogous Evans-Krylov estimate for the complex Monge-Ampère equation on $(\CC, \omega_\CC)$. The method of proof is very similar to Proposition \ref{SchauderLinear}, except that there no longer is a concept of globally constant 2-forms as there is for functions. Therefore, we include $\Sigma_{\loc}^2$ in the comparison set to handle the case $x_i \to \infty$ and get control over the regular Hölder seminorm. This section provides some necessary results for the regularity of the comparison sets $\Sigma_{3C}^2$ and $ \Sigma_{\loc}^2$.

\begin{lemma}\label{LemmaHolderBound}
	Let $K'\Subset K^\degree$, $K\Subset \CC$ be compact sets, and fix $\beta \in (0,1)$. Take a Kähler form $\omega$ on $K$ with $\frac{1}{C}\omega_\CC \leq \omega \leq C\omega_\CC$. Pick any $\Psi\in\Aut_\Scl(\CC)$ such that $\frac{1}{3C}\omega_\CC \leq \Psi^*\omega_\CC \leq 3C\omega_\CC$.
	Assume that
	\begin{equation*}
		\norm{\Scal \omega}_{0, K} \leq \delta, \quad \Norm{\omega - \Psi^*\omega_\CC}{0,K}\leq \delta
	\end{equation*}
	for $\delta \leq 1$. Then
	\begin{equation*}
		\norm{\omega-\Psi^*\omega_\CC}_{0,\beta, K'} \leq B \delta,
	\end{equation*}
	for any $\beta < 1$ and $B = B(C,K,K',\beta)$ independent of $\Psi$ and $\delta>0$.
\end{lemma}

\begin{proof}
	Choose a sequence of compact sets
	\begin{equation*}
		K' \Subset (K''')\degree \subset K''' \Subset (K'')\degree \subset K'' \subset K\degree.
	\end{equation*}
	By proposition \ref{PropRegCMA}, then $\norm{\omega}_{1,\beta, K''} \leq A$ for some constant $A > 0$ and any $\beta< 1$ as $\delta \leq 1$.	Similarly, $\norm{\Psi^*\omega_\CC}_{1,\beta, K''}\leq A$ as $\frac{1}{3C}\omega_\CC \leq \Psi^*\omega_\CC \leq 3C\omega_\CC$.  Write $e^G\omega^m =(\Psi^*\omega_\CC)^m$. The condition $\norm{\Scal \omega}_{0, K} \leq \delta$ is equivalent to $\norm{\Delta_\omega G}_{0, K} \leq \delta$. Then:
	\begin{align*}
		e^G\omega^m &= (\Psi^*\omega_\CC)^m\\
		&=  (\omega + (\Psi^*\omega_\CC- \omega))^m \\
		&=  \omega^m + m  \omega^{m-1} \wedge (\Psi^*\omega_\CC - \omega) + \OO((\Psi^*\omega_\CC - \omega)^2).
	\end{align*}
	For every point $p\in K'''$ and by \cite[Lemma 2.1]{chenalpha2015}, there is a ball $B_\epsilon(p)\subset K''$ with a $C^2$-potential $f$ such that $i \del\delbar f = \Psi^*\omega_\CC - \omega$ and $C^{0}(B_{\frac{\epsilon}{2}}(p))$-norm uniformly bounded by $\delta$. Then:
	\begin{equation}\label{LaplacianFLemma}
		m \Delta_{\omega} f = m \tr_{\omega} (\Psi^*\omega_\CC - \omega) = (e^G-1) + \OO((\omega - \Psi^*\omega_\CC )^2).
	\end{equation}
	Cover $K'''$ with finitely many such balls. Since:
	\begin{equation*}
		\Norm{\Delta_\omega G}{0,K''} \leq C\delta, \qquad \Norm{e^{G}-1}{0,K} < C' \Norm{\omega - \Psi^*\omega_\CC }{0,K} \leq C' \delta,
	\end{equation*}
	then $\Norm{G}{0,K}\leq C' \delta$, so Corollary \ref{HolderBoundFunction1Alpha} applied to each $B_\epsilon(p)$ implies that:
	\begin{equation*}
		\begin{aligned}
			\Norm{G}{1, \beta,K'''} &\leq C' \left(\norm{G}_{0, K''} + \Norm{\Delta_\omega G}{0,K''}\right) \leq C' \delta,
		\end{aligned}
	\end{equation*}
	for all $\beta < 1$ and $C'$ depending on $\beta$. Using the product rule, this now lifts to $e^{G}-1$:
	\begin{equation*}
		\Norm{e^{G}-1}{1, \beta,K'''} \leq C' \delta.
	\end{equation*}
	We estimate the remaining terms on the right-hand side as:
	\begin{align*}
		\Norm{({\omega} - \Psi^*\omega_\CC)^l}{0, \beta,K'''} &\leq D' \Norm{{\omega} - \Psi^*{\omega_\CC}}{0,K'''}^{l-1} \Norm{{\omega} - {\Psi^*\omega_\CC}}{0, \beta,K'''}\\ 
		&\leq D' \Norm{{\omega} - \Psi^*\omega_\CC}{0, K''} \\
		& \leq D' \delta
	\end{align*}
	for any $l \geq 2$. The result follows by applying the Schauder estimates for $\Delta_{{\omega}}$ to \eqref{LaplacianFLemma}.
\end{proof}

\begin{lemma}\label{LemmaCompactnessAutomorphisms}
	The set $\{\Psi \in \Aut_{\Scl}(\CC) \mid \frac{1}{C} \leq \Norm{\Psi^*\omega_\CC}{0, K} \leq C \}$ is compact for any constant $C \geq 1$ and compact set $K \in \CC$ such that the interior $K^\degree$ contains a level set $\{r = c\} \subset K^\degree$. 
\end{lemma}

\begin{proof}
	\sloppy By Lemma \ref{ConeEmbedding}, there is an embedding $\CC \to \C^N$ such that the coordinate functions $z_1,\dots, z_N$ are linearly independent holomorphic and homogeneous functions on $\CC$ and $\Aut_{\Scl}(\CC)\hookrightarrow \GL(N,\C)$. Hence, we identify $\Aut_{\Scl}(\CC)$ with a closed subgroup of $\GL(N,\C)$. Then:
	\begin{align*}
		C \geq \Norm{\Psi^*\omega_\CC}{0,K} &= \max_{p\in K} \max_{\substack{v\in T_p \CC \\ \norm{v}_{\omega_\CC} = 1}} |g_\CC(D\Psi(v), D\Psi(v))|^{\frac{1}{2}} \\
		&\geq \frac{1}{C^2} \max_{p\in K}\max_{\substack{v\in T_p \CC \\ \norm{v}_{\omega_\eucl} = 1}} |g_\eucl(\Psi v,\Psi v)|^{\frac{1}{2}},
	\end{align*}
	where we have identified $T_p\CC \subset \C^N$ and $\Psi\in \GL(N,\C)$ by the embedding $\CC \to \C^N$ above. Next, we show that
	\begin{equation*}
		\Span_{\C}\left\{\bigcup_{p\in K} T_p\CC\right\} = \C^N. 
	\end{equation*}
	Assume not, then there exist constants $\lambda_1,\dots, \lambda_N$ such that
	\begin{equation}\label{CoordinateFunctionsSplit}
		\sum_{i=1}^N \lambda_i z_i(p) = \text{constant}, \quad \text{for all $p\in K$}.
	\end{equation}
	By homogeneity, and since $K^\degree$ contains a level set $\{r = c\}$, it follows that \eqref{CoordinateFunctionsSplit} splits into sums of coordinates with the same homogeneity, all equal to zero. By homogeneity, then
	\begin{equation*}
		\sum_{i=1}^N \lambda_i z_i(p) = 0, \quad \text{for all $p\in \CC$},
	\end{equation*}
	contradicting the fact that $z_1,\dots,z_N$ are linearly independent on $\CC$. The desired result follows, and we conclude that
	\begin{align*}
		C \geq \Norm{\Psi^*\omega_\CC}{0,K}	&= \frac{1}{C^2} \max_{p\in K}\max_{\substack{v\in T_p \CC \\ \norm{v}_{\omega_\eucl} = 1}} |g_\eucl(\Psi v,\Psi v)|^{\frac{1}{2}}\\
		&= D \max_{v\in S^{2N-1}\subset \C^N} |g_\eucl(\Psi v,\Psi v)|^{\frac{1}{2}}\\
		&= D \Norm{\Psi}{L^\infty_{\text{Op}}(K)},
	\end{align*}
	where $\Norm{\Psi}{L^\infty_{\text{Op}}(K)}$ is the operator norm. The constant $\frac{1}{C^2}$ was changed as $\bigcup_{p\in K} T_p\CC$ might not equal $\C^N$, only the span. By linear combinations, we can therefore estimate 
	\begin{equation*}
		\max_{v\in S^{2N-1}\subset \C^N} |g_\eucl(\Psi v,\Psi v)|^{\frac{1}{2}}.
	\end{equation*}
	$GL(N,\C)$ lies in the finite-dimensional vector space of $N\times N$ matrices, and the operator norms extend to this space. All norms are equivalent on a finite-dimensional vector space, so if $\Psi = (g_{ij})$:
	\begin{equation*}
		|g_{ij}| \leq D'.
	\end{equation*}
	Compactness now follows by $\Aut_{\Scl}(\CC)$ being a closed subgroup of $\GL(N,\C)$ under the embedding induced by $\CC \to \C^N$.
\end{proof}

\begin{lemma}\label{LemmaAutomorphismSubconvergence}
	If a sequence $\Phi_i^*\omega_\CC$ with $\frac{1}{C}\omega_\CC \leq \Phi_i^*\omega_\CC \leq C \omega_\CC$ converges to an element $\Psi_\infty^*\omega_\CC$ in $C^\infty_\loc(\CC)$, then there exists a subsequence $(i_k)$ and an isometry $F\in \Isom(\omega_\CC)\cap \Aut_\Scl(\CC)$ such that $\Phi_{i_k}\circ F\to\Psi_\infty$ in $C^{\infty}_\loc(\CC)$ as automorphisms.
\end{lemma}

\begin{proof}
	Assume that $\Phi_i^*\omega_\CC\to\Psi_\infty^*\omega_\CC$ in $C^{\infty}_\loc(\CC)$. By Lemma \ref{LemmaCompactnessAutomorphisms}, the set $\{\Phi_i\}$ is compact and subconverges to an element $\Phi_\infty\in \Aut_{\Scl}(\CC)$ such that $\Phi_\infty^*\omega_\CC = \Psi_\infty^* \omega_\CC$. Define $F \coloneqq \Phi_\infty^{-1} \circ \Psi$, an isometry by definition, and the result follows.
\end{proof}

\begin{corollary}\label{CorollaryAutomorphismDerivative}
	Assume that $\Phi_i^*\omega_\CC \to \Phi_\infty^*\omega_\CC$ for $\Phi_i,\Phi_\infty\in \Aut_{\Scl}(\CC)$ with $\frac{1}{C}\omega_\CC \leq \Phi_i^*\omega_\CC \leq C \omega_\CC$ and convergence in $C^{\infty}_\loc(\CC)$. For any closed $(1,1)$-form $\eta \in \Omega^2(\CC)$ such that $\tr_{\Phi_\infty^*\omega_\CC} \eta = \text{constant}$ and $\LL_{r\partial_r} \eta = 2\eta$, by taking a subsequence of $\Phi_i$ we can find a sequence $\Psi_i \in \Aut_\Scl(\CC)$ such that
	\begin{equation*}
		\frac{\Psi_i^*\omega_\CC - \Phi_i^*\omega_\CC}{i^{-1}} - \eta \to 0, \quad i \to \infty.
	\end{equation*}
	in $C^\infty_\loc(\CC)$. 
\end{corollary}

\begin{proof}
	By applying an automorphism, we assume without loss of generality that $\Phi_\infty^*\omega_\CC = \omega_\CC$. By Proposition \ref{PropGrowthRateToHomogeneousForm}, write $\eta = i\del\delbar (u + cr^2)$ for some $c \in \R$ and a 2-homogeneous, $\xi$-invariant function $u$ harmonic with respect to $\omega_\CC$. First assume that $c=0$ so that $\eta$ is trace-free. By Theorem \ref{DecompositionHoloFields}, $\nabla u$ is a holomorphic vector field; hence its flow generates a 1-parameter subgroup $(\Psi_t) \subset \Aut_\Scl(\CC)$. By definition:
	\begin{equation*}
		\omega_\CC(\nabla u, \cdot) = g_\CC(\nabla u, J\cdot) = du(J\cdot) = d^c u,
	\end{equation*}
	so
	\begin{equation}\label{DerivativeAutomorphism}
		\frac{d}{dt}\Bigg|_{t=0}\left(\frac{\Psi_t^*\omega_\CC - \omega_\CC}{t}\right) = \LL_{\nabla u} \omega_\CC = d(\omega_\CC(\nabla u, \cdot)) = dd^c u = \eta.
	\end{equation}
	This proves the corollary if $\Phi_i^*\omega_\CC = \omega_\CC$ is fixed. If not, then by taking a subsequence and precomposing with a fixed $F \in \Aut_\Scl^*(\omega_\CC)$, we assume that $\Phi_i \to \Id \in \Aut_{\Scl}(\CC)$ in $C^\infty_\loc(\CC)$ by Lemma \ref{LemmaAutomorphismSubconvergence}. Therefore:
	\begin{equation*}
		(\Phi_i^* - \Id^*)\left(\frac{\Psi_{\frac{1}{i}}^*\omega_\CC - \omega_\CC}{\frac{1}{i}}\right) \to 0, \quad i \to \infty,
	\end{equation*}
	due to the convergence $\Phi_t \to \Id$. We conclude that $\tilde \Psi_i \coloneqq \Psi_{\frac{1}{i}} \circ \Phi_i$ is the desired sequence for the corollary if $c=0$. If $c\neq 0$, multiply $\tilde \Psi_i$ by $(1+\frac{c}{i})$ such that \eqref{DerivativeAutomorphism} produces an extra $c\omega_\CC = c\,i\del\delbar r^2$ as $i\to\infty$.
\end{proof}

We also need the following regularity result on $\Sigma_{\loc}^2$:

\begin{lemma}\label{LocalConstantsHolder}
	Let $\eta\in \Sigma_{\loc}^2$ be defined on a coordinate ball $B_\rho(x)\in \UU$. Pick $\alpha\in (0,1)$. Denote $\nu$ as $\nu = \dist_{\omega_\CC}(o,x)>0$. Then $\Norm{\eta}{0,\alpha,V} \leq C'(1+\nu^{-\alpha})|\eta(y)|_{\omega_\CC}$ for any $y\in B_\rho(x)$ and for some $C'>0$ independent of $B_\rho(x)$, $y$, and $\eta$.
\end{lemma}

\begin{proof}
	Recall that the coordinate ball $B_\rho(x)\in\UU$ and $\eta\in \Sigma_{\loc}^2$ were defined by scaling elements and a ball centered at a point $p\in \{r=1\}$. Rescale by $\nu^{-1}$ such that $\Phi_{\nu}^{-1}(x) = p$, $\tilde \rho = \frac{\rho}{\nu}$, and $\tilde \eta = \nu^{-2}\Phi_{\nu}^* (\eta)$. As there are only finitely many coordinate balls centered around $\{r=1\}$, there exists a uniform constant $C'>0$ such that
	\begin{equation*}
		\Norm{\tilde \eta}{0, B_{\tilde \rho}(p)} \leq C' |\tilde\eta(y)|_{\omega_\CC}
	\end{equation*}
	for any $y \in  B_{\tilde \rho}(p)$. As the set of constant $(1,1)$-forms in the associated coordinates of $ B_{\tilde \rho}(p)$ is finite-dimensional and all elements are smooth, there exists a constant $C''>0$ such that
	\begin{equation*}
		\Norm{\tilde \eta}{0, \alpha, B_{\tilde \rho}(p)} \leq C'' \Norm{\tilde \eta}{0, B_{\tilde \rho}(p)}  \leq C' |\tilde\eta(y)|_{\omega_\CC}.
	\end{equation*}
	The lemma now follows by scaling back.
\end{proof}

We finally show that the regular Hölder norm controls the primed norm:

\begin{lemma}\label{LemmaHolderBoundToPrimeBound}
	Let $V\subset \CC$ be open with $\nu \coloneqq \dist_{\omega_\CC}(o,V)>0$. Then for any form $\psi\in \Omega^{1,1}(V)$:
	\begin{equation*}
		\Norm{\psi}{0,\alpha,V, \Sigma_{\loc}^2}' \leq \Norm{\psi}{0, \alpha, V}+\nu^{-\alpha}B \Norm{\psi}{0, V},
	\end{equation*}
	with $B>0$ independent of $V$.
\end{lemma}

\begin{proof}
	Take a point $x\in V$ and $\rho>0$. Then 
	\begin{align*}
		\Norm{\psi}{0,\alpha,V,\Sigma_{\loc}^2}' &= \Norm{\psi}{0,V} + [\psi]_{\alpha,V,\Sigma_{\loc}^2}'\\
		&= \Norm{\psi}{0,V} + \sup_{\rho>0, x\in V} \min_{\eta \in \Sigma_{\loc}^2} \rho^{-\alpha} \mathds{1}_{V\cap B_\rho(x)}(\eta) \Norm{\psi - \eta}{0,B_\rho(x)\cap V}\\
		&= \Norm{\psi}{0,V} + \sup_{\rho>0, x\in V} \min_{\eta \in \Sigma_{\loc}^2} \rho^{-\alpha}\mathds{1}_{V\cap B_\rho(x)}(\eta) \sup_{y\in B_\rho(x)\cap V} |\psi(y) - \eta(y)|_{\omega_\CC}\\
		&\leq \Norm{\psi}{0,V} + \sup_{\rho>0, x\in V} \min_{\eta \in \Sigma_{\loc}^2} \rho^{-\alpha} \mathds{1}_{V\cap B_\rho(x)}(\eta) \sup_{y\in  B_\rho(x)\cap V} \Big(|\P_{xy}\psi(y) - \psi(x)|_{\omega_\CC} \\
		&\hspace{60mm} + |\psi(x) - \eta(x)|_{\omega_\CC}+ |\eta(x)  - \P_{xy}\eta(y)|_{\omega_\CC}\Big).
	\end{align*}
	The term
	\begin{equation*}
		\sup_{\rho>0, x\in V} \rho^{-\alpha}  \sup_{y\in  B_\rho(x)\cap V} |\P_{xy}\psi(y) - \psi(x)|_{\omega_\CC} = [\psi]_{\alpha,V}
	\end{equation*}
	equals the regular Hölder seminorm.
	
	To estimate the other terms, fix a choice of $x$ and $\rho$. Consider the first of two possibilities: If $B_\rho(x)\cap V$ is contained in a coordinate ball, then choose some $\eta\in \Sigma_{\loc}^2$ defined in a coordinate ball $U$ containing $B_\rho(x)\cap V$ and such that $\eta(x) = \psi(x)$. Continuing:
	\begin{align*}
		& \Norm{\psi}{0,V} + \min_{\eta \in \Sigma_{\loc}^2} \rho^{-\alpha} \mathds{1}_{V\cap B_\rho(x)}(\eta) \sup_{y\in  B_\rho(x)\cap V} \Big(|\P_{xy}\psi(y) - \psi(x)|_{\omega_\CC} \\
		&\hspace{60mm} + |\psi(x) - \eta(x)|_{\omega_\CC}+ |\eta(x)  - \P_{xy}\eta(y)|_{\omega_\CC}\Big)\\
		&\leq \Norm{\psi}{0,\alpha, V} + \rho^{-\alpha}  \sup_{y\in  B_\rho(x)\cap V} \Big(|\eta(x)  - \P_{xy}\eta(y)|_{\omega_\CC}\Big).
	\end{align*}
	By Lemma \ref{LocalConstantsHolder} and $\eta = \psi(x)\in \Sigma_{\loc}^2$, it follows that
	\begin{equation*}
		\Norm{\eta}{0,\alpha, B_\rho(x)} \leq C' (1+\nu^{-\alpha})|\psi(x)|_{\omega_\CC} \leq C' (1+\nu^{-\alpha}) \Norm{\psi}{0,V},
	\end{equation*}
	so we conclude:
	\begin{equation*}
		\rho^{-\alpha}  \sup_{y\in  B_\rho(x)\cap V} \Big(|\eta(x)  - \P_{xy}\eta(y)|_{\omega_\CC}\Big)  \leq C' \nu^{-\alpha}\Norm{\psi}{0,V}.
	\end{equation*}
	This proves the lemma when $B_\rho(x)$ is contained in a coordinate ball.
	
	Now assume that $B_\rho(x)$ is not contained in a coordinate ball. Then $\eta = 0$ is the only possibility. By the scaling behavior of the coordinate balls, it follows that there exists a constant $b>0$ such that $b\nu\leq \rho$, otherwise $B_\rho(x)$ would be contained in a coordinate ball. Therefore, we conclude that
	\begin{align*}
		& \Norm{\psi}{0,V} + \min_{\eta \in \Sigma_{\loc}^2} \rho^{-\alpha} \mathds{1}_{V\cap B_\rho(x)}(\eta) \sup_{y\in  B_\rho(x)\cap V} \Big(|\P_{xy}\psi(y) - \psi(x)|_{\omega_\CC} \\
		&\hspace{60mm} + |\psi(x) - \eta(x)|_{\omega_\CC}+ |\eta(x)  - \P_{xy}\eta(y)|_{\omega_\CC}\Big)\\
		&\leq \Norm{\psi}{0,\alpha, V} + (1+\nu^{-\alpha}) b^{-\alpha}\Norm{\psi}{0, B_{\rho}(x)\cap V}\\
		&\leq \Norm{\psi}{0,\alpha, V} + (1+\nu^{-\alpha}) b^{-\alpha}\Norm{\psi}{0, V},
	\end{align*}
	finishing the proof of the lemma.
\end{proof}

\subsection{Nonlinear $C^{0,\alpha}$-Estimate and Proof}

\subsubsection{Statement of Main Theorem}

\begin{theorem}\label{TheoremHolderBound}
	Let $(\CC,\omega_\CC)$ be a Calabi-Yau cone with cone metric $\omega_\CC$, and let $\omega$ be a Kähler metric on $B_3(o)\subset \CC$ such that
	\begin{equation}\label{ConditionUniformBound}
		\frac{1}{C}\omega_\CC \leq \omega\leq C\omega_\CC, \quad \Norm{\Scal(\omega)}{0, B_3(o)} \leq D,
	\end{equation}
	for some constants $C,D>0$. Then for any $\alpha = \alpha(\CC,\omega_\CC)>0$ small enough, there exists a constant $C'$ with
	\begin{equation*}
		[\omega]'_{\alpha,B_1(o),\Sigma_{3C}^2\times \Sigma_{\loc}^2}\leq C',
	\end{equation*}
	where $\alpha= \alpha(\CC,\omega_\CC)$ and $C'= C'(\CC,C, D, \alpha,\omega_\CC, \UU)$ are independent of $\omega$.
\end{theorem}

\subsubsection{Proof: Picking Weight-Function to make Seminorm Well-Defined}

The proof proceeds by contradiction and is similar to Proposition \ref{SchauderLinear}. All constants might change from line to line below. First, for any $\omega$ satisfying \eqref{ConditionUniformBound}, the seminorm $[\omega]_{\alpha, B_1(o), \Sigma_{3C}^2 \times \Sigma_{\loc}^2}'$ might not exist a priori. However, as in Proposition \ref{SchauderLinear}, we can control the growth of $[\omega]_{\alpha, B_1(o)\setminus B_\delta(o), \Sigma_{3C}^2 \times \Sigma_{\loc}^2}'$ as $\delta\to0$ as follows: Take a point $p\in B_2(o)$ with $\rho = \dist_{\omega_\CC}(o,p)$. Apply a rescaling such that $\Phi^{-1}_\rho(p) = \tilde p$ with $\dist_{\omega_\CC}(o,\tilde p) = 1$. Define $\tilde \omega \coloneqq \rho^{-2}\Phi^*_\rho(\omega)$.  The regularity of the complex Monge-Ampère equation provides an $A > 0$ such that $[\tilde\omega]_{\alpha, B_{\frac{1}{2}}(\tilde p)} < A$ for any $\omega$ satisfying the requirements of Theorem \ref{TheoremHolderBound}. Rescaling back, we conclude as in Proposition \ref{SchauderLinear}:
\begin{equation*}
	[\omega]_{\alpha,B_1(o)\setminus B_\delta(o) } < A\delta^{-\alpha}.
\end{equation*}
Lemma \ref{LemmaHolderBoundToPrimeBound} transforms this into an estimate on the primed seminorm:
\begin{equation}\label{SeminormGeneralBlowupRegularHolder}
	[\omega]_{\alpha, B_1(o)\setminus B_\delta(o), \Sigma_{\loc}^2}' \leq A \delta^{-\alpha}.
\end{equation}
For any fixed $\Psi\in \Aut_\Scl(\CC)$ and as $\omega_\CC$ is 2-homogeneous, then by Lemma \ref{LemmaHolderBoundToPrimeBound}:
\begin{equation*}
	[\Psi^*\omega_\CC]_{\alpha,B_1(o)\setminus B_\delta(o), \Sigma_\loc^2 }' < A\delta^{-\alpha}.
\end{equation*}
Hence, the seminorm \eqref{SeminormGeneralBlowupRegularHolder} will not blow up any faster at the apex by including $\Sigma_{3C}^2$. We conclude that:
\begin{equation}\label{SeminormGeneralBlowup}
	[\omega]_{\alpha, B_1(o)\setminus B_\delta(o), \Sigma_{3C}^2\times \Sigma_{\loc}^2}' \leq A \delta^{-\alpha}.
\end{equation}

This leads us to define the following function to control the growth of $[\omega]_{\alpha, B_1(o), \Sigma_{3C}^2 \times \Sigma_{\loc}^2}'$ at $o$: For any $\delta > 0$, define the function $f_\delta \colon [0,2] \to [0,1]$ by
\begin{equation*}
	f_\delta(r) \coloneqq \begin{cases}
		\delta^{-\frac{1+\alpha}{2}}r^{\frac{1+\alpha}{2}} & 0 \leq r \leq \delta,\\
		1 &  r\in (\delta,1),\\
		2-r & r \in [1,2].
	\end{cases}	
\end{equation*}
Using this seminorm, it is guaranteed that $[\omega]_{\alpha, B_2(o), f_{\delta},\Sigma_{3C}^2 \times \Sigma_{\loc}^2}'$ always exists for any $\delta > 0$.

Now, assume that there is a sequence of Kähler metrics $(\omega_i)_{i\in \N}$ on $B_3(o)$ such that the following holds:
\begin{enumerate}
	\item $\frac{1}{C}\omega_\CC \leq \omega_i \leq C\omega_\CC$,
	\item $\Norm{\Scal(\omega_i)}{0, B_3(o)} \leq D$,
	\item $[\omega_i]'_{\alpha,B_2(o), f_{i^{-1}},\Sigma_{3C}^2 \times \Sigma_{\loc}^2} \to \infty$ as $i \to \infty$.
\end{enumerate}
If no such sequence exists, then the theorem follows because $f_{i^{-1}}\to 1$ as $i\to\infty$ in $C_\loc^0(B_1(o))$. In deriving a contradiction, we pass to a subsequence several times without explicitly mentioning it.\\

For a sequence $(\epsilon_i)$ and $\omega_i = e^{F_i}\omega_{\CC}^m$, form the rescaled metric $\tilde \omega_i \coloneqq \Phi^*_{\epsilon_i}(\epsilon_i^{-2}\omega_i)$ and $\tilde F_i \coloneqq \Phi_{\epsilon_i}^* F_i$ as in Proposition \ref{SchauderLinear}, now a metric defined on $B_{3(\epsilon_i)^{-1}}(o)$. This rescaling preserves the $C^0$-norm of $\omega_i$ by scale invariance of $\omega_\CC$, but the ${C^{0,\alpha}}'$-seminorm of $\tilde \omega_i$ shrinks:
\begin{equation}\label{RescalingShrinks}
	\begin{aligned}
		[\tilde \omega_i]'_{\alpha,B_{2(\epsilon_i)^{-1}}(o),\Phi_{\epsilon_i}^*(f_{i^{-1}}),\Sigma_{3C}^2 \times \Sigma_{\loc}^2} = \epsilon_i^{\alpha}[\omega_i]'_{\alpha,B_2(o), f_{i^{-1}},\Sigma_{3C}^2 \times \Sigma_{\loc}^2}.
	\end{aligned}
\end{equation}
Fix the rescaled seminorm along the blowup by choosing $(\epsilon_i)$ such that
\begin{equation}\label{EqualToKConePrel}
	\begin{aligned}
		[\tilde \omega_i]'_{\alpha,B_{2(\epsilon_i)^{-1}}(o),\Phi_{\epsilon_i}^*(f_{i^{-1}}),\Sigma_{3C}^2 \times \Sigma_{\loc}^2} = 1
	\end{aligned}
\end{equation}
for every $i$. The condition $[\omega_i]'_{\alpha,B_{2}(o), f_{i^{-1}},\Sigma_{3C}^2 \times \Sigma_{\loc}^2} \to \infty$ as $i \to \infty$ implies that $\epsilon_i \to 0$. Therefore, the metrics in the sequence $(\tilde\omega_i)$ are defined on increasingly larger balls as $i \to \infty$, and by Lemma \ref{LemmaScalarCurvControlSublimit} there exists a $C^{1,\beta}_\loc$-sublimit $\omega_\infty = \Psi^*\omega_\CC$ for some $\Psi \in \Aut_\Scl(\CC)$ and for all $\beta < 1$. The next sections utilize this limiting metric to obtain a contradiction with \eqref{EqualToKConePrel}.

\subsubsection{Changing Weight-Function and Picking Maximizing Point and Distance.}\label{SectionChangingWeightFunction}

The previous choice of weight function was picked to ensure that the seminorm $[\omega_i]'_{\alpha,B_2(o), f_{i^{-1}},\Sigma_{3C}^2 \times \Sigma_{\loc}^2}$ exists for any $i>0$. For the proof, it will be convenient to change the function slightly, and also pick points $x_i\in \CC$, radii $\rho_i, \nu_i>0$, and elements $\pi_i\in \Sigma_{3C}^2, \eta_i\in  \Sigma_{\loc}^2$ realizing the seminorm. We do this as follows: \\

For every $i$, pick a point $x_i \in \overline{2B_{\epsilon_i^{-1}}(o)}$ and radii $\rho_i, \nu_i > 0$ realizing at least half of the seminorm $[\tilde \omega_i]'_{\alpha,B_{2(\epsilon_i)^{-1}}(o),\Phi_{\epsilon_i}^*(f_{i^{-1}}),\Sigma_{3C}^2 \times \Sigma_{\loc}^2}$. Define $\delta_i \coloneqq \dist_{\omega_{{\CC}}}(o, B_{\rho_i}(x_i)\setminus B_{\nu_i}(x_i))$, then this means:
\begin{align}\label{Datamaximizing}
	\begin{aligned}
		&\frac{1}{2}[\tilde \omega_i]'_{\alpha,B_{2(\epsilon_i)^{-1}}(o),\Phi_{\epsilon_i}^*(f_{i^{-1}}),\Sigma_{3C}^2 \times \Sigma_{\loc}^2} \\
		\leq& \rho_i^{-\alpha}\Phi_{\epsilon_i}^*(f_{i^{-1}})(\delta_i) \min_{(\pi,\eta)\in \Sigma_{3C}^2\times \Sigma_\loc^2}  \mathds{1}_{B_{\rho_i}(x_i)\setminus B_{\nu_i}(x_i)}(\eta)\Norm{\tilde\omega_i-\pi - \eta}{0,(B_{\rho_i}(x_i)\setminus B_{\nu_i}(x_i))\cap B_{2(\epsilon_i)^{-1}}(o)}.
	\end{aligned}
\end{align}
A priori, it could be that $\Phi_{\epsilon_i}^*(f_{i^{-1}})(\delta_i)\to 0$ as $i\to\infty$. For convenience in the proof and as in Proposition \ref{SchauderLinear}, we redefine $\Phi_{\epsilon_i}^*(f_{i^{-1}})$  to $\tilde f_{i}$ such that this does not occur. This can happen in two places: if $\delta_i< (i\epsilon_i)^{-1}$, or if $\delta_i \in [\epsilon_i^{-2},2\epsilon_i^{-1}]$:\\

If $\delta_i \geq (i\epsilon_i)^{-1}$ already, then define $\tilde f_i \coloneqq \Phi_{\epsilon_i}^*(f_{i^{-1}})$; otherwise if $\delta_i < (i\epsilon_i)^{-1}$:
\begin{equation*}
	\tilde f_i(r) \coloneqq \begin{cases}
		(\delta_i)^{-\frac{1+\alpha}{2}}r^{\frac{1+\alpha}{2}}, & 0 \leq r \leq \delta_i,\\
		1, &  r \in (\delta_i, \epsilon_i^{-1}),\\
		2-\epsilon_ir, & r \in [\epsilon_i^{-1},2\epsilon_i^{-1}].
	\end{cases}
\end{equation*}
For the weight functions, we have the crucial property:
\begin{equation*}
	\frac{\tilde f_i(r)}{\tilde f_i(\delta_i)} \leq \frac{\Phi_{\epsilon_i}^*(f_{i^{-1}})(r)}{\Phi_{\epsilon_i}^*(f_{i^{-1}})(\delta_i)}
\end{equation*}
for any $r \in (0,2\epsilon_i^{-1})$. Let $(\pi_i,\eta_i) \in \Sigma_{3C}^2 \times \Sigma_{\loc}^2$ be the minimizers of \eqref{EqualToKConePrel} on $B_{\rho_i}(x_i) \setminus B_{\nu_i}(x_i)$. Thus, for any other $x_i'\in \CC,\rho_i',\nu_i'>0$ with $\delta_i' \coloneqq \dist_{\omega_\CC}(o, B_{\rho_i'}(x_i')\setminus B_{\nu_i}(x_i'))$, then 
\begin{align*}
	&\tilde f_i (\delta_i) \rho_i^{-\alpha} \Norm{\tilde\omega_i-\pi_i - \eta_i}{0,(B_{\rho_i}(x_i)\setminus B_{\nu_i}(x_i))\cap B_{2(\epsilon_i)^{-1}}(o)}\\ 
	=& \tilde f_i (\delta_i') \frac{\tilde f_i (\delta_i)}{\tilde f_i (\delta_i')} \rho_i^{-\alpha} \Norm{\tilde\omega_i-\pi_i - \eta_i}{0,(B_{\rho_i}(x_i)\setminus B_{\nu_i}(x_i))\cap B_{2(\epsilon_i)^{-1}}(o)}\\
	\geq& \frac{1}{2} \tilde f_i (\delta_i') \frac{\tilde f_i (\delta_i)}{\tilde f_i (\delta_i')} \frac{\Phi_{\epsilon_i}^* f_{i^{-1}} (\delta_i')}{\Phi_{\epsilon_i}^* f_{i^{-1}} (\delta_i)} 
	(\rho_i')^{-\alpha} \min_{(\pi,\eta)\in \Sigma_{3C}^2\times \Sigma_\loc^2}\Norm{\tilde\omega_i-\pi - \eta}{0,(B_{\rho_i'}(x_i')\setminus B_{\nu_i'}(x_i'))\cap B_{2(\epsilon_i)^{-1}}(o)}\\
	\geq&\frac{1}{2} \tilde f_i (\delta_i')  (\rho_i')^{-\alpha} \min_{(\pi,\eta)\in \Sigma_{3C}^2\times \Sigma_\loc^2}\Norm{\tilde\omega_i-\pi - \eta}{0,(B_{\rho_i'}(x_i')\setminus B_{\nu_i'}(x_i'))\cap B_{2(\epsilon_i)^{-1}}(o)},
\end{align*}
with the penultimate inequality following from \eqref{Datamaximizing} and $x_i,\rho_i,\nu_i, \pi_i, \eta_i$ realizing at least half of the seminorm. Thus, $x_i,\rho_i,\nu_i, \pi_i, \eta_i$ also realizes at least half of $[\tilde\omega_i]_{\alpha,B_{2(\epsilon_i)^{-1}}(o),\tilde f_i,\Sigma_{3C}^2 \times \Sigma_{\loc}^2}'$:
\begin{equation*}
	\frac{1}{2} [\tilde\omega_i]_{\alpha,B_{2(\epsilon_i)^{-1}}(o),\tilde f_i,\Sigma_{3C}^2 \times \Sigma_{\loc}^2}' \leq \rho_i^{-\alpha} \tilde f_i(\delta_i) \Norm{\tilde\omega_i-\pi_i - \eta_i}{0,(B_{\rho_i}(x_i)\setminus B_{\nu_i}(x_i))\cap B_{2(\epsilon_i)^{-1}}(o)}.
\end{equation*}
\sloppy For the case $\delta_i \in [\epsilon_i^{-1}, 2\epsilon_i^{-1}]$, then as $\Norm{\tilde \omega_i}{0, B_{3(\epsilon_i)^{-1}}(o)} \leq C_1$ and the scalar curvatures scales by:
\begin{equation*}
	\Norm{\Scal(\tilde\omega_i)}{0, B_{3(\epsilon_i)^{-1}(o)}}\leq D\epsilon_i^{2},
\end{equation*}
regularity of the complex Monge-Ampère equation and Lemma \ref{LemmaHolderBoundToPrimeBound} implies that
\begin{align*}
	[\tilde\omega_i]_{\alpha,B_{2(\epsilon_i)^{-1}}(o) \setminus B_1(o),1,\Sigma_{3C}^2 \times \Sigma_{\loc}^2}' &\leq 
	\Norm{\tilde\omega_i}{0,\alpha,B_{2(\epsilon_i)^{-1}}(o) \setminus B_1(o),1,\Sigma_{3C}^2 \times \Sigma_{\loc}^2}'\\
	&\leq \Norm{\tilde\omega_i - \omega_\CC}{0,\alpha,B_{2(\epsilon_i)^{-1}}(o) \setminus B_1(o),1,\Sigma_{\loc}^2}'\\
	&\leq \Norm{\tilde\omega_i}{0,\alpha,B_{2(\epsilon_i)^{-1}}(o) \setminus B_1(o),1,\Sigma_{\loc}^2}' + \Norm{\omega_\CC}{0,\alpha,B_{2(\epsilon_i)^{-1}}(o) \setminus B_1(o),1,\Sigma_{\loc}^2}'\\
	&\leq \Norm{\tilde\omega_i}{0,\alpha,B_{2(\epsilon_i)^{-1}}(o) \setminus B_1(o)} + \Norm{\omega_\CC}{0,\alpha,B_{2(\epsilon_i)^{-1}}(o) \setminus B_1(o)}\\
	 &\leq \frac{1}{c} 
\end{align*}
for some constant $c>0$ and weight function identically $1$. Therefore,
\begin{equation*}
	\tilde f_i(\delta_i) \geq c,
\end{equation*}
i.e. $\tilde f_i(\delta_i)$ is bounded from below. This also further shows that $\delta_i \leq \frac{2-c}{\epsilon_i}$. So if $\rho_i$ stays bounded, the distance $\dist_{\omega_\CC}(\partial(B_{2(\epsilon_i)^{-1}}(o)), B_{\rho_i}(x_i)\setminus B_{\nu_i}(x_i))\to \infty$ as $i\to\infty$, i.e. our maximizing ball becomes infinitely far away from the outer boundary in the limit.\\

Finally, as $\tilde f_i \geq \Phi_{\epsilon_i}^*(f_{i^{-1}})$, then
\begin{equation*}
	[\tilde\omega_i]_{\alpha,B_{2(\epsilon_i)^{-1}}(o),\tilde f_i,\Sigma_{3C}^2 \times \Sigma_{\loc}^2}' \geq 	[\tilde \omega_i]'_{\alpha,B_{2(\epsilon_i)^{-1}}(o),\Phi_{\epsilon_i}^*(f_{i^{-1}}),\Sigma_{3C}^2 \times \Sigma_{\loc}^2} = 1.
\end{equation*}
If necessary and without changing notation, make $\epsilon_i\to 0$ even faster such that
\begin{equation}\label{EqualToKCone}
	[\tilde\omega_i]_{\alpha,B_{2(\epsilon_i)^{-1}}(o),\tilde f_i,\Sigma_{3C}^2 \times \Sigma_{\loc}^2}' = 1
\end{equation}
for any $i \in \N$. This rescaling preserves the discussion above, in particular the lower bound $\tilde f_i(\delta_i) \geq c$.

\subsubsection{Setting up Contradiction and Obtaining Limiting Function and Metric}

Due to Lemma \ref{LemmaScalarCurvControlSublimit}, there exists a sublimit $\tilde\omega_i \to\omega_\infty = \Psi^*\omega_\CC$ in $C^{1,\beta}_\loc(\CC)$ for all $\beta <1$ and for some $\Psi\in \Aut_\Scl(\CC)$. Using this, we prove the following claim:\\

\textbf{Claim:} The sequence $(\tilde \omega_i)$ satisfies
\begin{equation*}
	[\tilde\omega_i]_{\alpha,B_{2(\epsilon_i)^{-1}}(o),\tilde f_i,\Sigma_{3C}^2 \times \Sigma_{\loc}^2}' \to 0, \quad i \to \infty.
\end{equation*}
This is obviously a contradiction with \eqref{EqualToKCone}, thereby proving that 
\begin{equation*}
	[\omega_i]'_{\alpha,B_2(o), f_{i^{-1}},\Sigma_{3C}^2 \times \Sigma_{\loc}^2} \leq C'
\end{equation*}
for some constant $C'$ and all $i > 0$ if the claim is true, and hence proving Theorem \ref{TheoremHolderBound}.\\

\textbf{Proof:} We prove the claim by studying three cases:\\

\subsubsection{\textbf{Case 1:} $\rho_i \to \infty$ as $i \to \infty$} $\Norm{\tilde\omega_i}{0,\CC} \leq C$ and $\omega_\CC \in \Sigma_{3C}^2$ implies that
\begin{equation*}
	\tilde f_i(\delta_i) \rho_i^{-\alpha} \min_{\pi\in \Sigma_{3C}^2}\Norm{\tilde\omega_i - \pi}{0,B_{\rho_i}(x_i)} \leq (C+1) \rho_i^{-\alpha} \to 0, \quad i \to \infty,
\end{equation*}
showing the claim if $\rho_i\to\infty$.
\vspace{2mm}

\subsubsection{\textbf{Case 2:} $\rho_i$ is bounded away from $0$ and $\infty$, i.e., there exist $b, B > 0$ such that $b < \rho_i < B$.}\label{SectionCase2}

\begin{enumerate}[\text{Subcase 2.}1)]
	\item $0 < a < \delta_i < A < \infty$ for some $a,A>0$: The convergence $\tilde\omega_i \to \omega_\infty \in \Sigma_{3C}^2$ in $C^{1,\beta}(K)$ for a compact set $K$ containing all $B_{\rho_i}(x_i) \setminus B_{\nu_i}(x_i)$ implies that
	\begin{equation*}
		\tilde f_i(\delta_i) \rho_i^{-\alpha} \min_{\pi\in \Sigma_{3C}^2}\Norm{\tilde\omega_i - \pi}{0, B_{\rho_i}(x_i) \setminus B_{\nu_i}(x_i)} \leq b^{-\alpha} \Norm{\tilde\omega_i - \omega_\infty}{0, K} \to 0.
	\end{equation*}
	
	\item $\delta_i \to 0$: In this case, $B_{\rho_i}(x_i) \setminus B_{\nu_i}(x_i)$ leaves every compact set in $\CC$, and so we cannot use the convergence $\tilde\omega_i \to \omega_\infty$ in $C^{1,\beta}_\loc(\CC)$ directly. For this part, we assume that $\rho_i = 1$ at all times to ease notation. As $B_1(x_i)$ is not contained in a coordinate ball for $i$ large enough, no element in $\Sigma_{\loc}^2$ other than 0 appear. Note that $\tilde f_i(\delta_i) = 1$, but this does not play a role here.
	
	As $\tilde\omega_i \to \omega_\infty \in \Sigma_{3C}^2$ in $C_\loc^{1,\beta}(\CC)$, then for any small but fixed $\mu > 0$ and $i$ large enough:
	\begin{equation}\label{EqOmega_iConvMu}
		\Norm{\tilde\omega_i - \omega_\infty}{0, [\mu, 2] \times L} \leq \mu.
	\end{equation}
	By the Hölder bound \eqref{EqualToKCone} and choosing $x = o$, $\rho = 2\mu$, and $\nu = \delta_i$ in the seminorm \eqref{EqualToKCone}, there exists $\psi_i \in \Sigma_{3C}^2$ such that
	\begin{equation*}
		\tilde f_i(\delta_i)\Norm{\tilde\omega_i - \psi_i}{0, [\delta_i, 2\mu] \times L} \leq (2\mu)^\alpha.
	\end{equation*}
	The triangle inequality now implies that
	\begin{align*}
		\tilde f_i(\delta_i)\Norm{\omega_\infty - \psi_i}{0, [\mu, 2\mu] \times L} &= \tilde f_i(\delta_i)\Norm{\omega_\infty - \tilde\omega_i + \tilde\omega_i - \psi_i}{0, [\mu, 2\mu] \times L}\\ 
		&\leq \tilde f_i(\delta_i)\Norm{\omega_\infty - \tilde\omega_i }{0, [\mu, 2\mu] \times L} + \tilde f_i(\delta_i)\Norm{\tilde\omega_i - \psi_i}{0, [\mu, 2\mu] \times L}\\ 
		&\leq \mu + (2\mu)^\alpha \\
		&\leq 3\mu^\alpha.
	\end{align*}
	$|\omega_{\infty} - \psi_i|_{\omega_\CC}$ is invariant under scaling, so it follows that
	\begin{equation*}
		\tilde f_i(\delta_i)\Norm{\omega_\infty - \psi_i}{0, (0, 2] \times L} \leq 3\mu^\alpha.
	\end{equation*}
	Hence,
	\begin{align*}
		\tilde f_i(\delta_i) \min_{\pi\in \Sigma_{3C}^2}\Norm{\tilde\omega_i -\pi}{0, B_{\rho_i}(x_i) \setminus B_{\nu_i}(x_i)} &\leq \tilde f_i(\delta_i)\Norm{\tilde\omega_i - \omega_\infty}{0, [\delta_i, 2] \times L}\\
		&\leq \tilde f_i(\delta_i)\left(\Norm{\tilde\omega_i - \omega_\infty}{0, [\delta_i, 2\mu] \times L} + \Norm{\tilde\omega_i - \omega_\infty}{0, [2\mu, 2] \times L}\right)\\
		&\leq \tilde f_i(\delta_i)\Norm{\tilde\omega_i - \psi_i + \psi_i - \omega_\infty}{0, [\delta_i, 2\mu] \times L}+ \mu\\ 
		&\leq \tilde f_i(\delta_i)\Norm{\tilde\omega_i - \psi_i}{0, [\delta_i, 2\mu] \times L} + \tilde f_i(\delta_i)\Norm{\psi_i - \omega_\infty}{0, [\delta_i, 2\mu] \times L}+ \mu\\
		&\leq (2\mu)^\alpha + 3\mu^\alpha + \mu,
	\end{align*}
	proving the claim in this case as $x_i,\rho_i,\nu_i$ realizes at least half of the seminorm \eqref{Datamaximizing}.
	\item $\delta_i\to\infty$. As $\dist_{\omega_{\CC}}(x_i,o)\to\infty$, and as the coordinate balls in $\UU$ have increasing size moving away from the apex, $B_{\rho_i}(x_i)$ is always contained in a coordinate ball for $i$ sufficiently large. Fix the associated holomorphic coordinates on any such coordinate ball, and apply translations such that $x_i=0$ is always the center. In these coordinates, $\Sigma_{\loc}^2$ is by definition the set of constant $(1,1)$-forms. Recall that $\dist(x_i, \partial (B_{2(\epsilon_i)^{-1}}(o))) \to \infty$, so for every $i$, choose an increasing sequence of balls $B_{r_i}(x_i) \subset B_{2(\epsilon_i)^{-1}}(o)$ containing $B_{\rho_i}(x_i)$ such that $r_i \to \infty$ and $\dist_{\omega_{{\CC}}}(B_{r_i}(x_i),\partial( B_{2(\epsilon_i)^{-1}}(o)))\to\infty$. Also ensure $r_i\to\infty$ so slowly such that $B_{r_i}(x_i)$ is always contained in a coordinate ball.
	
	By the Liouville Theorem (Theorem \ref{LiouvilleTheorem}), $\omega_\CC|_{B_{r_i}(x_i)} \to A^*\omega_\eucl$ in $C_\loc^\infty(\C^m)$ and $\tilde\omega_i|_{B_{r_i}(x_i)} \to B^*\omega_\eucl$ in $C_\loc^{1,\beta}(\C^m)$ for some $A,B \in \GL(m,\C)$. Picking $\pi = \omega_\CC$ and $\eta = -A^*\omega_\eucl + B^*\omega_\eucl \in \Sigma_{\loc}^2$, then
	\begin{equation*}
		\tilde f_i(\delta_i) \rho_i^{-\alpha} \min_{(\pi,\eta)\in \Sigma_{3C}^2\times \Sigma_\loc^2}\Norm{\tilde\omega_i - \pi - \eta}{0, B_{\rho_i}(x_i)} \leq b^{-\alpha} \Norm{\tilde\omega_i - \omega_\CC + A^*\omega_\eucl - B^*\omega_\eucl}{0, B_{\rho_i}(x_i)} \to 0
	\end{equation*}
	as $i\to\infty$, and the result follows.
\end{enumerate}

\subsubsection{\textbf{Case 3:} $\rho_i \to 0$. Investigating the relationship between $\delta_i$ and $\rho_i$} In this case, we linearize the complex Monge-Ampère equation and apply Proposition \ref{SchauderLinear}. This part of the proof also differs the most from the proof of Proposition \ref{SchauderLinear}. Since we cannot directly cite a Liouville theorem as in Case 2, we first need to build up the necessary estimates. \\

The first thing that needs to be checked is the relationship between $\delta_i$ and $\rho_i$. Due to the regularity of the complex Monge-Ampère equation, we show that, under the condition $\rho_i \to 0$, then
\begin{equation}\label{RhoQuotient}
	\delta_i \to 0, \quad \text{and} \quad \frac{\delta_i}{\rho_i} \leq A < \infty,
\end{equation}
as $i\to\infty$ and for some $A>0$. $\delta_i$ therefore goes to zero, and goes to zero at least as fast as $\rho_i$. To show this, we check two cases:\\

\begin{enumerate}
	\item $b \leq \delta_i$: The $C^{1,\beta}_\loc(\CC)$-regularity of $\tilde\omega_i$ trivially implies a $C^{0,\beta}_\loc(\CC)$-bound on $\tilde\omega_i$ for any $\beta < 1$. As $\delta_i\geq b$ by assumption, this estimate is uniform on $B_{\frac{b}{2}}( x_i)$.	Lemma \ref{LemmaHolderBoundToPrimeBound} shows a similar bound for the primed seminorm:
	\begin{equation*}
		[\tilde\omega_i]_{\beta, B_{\frac{b}{2}}(\hat x_i), \Sigma_{\loc}^2}' \leq D.
	\end{equation*}
	The same is true for any element in $\Sigma_{3C}^2$, especially $\omega_\CC$. Pick $\beta > \alpha$. As $B_{\rho_i}(x_i)\subset B_{\frac{b}{2}}( x_i)$ for $i$ big enough, it follows that:
	\begin{equation}\label{EqHigherBetaContradiction}
		\begin{aligned}
			\frac{1}{2} &\leq \rho_i^{-\alpha} \min_{(\pi,\eta)\in \Sigma_{3C}^2\times \Sigma_\loc^2} \Norm{\tilde \omega_i - \pi - \eta}{0,B_{\rho_i}(x_i)} \\
			&\leq \rho_i^{\beta-\alpha} \rho_i^{-\beta}\min_{\eta \in \Sigma_\loc^2} \Norm{\tilde \omega_i - \omega_\CC - \eta}{0,B_{\rho_i}(x_i)} \\
			&\leq \rho_i^{\beta-\alpha} D_K \to 0,
		\end{aligned}
	\end{equation}
	which leads to a contradiction.
	
	\item $\delta_i \to 0$, but $\frac{\rho_i}{\delta_i} \to 0$: Blow up by $\delta_i^{-1}$ such that the maximizing ball of the seminorm has distance 1 to the apex:
	\begin{itemize}
		\item $x_i^\sharp \coloneqq \Phi_{\delta_i}^{-1}(x_i)$,
		\item $\delta_i^{-2}\Phi_{\delta_i}^*(\omega_\CC) = \omega_\CC$,
		\item $ \omega_i^\sharp \coloneqq \delta_i^{-2}\Phi_{\delta_i}^*(\tilde\omega_i)$,
		\item $ \rho_i^\sharp \coloneqq \frac{\rho_i}{\delta_i}$, $\nu_i^\sharp \coloneqq \frac{\nu_i}{\delta_i} < 1$, $ \delta_i^\sharp \coloneqq \frac{\delta_i}{\delta_i} = 1$,
		\item $\Phi_{\delta_i}^*(\tilde f_i) \coloneqq f_i^\sharp$.
	\end{itemize}
	By assumption, $\rho_i^\sharp \to 0$ as $i \to \infty$, while $1 = \delta_i^\sharp = \dist_{\omega_\CC}(o, B_{\rho_i^\sharp}(x_i^\sharp))$. The blowup by $\delta_i^{-1}$ also scales the seminorm \eqref{EqualToKCone}:
	\begin{equation}\label{SeminormDeltaBlowup}
		[\omega_i^\sharp]_{\alpha,B_{2(\delta_i\epsilon_i)^{-1}}(o), f_i^\sharp,\Sigma_{3C}^2 \times \Sigma_{\loc}^2}' = \delta_i^{\alpha},
	\end{equation}
	with the ball $B_{\rho_i^\sharp}(x_i^\sharp)$ realizing at least half of the seminorm in \eqref{SeminormDeltaBlowup}. Let $\pi_i^\sharp \in \Sigma_{3C}^2$ be the minimizing element in \eqref{SeminormDeltaBlowup} on $B_{3}(o)\setminus B_{\frac{1}{4}}(o)$. As $B_3(o)\setminus B_{\frac{1}{4}}(o)$ is not contained in any coordinate ball $U$, there is no contribution from $\Sigma_{\loc}^2$. \eqref{SeminormDeltaBlowup} shows that:
	\begin{equation}\label{DifferenceOmegaPiSharpLinfty}
		\Norm{\omega_i^\sharp - \pi_i^\sharp}{0,B_3(o)\setminus B_{\frac{1}{4}}(o) } \leq 3^\alpha f_i^\sharp\left(\frac{1}{4}\right)^{-1} \delta_i^\alpha,
	\end{equation}
	where $3^\alpha f_i^\sharp(\frac{1}{4})^{-1} \leq C < \infty$ uniformly by the condition $f_i^\sharp(1) = 1$. By Lemma \ref{LemmaHolderBound} and Lemma \ref{LemmaHolderBoundToPrimeBound}, then
	\begin{align*}
		\Norm{\omega_i^\sharp - \pi_i^\sharp}{0,\beta,B_2(o)\setminus B_{\frac{1}{2}}(o), \Sigma_\loc^2 }' &\leq C'\Norm{\omega_i^\sharp - \pi_i^\sharp}{0,\beta,B_2(o)\setminus B_{\frac{1}{2}}(o) } \\
		&\leq C'\Norm{\omega_i^\sharp - \pi_i^\sharp}{0,B_3(o)\setminus B_{\frac{1}{4}}(o) } \\
		&\leq C' \delta_i^\alpha
	\end{align*}
	for all $\beta < 1$ and $C'$ depending on $\beta$.
	
	As $\rho_i^\sharp \to 0$, all $B_{\rho_i^\sharp}(x_i^\sharp)$ are contained in $B_2(o)\setminus B_{\frac{1}{2}}(o)$ for $i$ sufficiently large. Hence, on $B_2(o)\setminus B_{\frac{1}{2}}(o)$, we proceed as in \eqref{EqHigherBetaContradiction}:
	\begin{align*}
		\frac{\delta_i^\alpha}{2} &\leq (\rho_i^\sharp)^{-\alpha} \min_{(\pi,\eta)\in \Sigma_{3C}^2\times \Sigma_\loc^2}\Norm{\omega_i^\sharp - \pi - \eta}{0, B_{\rho_i^\sharp}(x_i^\sharp)} \\
		&\leq (\rho_i^\sharp)^{-\alpha} \min_{\eta\in  \Sigma_\loc^2}\Norm{\omega_i^\sharp - \pi_i^\sharp - \eta}{0, B_{\rho_i^\sharp}(x_i^\sharp)} \\
		&=  (\rho_i^\sharp)^{\beta-\alpha}(\rho_i^\sharp)^{-\beta} \min_{\eta\in  \Sigma_\loc^2}\Norm{\omega_i^\sharp - \pi_i^\sharp - \eta}{0, B_{\rho_i^\sharp}(x_i^\sharp)}\\
		&\leq (\rho_i^\sharp)^{\beta-\alpha} C'\delta_i^\alpha.
	\end{align*}
	As $\alpha < 1$ and if we pick $\beta > \alpha$, then:
	\begin{equation*}
		\frac{1}{2} \leq C' (\rho_i^\sharp)^{\beta-\alpha} \to 0,
	\end{equation*}
	a clear contradiction.
\end{enumerate}

From the above three cases, \eqref{RhoQuotient} follows.

\subsubsection{\textbf{Case 3:} $\rho_i \to 0$. Linearizing the Complex Monge-Ampère Equation and Proof}
First, we rescale $\tilde\omega_i$ such that $\rho_i$ equals 1 in the rescaled coordinates, denoted by $\hat\cdot$. More precisely, define:
\begin{itemize}
	\item $\hat x_i \coloneqq \Phi_{\rho_i}^{-1}(x_i)$,
	\item $\rho_i^{-2}\Phi_{\rho_i}^*(\omega_\CC) = \omega_\CC$,
	\item $\hat \omega_i \coloneqq \rho_i^{-2}\Phi_{\rho_i}^*(\tilde\omega_i)$,
	\item $\hat \rho_i \coloneqq \frac{\rho_i}{\rho_i} = 1$, $\hat \nu_i \coloneqq \frac{\nu_i}{\rho_i} < 1$, $\hat \delta_i \coloneqq \frac{\delta_i}{\rho_i} \leq A$,
	\item $\hat f_i \coloneqq \Phi_{\rho_i}^*(\tilde f_i)$.
\end{itemize}
$\omega_\CC$ is invariant under this rescaling as it is a cone metric. The rescaling has the effect of decreasing the ${C^{0,\alpha}}'$-seminorm: 
\begin{equation}\label{SeminormDoubleBlowup}
	[\hat\omega_i]_{\alpha, B_{2(\epsilon_i\rho_i)^{-1}}(o),\hat f_i, \Sigma_{3C}^2 \times \Sigma_{\loc}^2}' = \rho_i^\alpha [\tilde{\omega}_i]_{\alpha, B_{2(\epsilon_i)^{-1}}(o),\tilde f_i, \Sigma_{3C}^2 \times \Sigma_{\loc}^2}' = \rho_i^\alpha.
\end{equation}
We know that $\hat \omega_i \to \hat \omega_\infty$ in $C_\loc^{1,\beta}(\CC)$ for some $\hat \omega_\infty = \hat \Psi^*\omega_\CC$, but the rate of convergence is not known. Notice that by \eqref{RhoQuotient}, then $\hat \delta_i\leq A<\infty$. The proof proceeds by modifying $\hat \omega_i$ to $\check\omega_i$ with a nonvanishing primed seminorm. First, we subtract an element $\pi_i \in \Sigma_{3C}^2$ such that we obtain uniform $C^{0,\beta}_{\loc}$-estimates of $\rho_i^{-\alpha}(\hat \omega_i - \pi_i)$ for any $\beta < 1$. For this, pick $\pi_i^{R,\nu} \in \Sigma_{3C}^2$ realizing \eqref{SeminormDoubleBlowup} on $B_R(o)\setminus B_\nu(o)$, i.e.,
\begin{equation*}
	\Norm{\hat \omega_i - \pi_i^{R,\nu}}{0, B_R(o)\setminus B_\nu(o)} \leq R^{\alpha}\hat f_i(\nu)^{-1} \rho_i^\alpha,
\end{equation*}
where $\hat f_i(\nu)^{-1}$ has an upper bound because $\hat\delta_i \leq A$ uniformly. Set $\pi_i \coloneqq \pi_i^{2,1}$. Then \eqref{SeminormDoubleBlowup} implies:
\begin{equation}\label{EqInftyNormOmegaPi}
	\begin{aligned}
		\hat f_i(\nu)\Norm{\hat \omega_i - \pi_i}{B_R(o)\setminus B_{\nu}(o)} &\leq \hat f_i(\nu)\Norm{\hat \omega_i - \pi_i^{R,\nu}}{B_R(o)\setminus B_{\nu}(o)} + \hat f_i(\nu)\Norm{\pi_i^{R,\nu} - \pi_i}{B_R(o)\setminus B_{\nu}(o)}\\
		&\leq R^\alpha\rho_i^\alpha + \hat f_i(\nu)\Norm{\pi_i^{R,\nu} - \pi_i}{B_2(o)\setminus B_{1}(o)} \\
		&\leq R^\alpha\rho_i^\alpha + \hat f_i(\nu)\Norm{\pi_i^{R,\nu} - \hat\omega_i}{B_2(o)\setminus B_{1}(o)} + \hat f_i(\nu)\Norm{\hat\omega_i - \pi_i}{B_2(o)\setminus B_{1}(o)}\\
		&\leq (2R^\alpha + 2)\rho_i^\alpha.
	\end{aligned}
\end{equation}
Notice that $\Norm{\pi_i^{R,\nu} - \pi_i}{B_R(o)\setminus B_{\nu}(o)} = \Norm{\pi_i^{R,\nu} - \pi_i}{B_2(o)\setminus B_{1}(o)}$ as $\pi_i^{R,\nu}, \pi_i\in \Sigma_{3C}^2$ are invariant under scaling. Hence, by Lemma \ref{LemmaHolderBound}, on any compact set $K\Subset \CC$ and for any $\beta < 1$, there exists a $D_K>0$ such that 
\begin{equation}\label{HatOmegaHolderBound}
	\Norm{\hat \omega_i - \pi_i}{0,\beta,K} \leq D_K \rho_i^\alpha.
\end{equation}
Given this, define $\check \omega_i \coloneqq \rho_i^{-\alpha}(\hat \omega_i - \pi_i)$ with
\begin{equation}\label{EqCheckOmegaRegularity}
	[\check\omega_i]'_{\alpha, B_{2(\epsilon_i\rho_i)^{-1}}(o), \hat f_i, (\rho_i^{-\alpha}(\Sigma_{3C}^2 - \pi_i)) \times \Sigma_{\loc}^2} = 1.
\end{equation}
Notice the new comparison set incorporating $\pi_i$ and multiplication by $\rho_i^{-\alpha}$, defined as 
\begin{equation*}
	\rho_i^{-\alpha}(\Sigma_{3C}^2 - \pi_i) \coloneqq \{\rho_i^{-\alpha}(\Psi^*\omega_\CC - \pi_i) \in \Omega^2(\CC) \mid \Psi^*\omega_\CC \in \Sigma_{3C}^2 \}.
\end{equation*}
As $\hat f_i(\delta_i) = 1$ and $\hat \delta_i < A<\infty$ for some $A>0$, the decay condition of $\hat f_i$ and \eqref{EqInftyNormOmegaPi} show that
\begin{equation}\label{EqUniformEstimateOmegaPi}
	|\check \omega_i| \leq C \begin{cases}
		r^{-\frac{1+\alpha}{2}} & r \leq 1, \\ r^{\alpha} & r > 1.
	\end{cases}
\end{equation}
If $\hat \delta_i \to 0$, then $r^{-\frac{1+\alpha}{2}}$ gets replaced by the constant function $1$ in the limit. \eqref{HatOmegaHolderBound} shows that $\check\omega_i\in C^{0,\beta}(K)$ for any compact set $K\Subset \CC$.\\

To obtain a differential equation for $\check\omega_i$, rewrite $\hat \omega_i^m = \frac{\hat\omega_i^m}{\pi_i^m}\pi_i^m$ as
\begin{equation}\label{EqEtaBounded}
	\begin{aligned}
		\frac{\hat\omega_i^m}{\pi_i^m}\pi_i^m &= \hat\omega_i^m = \rho_i^{\alpha \cdot m} (\rho_i^{-\alpha}(\hat\omega_i - \pi_i) + \rho_i^{-\alpha}\pi_i)^m \\
		&= \rho_i^{\alpha \cdot m} (\check\omega_i + \rho_i^{-\alpha}\pi_i)^m \\
		&= \pi_i^m + m\rho_i^\alpha\pi_i^{m-1}\wedge \check\omega_i + \OO(\rho_i^{2\alpha}).
	\end{aligned}
\end{equation}
Subtracting $\pi_i^m$ and dividing by $\rho_i^\alpha\pi_i^m$ implies that
\begin{equation}\label{EqTrEta2}
	\begin{aligned}
		m\tr_{\pi_i}\check\omega_i + \OO(\rho_i^\alpha)
		= \rho_i^{-\alpha}\left(\frac{\hat\omega_i^m}{\pi_i^m}-1\right).
	\end{aligned}
\end{equation}
On the left-hand side, $\tr_{\pi_i}\check\omega_i$ is controlled in $C^{0,\beta}_\loc(\CC)$, and all terms in $\OO(\rho_i^\alpha)$ converge to $0$ in $C^{0,\beta}_\loc(\CC)$ as $\check\omega_i$ is controlled in $C^{0,\beta}_\loc(\CC)$ for $i$ big enough due to \eqref{HatOmegaHolderBound}. On the right-hand side, write
\begin{equation*}
	\omega_i^m = e^{G_i}\pi_i^m, \quad \hat\omega_i^m = e^{\hat G_i}\pi_i^m.
\end{equation*}
\eqref{EqTrEta2} becomes:
\begin{equation}\label{EqTrEta}
	m\tr_{\pi_i} \check\omega_i = \rho_i^{-\alpha}\left(e^{\hat G_i}-1\right) + \OO(\rho_i^\alpha).
\end{equation}
The blowups by $(\epsilon_i)$ and $(\rho_i)$ imply, by the uniform bound \eqref{ConditionUniformBound}, Ricci-flatness of $\pi_i$, and Proposition \ref{SchauderLinear}, that
\begin{equation}\label{EqRescaledG}
	[e^{\hat G_i}]_{0,\beta, B_{2(\epsilon_i\rho_i)^{-1}}(o)} \leq C_3 \epsilon_i^\beta\rho_i^\beta.
\end{equation}
As the right-hand side of \eqref{EqTrEta2} is bounded in $C^{0,\beta}_\loc(\CC)$, \eqref{EqRescaledG} shows that the right-hand side of \eqref{EqTrEta2} converges to a constant in $C^{0,\beta}_\loc(\CC)$ as $i \to \infty$. 

Taking $i\to\infty$, we obtain a limit $\check\omega_\infty$ in $C^{0,\beta}_{\loc}(\CC)$, satisfying
\begin{equation}\label{LaplacePhiInfty}
	\tr_{\hat\omega_\infty} \check\omega_\infty =  \text{constant},
\end{equation}
and with $\pi_i\to\hat \omega_\infty$ in $C^\infty_\loc(\CC)$. $\check\omega_\infty$ also satisfies the growth condition \eqref{EqUniformEstimateOmegaPi}, so as $-\frac{1+\alpha}{2}>-1$ and $\alpha$ is very small, Proposition \ref{PropGrowthRateToHomogeneousForm} shows that $\LL_{r\partial_r}\check\omega_\infty = 2\check\omega_\infty$ and $\check\omega_\infty = i\del\delbar (cr^2 + u)$ for some constant $c$ and 2-homogeneous, $\xi$-invariant, harmonic $u$.\\

To finally show the claim, we consider the following two cases: \\

\begin{enumerate}[\text{Subcase 3.}1)]
	\item $\hat \delta_i>a>0$ for some $a > 0$. Then:
	\begin{equation}\label{EqualityBlowupDistance1}
		\frac{1}{2} \leq \min_{\theta \in \rho_i^{-\alpha}(\Sigma_{3C}^2 - \pi_i) }\hat f_i(\hat \delta_i)\Norm{\check\omega_i - \theta}{0,B_{1}(\hat x_i)\setminus B_{\hat \nu_i}(\hat x_i)} \leq \min_{\theta \in \rho_i^{-\alpha}(\Sigma_{3C}^2 - \pi_i)} \Norm{\check\omega_i - \theta}{0,K},
	\end{equation}
	for some compact set $K\Subset\CC$ containing all $B_1(\hat x_i)\setminus B_{\hat \nu_i}(\hat x_i)$. Furthermore,
	\begin{equation*}
		\Norm{\check\omega_i - \check\omega_\infty}{0,K} \to 0, \quad i \to \infty,
	\end{equation*}
	where, by Corollary \ref{CorollaryAutomorphismDerivative}, for any $\mu > 0$ and $i$ big enough, there is an element $\theta_i \in \rho_i^{-\alpha}(\Sigma_{3C}^2 - \pi_i)$ such that
	\begin{equation*}
		\Norm{\check\omega_\infty - \theta_i}{0,K} \leq \mu,
	\end{equation*}
	showing the claim in this case by the triangle inequality.
	
	\item $\hat \delta_i \to 0$: We use the above arguments combined with the Hölder bound of \eqref{EqCheckOmegaRegularity} to extend the proof down to $o$ as in subcase 2.2. The proof is almost the same. As before, $B_{1}(\hat x_i)\setminus B_{\hat \nu_i}(\hat x_i)$ leaves every compact set in $\CC$, and so we cannot use the convergence $\check\omega_i \to \check \omega_\infty$ in $C^{0,\beta}_\loc(\CC)$ directly.
	
	As $\check \omega_i \to \check\omega_\infty$ in $C_\loc^{0,\alpha}(\CC)$, then for any small but fixed $\mu > 0$ and $i$ big enough:
	\begin{equation}\label{EqOmega_iConvMu2}
		\Norm{\check\omega_i - \check\omega_\infty}{0,[\mu,2] \times L} \leq \mu.
	\end{equation}
	By the Hölder bound \eqref{EqCheckOmegaRegularity} and choosing the set $[\hat \delta_i,2\mu]\times L$ in the seminorm, there exists $\psi_i \in \rho_i^{-\alpha}(\Sigma_{3C}^2 - \pi_i)$ such that
	\begin{equation*}
		\hat f_i(\hat \delta_i)\Norm{\check\omega_i - \psi_i}{0,[\hat \delta_i,2\mu] \times L} \leq (2\mu)^\alpha.
	\end{equation*}
	The triangle inequality now implies that
	\begin{equation}\label{ClosenessPsiCheckOmega}
		\begin{aligned}
		\hat f_i(\hat \delta_i)\Norm{\check\omega_\infty - \psi_i}{0,[\mu,2\mu] \times L} &= \hat f_i(\hat \delta_i)\Norm{\check\omega_\infty - \check\omega_i + \check\omega_i -  \psi_i}{0,[\mu,2\mu] \times L}\\ 
		&\leq \hat f_i(\hat \delta_i)\Norm{\check\omega_\infty - \check\omega_i }{0,[\mu,2\mu] \times L} + \hat f_i(\hat \delta_i)\Norm{\check\omega_i -  \psi_i}{0,[\mu,2\mu] \times L}\\ 
		&\leq \mu + (2\mu)^\alpha \\
		&\leq 3\mu^\alpha.
	\end{aligned}
	\end{equation}
	$|\check \omega_{\infty} - \psi_i|_{\omega_\CC}$ is invariant under scaling, so it follows that
	\begin{equation*}
		\hat f_i(\hat \delta_i)\Norm{\check\omega_\infty - \psi_i}{0,(0,2] \times L} \leq 3\mu^\alpha.
	\end{equation*}
	Therefore,
	\begin{align*}
		\hat f_i(\hat \delta_i)\Norm{\check\omega_i - \psi_i}{0,B_1(\hat x_i)\setminus B_{\hat \nu_i}(\hat x_i) } &\leq \hat f_i(\hat \delta_i)\Norm{\check\omega_i - \psi_i}{0,[\hat \delta_i,2\mu] \times L} + \hat f_i(\hat \delta_i)\Norm{\check\omega_i - \psi_i}{0,[2\mu,2] \times L}\\
		&\leq(2\mu)^{\alpha} + \hat f_i(\hat \delta_i)\Norm{\check\omega_i -\check\omega_\infty + \check\omega_\infty - \psi_i}{0,[2\mu,2] \times L}\\
		&\leq(2\mu)^{\alpha} + \hat f_i(\hat \delta_i)\Norm{\check\omega_i -\check\omega_\infty}{0,[2\mu,2] \times L} + \hat f_i(\hat \delta_i)\Norm{ \check\omega_\infty - \psi_i}{0,[2\mu,2] \times L}\\
		&\leq (2\mu)^{\alpha} + \mu + 3\mu^\alpha
	\end{align*}
	showing the claim in this case by choosing $\mu>0$ small enough.
\end{enumerate}

\subsection{Asymptotics of the Kähler Metric at $o$}

The following corollary follows from Theorem \ref{TheoremHolderBound}:

\begin{corollary}[{Corollary \ref{mCorollaryOfCalpha}}]\label{CorollaryOfCalpha}
Let $(\CC,\omega_\CC)$ be a Calabi-Yau cone with cone metric $\omega_\CC$, and let $\omega$ be a Kähler metric on $B_3(o)\subset \CC$ such that
\begin{equation}\label{ConditionUniformBoundCor}
	\frac{1}{C}\omega_\CC \leq \omega\leq C\omega_\CC, \quad \Norm{\Scal(\omega)}{0,B_3(o)} \leq D,
\end{equation}
for some constants $C,D>0$. Then, for any $\alpha>0$ small enough, there exists an automorphism $\Psi\in \Aut_\Scl(\CC)$ with
\begin{equation*}
	|\Psi^*\omega- \omega_\CC|_{\omega_\CC} \leq C' r^{\alpha}
\end{equation*}
for $r \leq 1$. $\alpha = \alpha(\CC,\omega_\CC)$ and $C'= C'(\CC,C, D, \alpha,\omega_\CC, \UU)$ are independent of $\omega$.\\

If $\Scal(\omega) = 0$, then
\begin{equation*}
	|\nabla_{\omega_\CC}^k (\Psi^*\omega) - \omega_\CC |_{\omega_\CC} \leq C_k' r^{\alpha - k}, \qquad k \in \mathbb{N}_0,
\end{equation*}
for $r\leq 1$ and $C_k' = C_k'(\CC, C, D, k, \alpha, \omega_\CC, \UU)$
\end{corollary}

From Corollary \ref{CorollaryOfCalpha}, one immediately concludes that the tangent cone at $o$ of any $\omega$ satisfying the requirements of Theorem \ref{TheoremHolderBound} is unique.

\begin{proof}
By Theorem \ref{TheoremHolderBound}, there is a bound $[\omega]'_{\alpha, B_1(o), \Sigma_{3C}^2 \times \Sigma_{\loc}^2} \leq C'$ on the ${C^{0,\alpha}}'$-seminorm. The seminorm contains a supremum over all points $x \in \CCbar$, so fix $x = o$ as the apex. The only achievable elements in the comparison set lie in $\Sigma_{3C}^2$, so by Theorem \ref{TheoremHolderBound} and for all $0 < \rho \leq 1$, there exists $\pi_\rho \in \Sigma_{3C}^2$ such that:
\begin{equation}\label{EqRhoRegularity}
	\Norm{\omega - \pi_\rho}{0, B_\rho(o)} \leq C'\rho^\alpha.
\end{equation}
As $\omega$ is bounded, $(\pi_\rho)$ is also bounded. Define the blowup $\tilde\omega_{\rho} \coloneqq \rho^{-2}\Phi_{\rho}^*(\omega)$. $\pi_{\rho}$ is invariant under this rescaling as $\pi_{\rho} \in \Sigma_{3C}^2$. Thus, \eqref{EqRhoRegularity} becomes:
\begin{equation*}
	\Norm{\tilde\omega_{\rho} - \pi_{\rho}}{0, B_1(o)} \leq C'\rho^\alpha.
\end{equation*}
Proposition \ref{PropRegCMA} and the Liouville Theorem (Theorem \ref{LiouvilleTheorem}) shows that there exists a $C^{1,\beta}_{\loc}(\CC)$-sublimit $\tilde\omega_{\rho_i} \to \Psi^*\omega_\CC$ for some $\Psi \in \Aut_{\Scl}(\CC)$ and $\rho_i \to 0$. The triangle inequality yields:
\begin{equation*}
	\left|\Norm{\tilde\omega_{\rho_i} - \Psi^*\omega_\CC}{0, [\frac{1}{2}, 1] \times L} - \Norm{\Psi^*\omega_\CC - \pi_{\rho_i}}{0, [\frac{1}{2}, 1] \times L} \right| \leq \Norm{\tilde\omega_{\rho_i} - \pi_{\rho_i}}{0, [\frac{1}{2}, 1] \times L} \leq C'\rho_i^\alpha.
\end{equation*}
Since $\Norm{\tilde\omega_{\rho_i} - \Psi^*\omega_\CC}{0, [\frac{1}{2}, 1] \times L} \to 0$ as $i \to \infty$, we conclude that $\Norm{\Psi^*\omega_\CC - \pi_{\rho_i}}{0, [\frac{1}{2}, 1] \times L} = \Norm{\Psi^*\omega_\CC - \pi_{\rho_i}}{0, \CC} \to 0$ due to the $r\del_r$-invariance of $\Psi^*\omega_\CC$ and $\pi_{\rho_i}$.

\textbf{Claim:} $\Psi^*\omega_\CC$ satisfies:
\begin{equation}\label{BlowupMetricExistence}
	\Norm{\omega - \Psi^*\omega_\CC}{0, B_{\rho_i}(o)} \leq 4C' \rho_i^\alpha.
\end{equation}

\textbf{Proof:} For each $\rho > 0$, define the norm:
\begin{equation*}
	\Norm{\omega}{\rho} \coloneqq \rho^{-\alpha} \Norm{\omega}{0, B_{\rho}(o)}.
\end{equation*}
In this notation, \eqref{EqRhoRegularity} reads:
\begin{equation}\label{EqRhoRegularityAlternative}
	\Norm{\omega - \pi_{\rho_i}}{\rho_i} \leq C'.
\end{equation}
For fixed $k$, and in each norm $\Norm{\cdot}{\rho_k}$, we know that $\pi_{\rho_i} \to \Psi^*\omega_\CC$. Thus, for each $k$, there exists $i_k \geq k$ with $\rho_{i_k} \leq \rho_k$ such that:
\begin{equation*}
	\Norm{\Psi^*\omega_\CC - \pi_{\rho_{i_k}}}{\rho_k} \leq C'.
\end{equation*}
Therefore:
\begin{equation}\label{EqOmegaPi0Control}
	\begin{aligned}
		\Norm{\omega - \Psi^*\omega_\CC}{\rho_k} &\leq \Norm{\omega - \pi_{\rho_k}}{\rho_k} + \Norm{\pi_{\rho_k} - \pi_{\rho_{i_k}}}{\rho_k} + \Norm{\pi_{\rho_{i_k}} - \Psi^*\omega_\CC}{\rho_k}\\
		&\leq C' + \Norm{\pi_{\rho_k} - \pi_{\rho_{i_k}}}{\rho_k} + C'.
	\end{aligned}
\end{equation}
Hence, to prove the claim, we only need to control $\Norm{\pi_{\rho_k} - \pi_{\rho_{i_k}}}{\rho_k}$. As $\pi_{\rho_k}$ and $\pi_{\rho_{i_k}}$ are invariant under $r\del_r$, we can restrict the $C^0$-norm to $B_{\rho_{i_k}}(o) \subset B_{\rho_k}(o)$:
\begin{equation}\label{EqPiKSubsequenceControl}
	\begin{aligned}
		\Norm{\pi_{\rho_k} - \pi_{\rho_{i_k}}}{\rho_k} &= \left(\frac{\rho_k}{\rho_{i_k}} \right)^{-\alpha} \Norm{\pi_{\rho_k} - \pi_{\rho_{i_k}}}{\rho_{i_k}} \\
		&\leq \left(\frac{\rho_k}{\rho_{i_k}} \right)^{-\alpha} \left(\Norm{\omega - \pi_{\rho_k}}{\rho_{i_k}} + \Norm{\omega - \pi_{\rho_{i_k}}}{\rho_{i_k}} \right)\\
		&\leq \Norm{\omega - \pi_{\rho_k}}{\rho_{k}} + \left(\frac{\rho_k}{\rho_{i_k}} \right)^{-\alpha}\Norm{\omega - \pi_{\rho_{i_k}}}{\rho_{i_k}}\\
		&\leq 2C',
	\end{aligned}
\end{equation}
by \eqref{EqRhoRegularityAlternative} and as $\left(\frac{\rho_k}{\rho_{i_k}} \right)^{-\alpha} \leq 1$. Combining \eqref{EqOmegaPi0Control} and \eqref{EqPiKSubsequenceControl} proves the claim. $\square$

To show that \eqref{BlowupMetricExistence} holds for any $0 < \rho \leq 1$, take another subsequence $\tilde\rho_i \to 0$ with sublimit $\tilde \Psi^*\omega_\CC$ and consider:
\begin{align*}
	\Norm{\Psi^*\omega_\CC - \tilde\Psi^*\omega_\CC}{0, \CC} &\leq \Norm{\Psi^*\omega_\CC - \omega}{0, B_{\rho_i}(o)} + \Norm{\omega - \tilde\Psi^*\omega_\CC}{0, B_{\tilde\rho_i}(o)} \\
	&\leq 4C'(\rho_i^\alpha + \tilde\rho_i^\alpha).
\end{align*}
The first inequality is again due to the $r\del_r$-invariance of $\Psi^*\omega_\CC$ and $\tilde\Psi^*\omega_\CC$. Letting $i \to \infty$, it follows that $\Psi^*\omega_\CC = \tilde \Psi^*\omega_\CC$, and hence proceeding by contradiction, \eqref{BlowupMetricExistence} holds for any $\rho > 0$:
\begin{equation}\label{ConvergenceBall}
	\Norm{\omega - \Psi^*\omega_\CC}{0, B_{\rho}(o)} \leq 4C'\rho^\alpha.
\end{equation}
A basic triangle inequality shows that $\frac{1}{2C}\leq \Norm{\Psi^*\omega_\CC}{0, \CC} \leq 2C$, and Lemma \ref{LemmaCompactnessAutomorphisms} shows that this set is compact.
Hence, we can apply $\Psi^{-1}$ to \eqref{ConvergenceBall} and increase $C'$ such that:
\begin{equation*}
	\Norm{(\Psi^{-1})^*\omega - \omega_\CC}{0, B_{\rho}(o)} \leq C'\rho^\alpha,
\end{equation*}
with $C'$ independent of $\omega$. Redefine $\Psi^{-1}$ to $\Psi$, and we have proved the corollary without derivatives.\\

For the rest of the proof, assume $\Psi= \Id$. To extend the estimate to derivatives, assume $\Scal(\omega) = 0$ and fix $p \in \CC$ with $\dist_{\omega_\CC}(o, p) = R < \frac{1}{2}$. Pull back by $R^{-1}$ such that $\tilde p \coloneqq \Phi_R^* (p)$ with $\dist_{\omega_\CC}(o, \tilde p) =1$. Assume that $B_\delta(\tilde p)$ is trivial and pick holomorphic normal coordinates $z_1 = x_1 + ix_2,\dots,z_m = x_{2m-1} + ix_{2m}$ for $\omega_\CC$. Let $\tilde \omega = R^{-2}\Phi_{R}^*(\Psi^*\omega)$ and write $\tilde \omega^m = e^{\tilde F}\omega_\CC^m$. Then
\begin{equation*}
	\Norm{\tilde \omega - \omega_\CC}{0, B_{2}(o)} \leq C'R^\alpha
\end{equation*}
by \eqref{ConvergenceBall} and the assumption $\Psi = \Id$. The uniform bound implies that
\begin{equation*}
	\Norm{e^{\tilde F}- 1}{0, B_{2}(o)} \leq  C''\Norm{\omega - \omega_\CC}{0, B_{2}(o)} \leq  C''R^\alpha,
\end{equation*}
for some constant $C''>0$. By the scalar-flatness of $\tilde \omega$ and Ricci-flatness of $\omega_\CC$, then
\begin{equation}\label{LaplacianTildeF}
	\Delta_{\tilde \omega} \log e^{\tilde F} = 0.
\end{equation}
To estimate $|\log e^{\tilde F}|$, assume first that $e^{\tilde F} \geq 1$. Then the inequality $\log(1+x) \leq x$ for $x \geq 0$ implies that
\begin{equation*}
	\log e^{\tilde F} \leq e^{\tilde F} - 1 \leq C''R^\alpha.
\end{equation*}
The case $e^{\tilde F} < 1$ is similar, using $|\log(x)| = \log(x^{-1})$. By Lemma \ref{LemmaHolderBound}:
\begin{equation*}
	\Norm{\tilde \omega-\omega_\CC}{0,\alpha, B_{\delta}(\tilde p)} \leq C'' R^\alpha.
\end{equation*}
Hence, using the Schauder estimates and bootstrapping, it follows that
\begin{equation*}
	\Norm{e^{\tilde F}-1}{2,\alpha, B_{\frac{\delta}{2}}(\tilde p)} \leq C'' R^\alpha
\end{equation*}
for some constant $C''>0$.\\

Given the regularity of $e^{\tilde F}$, we upgrade the estimate to $\omega-\omega_\CC$. 
\cite[Lemma 2.1]{chenalpha2015} shows that there exists a potential $\phi\in C_\loc^{\infty}(B_{\frac{\delta}{2}}(\tilde p))$ such that
\begin{equation*}
	i\del\delbar \phi = \tilde\omega - \omega_\CC, \quad \Norm{\phi}{2,\alpha, B_{\frac{\delta}{2}}(\tilde p)} \leq \Norm{\tilde \omega - \omega_\CC}{0,\alpha, B_{\delta}(\tilde p)} < C'R^\alpha.
\end{equation*}

For higher-order regularity, differentiate the complex Monge-Ampère equation:
\begin{equation*}
	\frac{\partial}{\partial x_i}\tilde\omega^m = \frac{\partial}{\partial x_i} e^{\tilde F}\omega_\CC^m,
\end{equation*}
to find
\begin{equation}\label{BootStrapCorollary}
	m \Delta_{\tilde\omega} \frac{\partial \phi}{\partial x_i} = \frac{\omega_\CC^m}{\tilde \omega^m}\frac{\partial}{\partial x_i} \left(e^{\tilde F} - 1\right) + \frac{m}{\tilde\omega^m}\left(\left(e^{\tilde F} - 1  \right) \omega_\CC^{m-1} \wedge \frac{\partial}{\partial x_i} \omega_\CC \right).
\end{equation}
Due to the above estimates, the $C^{0,\alpha}(B_{\frac{\delta}{2}}(\tilde p))$-norm of the right-hand side is bounded by $C R^{\alpha}$ for $R>0$ small enough. Hence, the Schauder estimates imply that $\Norm{\phi}{3,\alpha, B_{\frac{\delta}{3}}(\tilde p)} \leq CR^\alpha$. Differentiating \eqref{BootStrapCorollary} and bootstrapping \eqref{LaplacianTildeF} and \eqref{BootStrapCorollary}, we incrementally increase the regularity of $(e^{\tilde F} - 1)$ and $(\tilde\omega - \omega_\CC)$ with $C^{k,\alpha}(B_{\frac{\delta}{4}}(\tilde p))$-bounds bounded by $R^{\alpha}$. We conclude that $\Norm{\phi}{k,\alpha, B_{\frac{\delta}{4}}(\tilde p)} \leq C_k R^\alpha$.
Covering $\{r=1\}$ with finitely many such balls and scaling back implies the desired result:
\begin{equation*}
	|\nabla_{\omega_\CC}^k((\Psi^{-1})^*\omega - \omega_\CC)|_{\omega_\CC} \leq C_k R^{\alpha - k},
\end{equation*}
as $R\to 0$.
\end{proof}

\appendix

\section{Sobolev and Nash Inequalities on ${C(L)}$}

\begin{prop}\label{SobolevInequality}
	Let $({C(L)},g_{C(L)})$ be a Riemannian cone of real dimension $n$ with completion $\overline{C(L)} = {C(L)} \cup \{o\}$ and with the distance metric $\dist_{g_{{C(L)}}}$ on ${C(L)}$ extended to $\overline{C(L)}$. Let $g$ be another metric uniformly equivalent to ${C(L)}$, i.e., $\frac{1}{C'}g_{C(L)} \leq g \leq C' g_{C(L)}$. Then there exists a constant $C$ such that
	\begin{equation}\label{EqSobolevIneq}
		\left(\int_{B_r(p)} |u|^{\frac{2n}{n-2}} \dvol_{g}\right)^{\frac{n-2}{n}} \leq C \int_{B_r(p)} |\nabla_g u|_{g}^2 \dvol_{g},
	\end{equation}
	for $B_r(p)\subset {C(L)}$, $p\in \overline{C(L)}$, $0<r\leq \infty$, and $u\in C_0^\infty({C(L)})$.
\end{prop}

\begin{proof}
	See \cite{heinWeighted2011}, \cite{grigoryanStability2005}, and \cite{minerbeWeighted2009} for the method of proof. Take $u\in C_0^\infty(B_r(p))$. We assume $u\geq 0$ as $|\nabla |u||\leq |\nabla u|$. By the uniform equivalence, it is enough to show \eqref{SobolevInequality} for $g_{C(L)}$ and for $B_r(p) = {C(L)}$ (i.e., $r=\infty$), as $C_0^\infty(B_{r'}(p))\subset C_0^\infty({C(L)})$ for any $r'\leq \infty$. Let $A_i = [2^i,2^{i+1}]\times L$ for $i\in \Z$ and define the average with respect to $g_{C(L)}$ by
	\begin{equation*}
		u_i = \fint_{A_i} u.
	\end{equation*}
	Then
	\begin{align}
		\left(\int_{C(L)} u^{\frac{2n}{n-2}} \dvol_{g_{C(L)}}\right)^{\frac{n-2}{n}} &= \left(\sum_{i\in \Z} \int_{A_i} u^{\frac{2n}{n-2}} \dvol_{g_{C(L)}} \right)^{\frac{n-2}{n}} \nonumber\\
		&\leq \sum_{i\in \Z} \left( \int_{A_i}u^{\frac{2n}{n-2}}\dvol_{g_{C(L)}} \right)^{\frac{n-2}{n}} \nonumber\\
		&\leq \sum_{i\in \Z} 2\left( \left(\int_{A_i} (u-u_i)^{\frac{2n}{n-2}}\dvol_{g_{C(L)}}\right)^{\frac{n-2}{n}} + u_i^2 |A_i|^{\frac{n-2}{n}} \right). \label{ExpansionU}
	\end{align}
	Using the Neumann-type Sobolev inequality on the first term (scaled up or down from $A_0$) on $g_{C(L)}$, then
	\begin{align*}
		\left(\int_{C(L)} u^{\frac{2n}{n-2}}\dvol_{g_{C(L)}} \right)^{\frac{n-2}{n}} &\leq \sum_{i\in \Z} \left(C \left(\int_{A_i} |\nabla_{g_{C(L)}} u|^2 \dvol_{g_{C(L)}}\right)+ 2u_i^2 |A_i|^{\frac{n-2}{n}} \right)\\
		&= C\int_{C(L)} |\nabla_{g_{C(L)}} u|^2\dvol_{g_{C(L)}} + 2\sum_{i\in \Z} u_i^2 |A_i|^{\frac{n-2}{n}}.
	\end{align*}
	For some large $K\in \N$ not yet defined, expand the second term as
	\begin{align*}
		2\sum_{i\in \Z} u_i^2 |A_i|^{\frac{n-2}{n}} &\leq 4\sum_{i\in \Z} (u_i-u_{i+K})^2 |A_i|^{\frac{n-2}{n}} + 4\sum_{i\in \Z} u_{i+K}^2 |A_i|^{\frac{n-2}{n}} \\
		&=  4\sum_{i\in \Z} (u_i-u_{i+K})^2 |A_i|^{\frac{n-2}{n}} + 4\sum_{i\in \Z} \left(\frac{|A_{i}|}{|A_{i+K}|}\right)^{\frac{n-2}{n}} u_{i+K}^2 |A_{i+K}|^{\frac{n-2}{n}}.
	\end{align*}
	As $|A_i|/|A_{i+K}| = 2^{-nK}$, making $K$ large means that the second term above can be absorbed into the left-hand side by increasing the constant 2. For the first term, one estimates each term in the sum as follows: define $B_{i,K} \coloneqq \left( \bigcup_{j=0}^K A_{i+j}\right)$, then
	\begin{align*}
		(u_i-u_{i+K})^2 &= \left(\frac{1}{|A_i|}\int_{A_i} u(x) \dvol_{g_{C(L)}}(x) - \frac{1}{|A_{i+K}|}\int_{A_{i+K}} u(y) \dvol_{g_{C(L)}}(y) \right)^2\\
		&= \left(\frac{1}{|A_i||A_{i+k}|} \int_{A_i\times A_{i+K}} (u(x)-u(y)) \dvol_{g_{C(L)}}^2(x,y) \right)^2 \\
		&\leq \frac{1}{|A_i||A_{i+k}|}\int_{A_i\times A_{i+K}} |u(x)-u(y)|^2 \dvol_{g_{C(L)}}^2(x,y)\\
		&\leq \frac{1}{|A_i||A_{i+k}|}\int_{B_{i,K}\times B_{i,K}} |u(x)-u(y)|^2 \dvol_{g_{C(L)}}^2(x,y) \\
		&= \frac{1}{|A_i||A_{i+k}|}\int_{B_{i,K}\times B_{i,K}} |u(x)-u_{B_{i,K}} + u_{B_{i,K}} - u(y)|^2 \dvol_{g_{C(L)}}^2(x,y) \\
		&\leq \frac{4|B_{i,K}|}{|A_i||A_{i+k}|}\int_{B_{i,K}} |u(x)-u_{B_{i,K}}|^2 \dvol_{g_{C(L)}} \\
		&\leq \frac{|B_{i,K}|}{|A_i||A_{i+k}|} C 2^{2i} \int_{B_{i,K}} |\nabla_{g_{C(L)}} u|^2 \dvol_{g_{C(L)}},
	\end{align*}
	where $u_{B_{i,K}}$ denotes the average of $u$ over $B_{i,K}$. The last inequality is the Poincaré–Wirtinger inequality of $g_{C(L)}$ applied to each $B_{i,K}$ suitably scaled. Therefore, the first term in \eqref{ExpansionU} becomes
	\begin{align*}
		\sum_{i\in \Z} (u_i-u_{i+K})^2 |A_i|^{\frac{n-2}{n}} \leq C \sum_{i\in \Z} |A_i|^{\frac{n-2}{n}} \frac{|B_{i,K}|}{|A_i||A_{i+K}|} 2^{2(i+K)} \int_{B_{i,K}} |\nabla_{g_{C(L)}} u|^2.
	\end{align*}
	\sloppy The volumes are estimated as follows: $|A_i|^{\frac{n-2}{n}} \sim 2^{i(n-2)}, |B_{i,K}| \sim K\sum_{j=0}^K 2^{in},( |A_i||A_{i+K}|)^{-1}\sim 2^{-ni}2^{-n(i+K)}$, and ordering all the powers of $2$:
	\begin{equation*}
		2^{i(n-2)} \cdot 2^{-ni}2^{-n(i+K)} \cdot 2^{(i+j)n} \cdot 2^{2i} 2^K =  \text{constant}.
	\end{equation*}
	Furthermore,
	\begin{equation*}
		\sum_{i\in \Z} \int_{B_{i,K}} |\nabla_{g_{C(L)}} u|^2 \leq K \int_{C(L)} |\nabla_{g_{C(L)}} u|^2,
	\end{equation*}
	finishing the proof for the metric $g_{C(L)}$.
\end{proof}

\begin{prop}\label{NashInequality}
	Under the same assumptions as in Proposition \ref{SobolevInequality} and for all $u\in C_0^\infty(B_r(p))$, then
	\begin{equation*}
		\left(\int_{B_r(p)} u^2 \dvol_{g}\right)^{1+\frac{2}{n}} \leq C \left(\int_{B_r(p)} |\nabla_g u|^2 \dvol_{g}\right) \left(\int_{B_r(p)} |u| \dvol_{g} \right)^{\frac{4}{n}}.
	\end{equation*}
\end{prop}

\begin{proof}
	We again assume that $u\geq 0$. By Proposition \ref{SobolevInequality} and Hölder's inequality:
	\begin{align*}
		\int u^2\dvol_{g} &= \int_{B_r(p)} u^a u^b\dvol_{g} \\
		&\leq \left(\int_{B_r(p)} u^{ap}\dvol_{g} \right)^{\frac{1}{p}} \left(\int_{B_r(p)} u^{bq}\dvol_{g}  \right)^{\frac{1}{q}}\\
		&= C\left(\int_{B_r(p)} |\nabla_g u|^2\dvol_{g} \right)^{\frac{n}{p(n-2)}} \left(\int_{B_r(p)} u^{bq} \dvol_{g}\right)^{\frac{1}{q}},
	\end{align*}
	with $ap = \frac{2n}{n-2}$, $bq = 1$, $a+b=2$, and $\frac{1}{p} + \frac{1}{q} = 1$.
\end{proof}

\section{Gradient estimate of the Heat Kernel}

The following proposition is known for the Ricci flow as the Bernstein-Bando-Shi estimate \cite[Theorem 14.5]{chowRicci2007}, but we record the statement for the heat equation. See \cite[Proposition 2.4]{heingravitational2010} for the statement on elliptic differential equations.

\begin{prop}\label{GradientEstimate}
	Let $(M,g)$ be a Riemannian manifold of dimension $n$. Assume that $\Omega'\subset \Omega \subset M$ are open sets such that $\overline{\Omega'}\subset \Omega$ and $\Omega$ has compact closure in $M$. Let $0<t_1<t_2<t_3<t_4$. Assume that there exist Lipschitz continuous functions $\chi\colon M \times \R_+ \to \R, \beta\colon \Omega\to \R$ such that $\chi \equiv 1$ on $\Omega'\times [t_2,t_3]$, with support in $\Omega\times [t_1,t_4]$, and such that there exist constants $s>0,C>1$ with
	\begin{equation}\label{EqChiBeta}
		s \chi^{-\frac{1}{2}} |\nabla \chi| + s^2(\Delta \chi)^- + s^2|\del_t \chi| \leq C, \quad C\Delta \beta \geq 1, \quad 0\leq \beta < Cs^2,
	\end{equation}
	where the differential inequalities are to be interpreted in the weak sense. Let $\Rm$ be the curvature tensor of $g$, and assume that on $\Omega$:
	\begin{equation}\label{CurvatureEstimate}
		\sum_{j=0}^{k} s^{j+2}|\nabla^j \Rm| \leq C.
	\end{equation}
	Furthermore, on $\Omega\times[t_1,t_4]$, assume that $u$ is a smooth function satisfying $(\del_t - \Delta) u = 0$, and that there is $C_0>0$ such that
	\begin{equation}\label{HeatSolutionInductiveEstimate}
		\sum_{j=0}^{k-1} s^j|\nabla^j u| \leq C_0.
	\end{equation}
	Then $s^k|\nabla^k u| \leq C'C_0$ on $\Omega'\times [t_2,t_3]$, where $C'>1$ only depends on $k, \Omega$, and $\Omega'$.
\end{prop}

\begin{proof}
	By Bochner's formula:
	\begin{equation}\label{HeatEquationBochner}
		(\del_t - \Delta) \frac{1}{2} |\nabla^k u|^2 = - |\nabla^{k+1} u|^2 - \gen{[\Delta,\nabla^k]u,\nabla^k u},
	\end{equation}
	for $k\in \N_0$ and because $u$ solves the heat equation. The constant $C$ may change in the proof but is independent of estimates on $u$. Define $\Phi \coloneqq \frac{1}{2}(\chi^2 |\nabla^k u|^2 + s^{-2}C_1 |\nabla^{k-1}u|^2)$. Using \eqref{HeatEquationBochner}, one computes
	\begin{equation}\label{PhiExpansion1}
		\begin{aligned}
			(\del_t - \Delta)\Phi =& (\chi(\del_t \chi - \Delta \chi)- |\nabla\chi|^2 - s^2C_1)|\nabla^k u|^2 - \chi\gen{\nabla \chi, \nabla |\nabla^k u|^2} - \chi^2|\nabla^{k+1} u|^2 \\
			&- \chi^2\gen{[\Delta,\nabla^k] u, \nabla^k u} - C_1s^{-2}\gen{[\Delta,\nabla^{k-1}]u,\nabla^{k-1}u}.
		\end{aligned}
	\end{equation}
	By the inequality $ |\chi\gen{\nabla \chi, \nabla |\nabla^k u|^2}| \leq \epsilon\chi^2|\nabla^{k+1}u|^2 + \frac{4}{\epsilon}|\nabla\chi|^2|\nabla^k u|^2$ for every $\epsilon>0$ and the estimates for the commutator \cite[Lemma 2]{bandoReal1987}:
	\begin{equation*}
		|[\Delta,\nabla^k]T| \leq C(p+1)\sum_{j=0}^k\frac{(k+2)!}{(j+2)!(k-j)!}|\nabla^j \Rm||\nabla^{k-j}T|,
	\end{equation*}
	where $C$ only depends on $\dim M$ and $T$ is any $p$-tensor, \eqref{PhiExpansion1} can be estimated as
	\begin{align*}
		(\del_t - \Delta)\Phi =& (\chi(\del_t \chi - \Delta \chi)- |\nabla\chi|^2 - s^2C_1)|\nabla^k u|^2 - \chi\gen{\nabla \chi, \nabla |\nabla^k u|^2} - \chi^2|\nabla^{k+1} u|^2 \\
		&- \chi^2\gen{[\Delta,\nabla^k] u, \nabla^k u} - C_1s^{-2}\gen{[\Delta,\nabla^{k-1}]u,\nabla^{k-1}u}\\
		\leq& (\chi(\del_t \chi - \Delta \chi) -  |\nabla\chi|^2 - s^{-2}C_1)|\nabla^k u|^2  + \frac{4}{\epsilon}|\nabla\chi|^2|\nabla^k u|^2 \\
		&+\chi^2 C |\nabla^k u|\sum_{j=0}^k |\nabla^j\Rm||\nabla^{k-j}u| + C_1s^{-2}C |\nabla^{k-1} u| \sum_{j=0}^{k-1}|\nabla^j\Rm||\nabla^{k-j-1}u| \\
		\leq&  (\chi(\del_t \chi - \Delta \chi) - |\nabla\chi|^2 - s^{-2}C_1)|\nabla^k u|^2  + \frac{4}{\epsilon}|\nabla\chi|^2|\nabla^k u|^2\\
		&+\chi^2 C C_0 |\nabla^k u|\sum_{j=0}^{k-1} s^{-2-j} s^{-k+j} + C_1 C_0 s^{-2}C s^{-k+1} \sum_{j=0}^{k-1}s^{-j-2}s^{-k+j+1} \\
		&+ \chi^2 C s^{-2}|\nabla^k u|^2 \\
		\leq&  (\chi(\del_t \chi - \Delta \chi)-  |\nabla\chi|^2 - s^{-2}C_1)|\nabla^k u|^2  + \frac{4}{\epsilon}|\nabla\chi|^2|\nabla^k u|^2\\
		&+\chi^2 C C_0s^{-2k-2} |\nabla^k u| + CC_1 C_0 s^{-2k-2} + \chi^2 C s^{-2}|\nabla^k u|^2 + \chi^2 C s^{-2}|\nabla^k u|^2 \\
		\leq&  (\chi(\del_t \chi - \Delta \chi) - |\nabla\chi|^2 - s^{-2}C_1)|\nabla^k u|^2  + \frac{4}{\epsilon}|\nabla\chi|^2|\nabla^k u|^2\\
		&+\chi^2 C C_0s^{-2k-2} + \chi^2 C C_0 s^{-2} |\nabla^k u|^2 + CC_1 C_0 s^{-2k-2} + \chi^2 C s^{-2}|\nabla^k u|^2\\
		\leq&  (\chi(\del_t \chi - \Delta \chi)  - |\nabla\chi|^2 + \chi^2 C C_0 +  \chi^2 C s^{-2} + \frac{4}{\epsilon}|\nabla\chi|^2- s^{-2}C_1) |\nabla^k u|^2 \\
		&+(\chi^2+C_1) C C_0s^{-2k-2},
	\end{align*}
	where we have estimated $|\nabla^j\Rm||\nabla^{k-j}u|$ via \eqref{CurvatureEstimate} and \eqref{HeatSolutionInductiveEstimate}. By choosing $C_1\geq 1$ big enough, the first parenthesis above is nonpositive, and hence
	\begin{equation*}
		(\del_t - \Delta)\Phi \leq C C_1 C_0 s^{-2k-2}.
	\end{equation*}
	By the construction of $\beta$, then
	\begin{equation*}
		(\del_t - \Delta)(\Phi - C C_1 C_0 s^{-2k-2}\beta) \leq 0.
	\end{equation*}
	Using the maximum principle for parabolic partial differential equations, it follows that
	\begin{equation*}
		\sup_{\Omega' \times [t_2,t_3]} |\nabla^k u| \leq C s^{-k}.
	\end{equation*}
\end{proof}

\section{Harmonic Functions and Constant-Trace Forms of Prescribed Growth}

Throughout this section, $(C(L),g_{C(L)})$ denotes a $n$-dimensional Riemannian conical manifold (not necessarily Ricci-flat) of dimension $n \geq 3$. For Lemma \ref{Lemma1Form} and Proposition \ref{PropGrowthRateToHomogeneousForm}, we need to assume that $(C(L),\omega_{C(L)})$ is Ricci-flat and Kähler. We show that harmonic functions and constant trace-forms of growth $\OO(r^{\beta})$ as $r\to 0$ and $\OO(r^{\alpha})$ as $r\to\infty$ for $\beta \leq \alpha$ satisfy nice regularity properties. On ${C(L)} \cong \R_+\times L$ with coordinates $(r,x)$, define for $\beta \leq \alpha$ the function $\theta_{\beta,\alpha}\colon {C(L)} \to \R_+$ by
\begin{equation}\label{Theta}
	\theta_{\beta,\alpha}(r,x) = \begin{cases}
		r^{\beta} & r \leq 1,\\ r^{\alpha} & r > 1.
	\end{cases}
\end{equation}
The estimation constants may change in all lines below. Denote by $\Delta_L$ the Laplacian of $g_{C(L)}|_{\{r=1\}}$ and by $\dvol_L \coloneqq \dvol_{g_{C(L)}|_{\{r=1\}}}$ the volume form on the link. On the link $L \cong \{r = 1\}$ with Laplacian $\Delta_L$, let $(\phi_j)$ be an orthonormal basis of eigenfunctions of $-\Delta_L$ with eigenvalues $0 = \lambda_0 < \lambda_1 \leq \cdots$, counted with multiplicities.

\begin{prop}\label{PropGrowthRateToHomogeneousFunction}
	Let $f \colon {C(L)}\to \R$ be a smooth function such that $\Delta_{g_{{C(L)}}} f= 0$ and $|f|\leq C\theta_{\beta,\alpha}$. Then $f$ splits into a finite sum
	\begin{equation*}
		f = \sum_{j=0}^\infty (a_j^+ r^{c_j^+} + a_j^- r^{c_j^-}) \phi_j, \quad c_j^+ \geq 0, c_j^- \leq 0,
	\end{equation*}
	where $a_j^{\pm}\neq 0$ only if $c_j^{\pm}\in [\beta,\alpha]$. 
	
	If $(C(L), g_{C(L)})$ has nonnegative Ricci curvature, then $c_j^-\leq 1-n$ and $c_j^+ \geq 1$ for $j\geq 1$ with equality if and only if $(C(L)\cup \{o\},g_{C(L)} ) \cong (\R^n, g_\eucl)$.
\end{prop}

\begin{proof}
	Let $f \colon {C(L)} \to \R$ be harmonic with respect to $\Delta_{g_{{C(L)}}}$. On the link $L = \{r = 1\}$ with Laplacian $\Delta_L$, let $(\phi_j)$ be an orthonormal basis of eigenfunctions of $-\Delta_L$ with eigenvalues $0 = \lambda_0 < \lambda_1 \leq \cdots$. Expand $f$ in terms of these functions:
	\begin{equation*}
		f = \sum_{j=0}^\infty g_j(r) \phi_j(x),
	\end{equation*}
	with the sum converging in $L^2_\loc({C(L)})$. As $f$ is smooth, the convergence can be upgraded to $W^{2,k}_\loc({C(L)})$ for any $k\geq 0$, and hence also to $C^\infty_\loc({C(L)})$ by the Morrey embedding. The full Laplacian $\Delta_{g_{{C(L)}}}$ is:
	\begin{equation}\label{LaplacianLink}
		\Delta_{g_{{C(L)}}} f = \del_r^2 f + \frac{n-1}{r}\del_r f + \Delta_L f.
	\end{equation}
	We solve \eqref{LaplacianLink} for each eigenfunction $\phi_j$, and it follows that:
	\begin{equation}\label{ExpansionLaplacian}
		f = \sum_{j=0}^\infty (a_j^+ r^{c_j^+} + a_j^- r^{c_j^-}) \phi_j,
	\end{equation}
	with
	\begin{equation}\label{ExpressionCi}
		c_j^\pm = \frac{2-n\pm \sqrt{(n-2)^2 + 4\lambda_j}}{2}.
	\end{equation}
	Denote by $\gen{f,\phi_j}_L$ the $L^2$-inner product of $f$ and $\phi_j$ at every level set $\{r=R\}$ with the volume form at $\{r=1\}$: 
	\begin{equation}\label{L2InnerProduct}
		\gen{f,\phi_j}_L(R) = \int_{\{r=R\}} f(r,x) \phi_j(x) \dvol_L(x).
	\end{equation}
	Then
	\begin{equation*}
		|a_j^+ r^{c_j^+} + a_j^- r^{c_j^-}| = |\gen{f,\phi_j}_L| \leq |\gen{f,f}_L|^{\frac{1}{2}} \leq C\theta_{\beta,\alpha}(r).
	\end{equation*}
	Taking $r\to 0$ or $r\to \infty$, we see that $a_j^{\pm}\neq 0$ only if $c_j^\pm \in [\beta,\alpha]$.
	
	If $(C(L), g_{C(L)})$ has nonnegative Ricci curvature, then $(L, g_L)$ has Ricci curvature $(n - 2)$. The Lichnerowicz-Obata Theorem \cite[Theorem 2.1]{ivanovLichnerowicz2015} shows that $\lambda_1 \geq n - 1$ with equality if and only if $(C(L) \cup \{o\}, g_{C(L)}) \cong (\mathbb{R}^{n}, g_\eucl)$. Substituting this result into \eqref{ExpressionCi} shows that $c_j^+ \geq 1$ with equality if and only if $(C(L) \cup \{o\}, g_{C(L)}) \cong (\mathbb{R}^{2m}, g_\eucl)$. The same argument shows $c_j^- \leq 1-n$.
\end{proof}

A similar statement holds for exact 2-forms of constant trace in Kähler case. To prove this, we need the following two lemmas:

\begin{lemma}[Poincaré Lemma with Estimates]\label{PoincareLemma}
	Let $\eta\in \Omega^2({C(L)})$ satisfy $d\eta = 0$ and 
	\begin{equation*}
		|\nabla_{g_{C(L)}}^l \eta | \leq Cr^{-l}\theta_{\beta,\alpha}(r), \quad l=0,\dots,k,
	\end{equation*}
	for some $k\in \N$, constant $C>0$, and any $\beta > -2$. Then there exists a form $\zeta\in \Omega^1({C(L)})$ such that $\eta = d\zeta$ and with the estimates:
	\begin{equation*}
		|\nabla_{g_{C(L)}}^l \zeta| \leq C' r^{1-l}\theta_{\beta,\alpha}(r), \quad l=0,\dots,k,
	\end{equation*}
	for $C' = C'(C,k,\alpha)>0$.
\end{lemma}

\begin{proof}
	The proof mimics the classical proof of the Poincaré Lemma using the Lie derivative. Let $\Theta_{t}\colon {C(L)}\to {C(L)}$ denote the one-parameter family of scaling maps $p = (r',x)\mapsto (e^t\cdot r',x)$ with generating vector field $r\del_r$. Then
	\begin{equation*}
		\frac{d}{dt}\Theta_t^*\eta = \Theta_t^*\LL_{r\del_r}\eta.
	\end{equation*}
	Integrating both sides and using the Cartan formula:
	\begin{equation}\label{IntegralStokes}
		\begin{aligned}
			\Theta_{t_1}^*\eta - \Theta_{t_0}^*\eta &= d\int_{t_0}^{t_1}\Theta_t^*i(r\del_r)\eta \, dt + \int_{t_0}^{t_1}\Theta_t^*i(r\del_r)d\eta\, dt\\
		&= d\int_{t_0}^{t_1}\Theta_t^*i(r\del_r)\eta \, dt,
		\end{aligned}
	\end{equation}
	as $d\eta = 0$. Write $r = e^t$ and $p = (e^{t'},x)$. We estimate the integrand by
	\begin{equation*}
		|\Theta_t^*i(r\del_r)\eta_{(e^{t'},x)}| \leq e^t \Theta_t^*|i(r\del_r)\eta_{(e^{t'},x)}| \leq  e^t \Theta_t^*(|r\del_r||\eta_{(e^{t'},x)}|) \leq Ce^{2t+t'}\theta_{\beta,\alpha}(e^{t+t'}).
	\end{equation*}
	For any fixed $t_1\in \R$, $\beta > -2$ ensures that the limit of \eqref{IntegralStokes} exists as $t_0\to -\infty$. Hence, setting $t_1 = 0$ and $t_0\to-\infty$, we find:
	\begin{equation*}
		\eta = d\int_{-\infty}^{0}\Theta_t^*i(r\del_r)\eta \, dt \eqqcolon d \zeta.
	\end{equation*}
	By using the Koszul formula, we see that $\nabla_{g_{C(L)}} r\del_r = \Id$ regarded as a map $\X({C(L)})\to \X({C(L)})$. Hence $\nabla_{g_{C(L)}}^l r\del_r = 0$ for $l \geq 2$ and $|\nabla_{g_{C(L)}} r\del_r| = 1$. Therefore, we estimate $\zeta$ and its derivatives using the product rule:
	\begin{equation}\label{IntegralStokesDerivative}
		\begin{aligned}
			|\nabla_{g_{C(L)}}^l \zeta|_{(e^{t'},x)} &= \left|\nabla_{g_{C(L)}}^l \int_{-\infty}^{0}\Theta_t^*i(r\del_r)\eta_{(e^{t'},x)} \, dt\right|\\
			&\leq\int_{-\infty}^{0} \left|\nabla_{g_{C(L)}}^l \left(\Theta_t^*i(r\del_r)\eta_{(e^{t'},x)} \right)\right| \, dt \\
			&=\int_{-\infty}^{0} \left| \Theta_t^* \left(\nabla_{(\Theta_t^{-1})^*g_{C(L)}}^l\left(i(r\del_r)\eta_{(e^{t'},x)}\right)\right)\right| \, dt\\
			&=\int_{-\infty}^{0}\left| \Theta_t^* \left(\nabla_{g_{C(L)}}^l\left(i(r\del_r)\eta_{(e^{t'},x)}\right)\right)\right| \, dt\\
			&=\int_{-\infty}^{0} e^{(l+1)t}  \Theta_t^*\left|\nabla_{g_{C(L)}}^l(i(r\del_r)\eta)\right| \, dt\\
			&\leq \int_{-\infty}^{0} C e^{(1+l)t} e^{(1-l)(t'+t)}\theta_{\beta,\alpha}(e^{t'+t})  \, dt\\
			&= C\int_{-\infty}^{0} e^{(1-l)t'}e^{2t}\theta_{\beta,\alpha}(e^{t'+t})  \, dt\\
			&= C \begin{cases}
				\frac{e^{(1+\beta-l)t'}}{2+\beta} & t' \leq 0,\\
				\left(\frac{e^{(-1 -l)t'}}{2+\beta}-\frac{e^{(-1 - l) t'}}{2+\alpha}\right) + \frac{ e^{(1+\alpha-l)t'}}{2+\alpha} & t'>0,
			\end{cases}
		\end{aligned}
	\end{equation}
	for $l=0,\dots,k$. These are exactly the desired estimates.
\end{proof}

For the next lemma, we need the following definition: Let $\E = \{c_j^\pm \in \R \mid \text{$c_j^\pm$ is a solution of \eqref{ExpressionCi}} \}$, i.e. the set of homogeneity factors of harmonic functions on $(C(L), g_{C(L)})$.
\begin{lemma}\label{PotentialNiceChoice}
	Take $\beta \leq 0 < \alpha$ and such that $2+\alpha, 2+\beta \notin \E$. If $f\colon {C(L)} \to \R$ is a smooth function satisfying
	\begin{equation*}
		|\nabla^l f| \leq Cr^{-l}\theta_{\beta,\alpha}(r), \quad l=0,\dots, k,
	\end{equation*}
	for some $C>0$, $k \geq3n+1$, $\beta \leq 0 < \alpha$, then there exists a smooth function $u\colon {C(L)} \to \R$ such that 
	\begin{equation*}
		\Delta_{g_{C(L)}} u = f, \qquad |\nabla^{l}u| \leq C'r^{2-l}\theta_{\beta,\alpha}(r), \quad l=0,\dots, k+2,
	\end{equation*}
	for some constant $C' = C'(C,k,\alpha,\beta)>0$.
\end{lemma}

\begin{proof}
	As in the proof of Proposition \ref{PropGrowthRateToHomogeneousFunction}, pick an orthonormal basis of eigenfunctions $(\phi_j)$ of $-\Delta_L$ with eigenvalues $0 = \lambda_0 < \lambda_1 \leq \cdots$. On ${C(L)} \cong \R_+\times L$ with coordinates $(r,x)$, expand $f$ in terms of these functions:
	\begin{equation*}
		f = \sum_{j=0}^\infty g_j(r) \phi_j(x),
	\end{equation*}
	with the sum converging in $L^2_\loc({C(L)})$. As $f$ is smooth, the convergence can be upgraded to $W^{2,k}_\loc({C(L)})$ for any $k\geq 0$, and hence also to $C^\infty_\loc({C(L)})$ by the Morrey embedding. Let $\gen{f,\phi_j}_L$ be as in \eqref{L2InnerProduct}. The growth property of $f$ implies that 
	\begin{equation}\label{AsymptoticsG}
		|g_j| = |\gen{f,\phi_j}_L| \leq C \theta_{\beta,\alpha}(r).
	\end{equation}
	We seek a solution $u$ of the same form:
	\begin{equation*}
		u = \sum_{j=0}^\infty h_j(r) \phi_j(x).
	\end{equation*}
	Write $\Delta_{g_{C(L)}} u= \del_r^2 u + \frac{n-1}{r}\del_r u + \frac{1}{r^2}\Delta_L u$. We apply the Laplacian on each $h_j(r) \phi_j(x)$ to obtain an equation of the form:
	\begin{equation}\label{ODEj}
		(\del_r^2 h_j + \frac{n-1}{r}\del_r h_j - \frac{\lambda_j}{r^2} h_j)\phi_j(x) = g_j(r) \phi_j(x).
	\end{equation}
	The homogeneous solutions are $r^{c_j^+}$ and $r^{c_j^-}$ with
	\begin{equation}\label{HomogeneityCoefficients}
		c_j^\pm = \frac{2-n\pm \sqrt{(n-2)^2 + 4\lambda_j}}{2}.
	\end{equation}
	To solve for the inhomogeneous part, the Wronskian is:
	\begin{equation*}
		W_j(r) = \begin{vmatrix}
			r^{c_j^+} & r^{c_j^-} \\ c_j^+r^{c_j^+-1} & c_j^-r^{c_j^--1}
		\end{vmatrix},
	\end{equation*}
	giving the general solution of \eqref{ODEj}:
	\begin{equation*}
		h_j = \hat h_j + \Lambda_j^+ r^{c_j^+} + \Lambda_j^- r^{c_j^-},
	\end{equation*}
	with
	\begin{equation}\label{HatH}
		\hat h_j = 	\frac{1}{c_j^- - c_j^+}\left(-r^{c_j^+}\int_1^r s^{1-c_j^+}g_j(s) \, ds + r^{c_j^-}\int_1^r s^{1-c_j^-}g_j(s) \, ds  \right).
	\end{equation}
	We need to fix $\Lambda_j^\pm$ such that $h_j$ has the desired asymptotics. Set
	\begin{equation*}
		\Lambda_j^+ \coloneqq \begin{cases}
			0 & c_j^+ < 2+\alpha, \\
				\frac{1}{c_j^- - c_j^+}\left(r^{c_j^+}\int_1^\infty s^{1-c_j^+}g_j(s) \, ds \right) & c_j^+ > 2+ \alpha,
		\end{cases}
	\end{equation*}
	and
	\begin{equation*}
		\Lambda_j^- \coloneqq \begin{cases}
			0 & c_j^- > 2+\beta, \\
			\frac{-1}{c_j^- - c_j^+}\left(r^{c_j^-}\int_1^0 s^{1-c_j^-}g_j(s) \, ds \right) & c_j^- < 2+ \beta,
		\end{cases}
	\end{equation*}
	$2+\alpha \neq c_j^+$ and $2+\beta \neq c_j^-$ by assumption.
	The integrals of $\Lambda_j^\pm$ exist by the asymptotics of $g_j$. With these choices of $\Lambda_j^\pm$, it follows that
	\begin{equation*}
		|h_j| \leq C \theta_{\beta,\alpha}(r).
	\end{equation*}
	We show the the proof for $\Lambda_j^+$. If $c_j^+ < 2+\alpha$, then
	\begin{align*}
		\left|r^{c_j^+}\int_1^r s^{1-c_j^+}g_j(s) \, ds  \right| &\leq \left|Cr^{c_j^+}\int_1^r s^{1-c_j^+ + \alpha}\right| \\
		&= Cr^{c_j^+}\left| \left[ \frac{s^{2-c_j^+ + \alpha}}{2-c_j^+ + \alpha} \right]_1^r\right|\\
		&= C\left|\frac{r^{2 + \alpha} - r^{c_j^+}}{2-c_j^+ + \alpha}\right| \\
		&\leq Cr^2\theta_{\beta,\alpha}(r).
	\end{align*}
	 For $c_j^+ > 2 + \alpha$, then
	 \begin{align*}
	 	\left|- r^{c_j^+}\int_1^r s^{1-c_j^+}g_j(s)\, ds + r^{c_j^+}\int_1^\infty s^{1-c_j^+}g_j(s) \, ds  \right| &= \left|r^{c_j^+}\int_r^\infty s^{1-c_j^+}g_j(s) \, ds  \right|\\
	 	&\leq Cr^{c_j^+}\left|\left[\frac{s^{2-c_j^+ + \alpha}}{2-c_j^+ + \alpha}  \right]_r^\infty\right|\\
	 	&= C\frac{-r^{2 + \alpha}}{2-c_j^+ + \alpha} \leq C\theta_{\beta,\alpha}(r).
	 \end{align*}
	Next, we bound $h_j$ in terms of $g_j$ and the associated eigenvalue of $\phi_j$. First, for $j\neq 0$:
	\begin{equation}\label{AbsoluteG}
		\begin{aligned}
			|g_j(r)| &= \lambda_j^{-k}\left|\gen{f,\Delta_L^{k}\phi_j}_L\right|\\ 
			&= \lambda_j^{-k}\left|\gen{\Delta_L^{k} f,\phi_j}_L\right|\\
			&\leq C\lambda_j^{-k}r^{2k}\left|\gen{\Delta_{g_{{C(L)}}}^{k} f,\Delta_{{g_{{C(L)}}}}^k f}_L\right|^{\frac{1}{2}}\\
			&\leq C \lambda_j^{-k} \theta_{\beta,\alpha}(r),
		\end{aligned}
	\end{equation}
	for any $k\in \N$. By \eqref{HatH} and \eqref{AbsoluteG}, there exists a constant $C>0$ such that
	\begin{equation*}
		|h_j(r)| + r|h_j'(r)| + r^2|h_j''(r)| \leq C \lambda_j^{-k}r^{2} \theta_{\beta,\alpha}(r).
	\end{equation*} 
	
	The expression
	\begin{equation}\label{FormalU}
		u = \sum_{j=0}^\infty h_j(r)\phi_j(x)
	\end{equation}
	now formally satisfies $\Delta_{g_{C(L)}} u = f$. To show that $u$ is an actual solution, we need to show that the sum \eqref{FormalU} converges in $C^{2}_\loc({C(L)})$. The Morrey embedding \cite[10. Theorem]{krylovLectures2008} shows that $\Norm{\phi_j}{2l-\frac{n}{2},L}\leq C\Norm{\phi_j}{W^{2l,2}(L)}$ for $l\in \N_0$. As $-\Delta_L \phi_j = \lambda_j \phi_j$ and $\Norm{\phi_j}{L^2(\{r=1\})} = 1$ we obtain:
	\begin{equation*}
		\Norm{\phi_j}{2,L} \leq C(1+\lambda_j)^{l}.
	\end{equation*}
	for some constant $C>0$ depending only on $(C(L),g_{C(L)})$ and $l \geq 1 + \frac{n}{4}$ a natural number.
	
	By Weyl's Asymptotic Formula \cite[p. 155]{chavelHeatKernel1984}, $\lambda_j \sim j^{\frac{1}{2n}}$, so
	\begin{equation}\label{AbsoluteSumF}
		\sum_{j=1}^\infty |\nabla_{g_{{C(L)}}}^l(h_j(r)\phi_j(x))| \leq C\theta_{\beta,\alpha}(r) r^{2-k} \sum_{j=1}^\infty j^{\frac{-k+l}{2n}}, \quad l = 0,1,2.
	\end{equation}
	Choose $k$ such that $\frac{-k+l}{2n}< -1$, and the sum \eqref{FormalU} converges in $C^2_\loc({C(L)})$. For this, it is sufficient to choose $k = 3n+1$. Smoothness of $u$ follows by elliptic regularity, and the growth property of $\nabla_{g_{{C(L)}}}^lu$ follows from \eqref{AbsoluteSumF} for $l=0,1,2$. Growth of the higher order terms follow by the Schauder estimates and rescaling.
\end{proof}

\begin{lemma}\label{Lemma1Form}
	On the Ricci-flat Kähler cone $(\CC, \omega_\CC)$, let $\eta \in \Omega^1(L)$ be a real, co-closed 1-form such that $r^2 \eta$ is harmonic on $(\CC, \omega_\CC)$ and the dual vector field $\eta^\sharp \in \mathfrak{X}(\CC)$ is a Killing field on $(\CC, \omega_\CC)$. Then there exists a 2-homogeneous, $\xi$-invariant function $q \colon \CC \to \R$ with constant Laplacian such that
	\begin{equation*}
		\del(r^2 \eta)^{(0,1)} = i \del \delbar q.
	\end{equation*}
\end{lemma}

\begin{proof}
	First, if $(\CC \cup \{o\}, \omega_\CC) \cong (\C^m,\omega_\eucl)$, then $\del(r^2 \eta)^{(0,1)}$ is a $2$-homogeneous harmonic form on $(\C^m,\omega_\eucl)$. $\del(r^2 \eta)^{(0,1)}$ is therefore a constant $(1,1)$-form, and the lemma is trivial.
	
	Assuming $(\CC \cup \{o\}, \omega_\CC) \neq (\C^m,\omega_\eucl)$, let $g_\CC$ be the Riemannian metric associated to $\omega_\CC$, and extend to a $\C$-bilinear form on the complexified tangent bundle $T_\C \CC$. By Tanno \cite[Theorem 4.4]{tannoisometry1970}, if $(\CC, g_\CC)$ does not admit a hyperkähler structure, then all Killing vector fields are holomorphic. Since the dual vector field $(r^2 \eta)^\sharp$ is a homogeneous Killing field of linear growth on $\CC$, it commutes with the scaling vector field $r \del_r$ (see the proof of Theorem \ref{DecompositionHoloFields}):
	\begin{equation*}
		\left[(r^2 \eta)^\sharp, r \del_r \right] = \LL_{r \del_r} (r^2 \eta)^\sharp = (1 - 1)(r^2 \eta)^\sharp = 0.
	\end{equation*}
	
	By \cite[Theorem 2.14]{heinCalabiYau2017} (or see Theorem \ref{DecompositionHoloFields}), the Killing field $(r^2 \eta)^\sharp$ can be expressed as
	\begin{equation*}
		(r^2 \eta)^\sharp = J \nabla q,
	\end{equation*}
	for some $\xi$-invariant, 2-homogeneous function $q$ of constant Laplacian. Therefore:
	\begin{equation*}
		\del (r^2 \eta)^{(0,1)} = - \del (d^c q)^{(0,1)} = -i \del \delbar q.
	\end{equation*}
	
	Now consider the case where $(\CC, g_\CC)$ is hyperkähler, with complex structures $I$, $J$, and $K$, where $I$ corresponds to the already given complex structure on $\CC$. Suppose $(r^2 \eta)^\sharp$ is not holomorphic. Then, up to linear combination and dropping the holomorphic part, we must have $(r^2 \eta)^\sharp = J(r \del_r)$ or $(r^2 \eta)^\sharp = K(r \del_r)$ by \cite[Theorem 4.4]{tannoisometry1970}. Assume $(r^2 \eta)^\sharp = J(r \del_r)$. Then:
	\begin{equation*}
		r^2 \eta = (J(r \del_r))^\flat = g_\CC(J(r \del_r), \cdot),
	\end{equation*}
	and its $(0,1)$-part is given by:
	\begin{align*}
		(r^2 \eta)^{(0,1)} &= g_\CC(J(r \del_r), \cdot) + i g_\CC(J(r \del_r), I \cdot) \\
		&= g_\CC(J(r \del_r) - i I J(r \del_r), \cdot) \\
		&= g_\CC(J(r \del_r + i I(r \del_r)), \cdot).
	\end{align*}
	Define the vector field:
	\begin{equation*}
		W = r \del_r + i I(r \del_r) \in \Gamma(T^{0,1} \CC).
	\end{equation*}
	Now compute $\del(r^2 \eta)^{(0,1)}$ for $X \in \Gamma(T^{1,0} \CC)$ and $Y \in \Gamma(T^{0,1} \CC)$:
	\begin{align*}
		(\del(r^2 \eta^{(0,1)}))(X, Y) &= 	(d(r^2 \eta^{(0,1)}))(X, Y)\\
		&= \nabla_X(r^2 \eta^{(0,1)}) (Y) - \nabla_Y(r^2 \eta^{(0,1)})(X)\\
		&= g_\CC(\nabla_X JW, Y) + g_\CC(\nabla_Y JW, X).
	\end{align*}
	As $J \colon T^{1,0}\CC \to  T^{0,1}\CC$ and $J \colon T^{0,1}\CC \to  T^{1,0}\CC$, and $\nabla_Z$ preserves the decomposition $T_\C \CC = T^{1,0}\CC \oplus T^{0,1}\CC$ for any vector field $Z$, then $\nabla_Y JW, X \in \Gamma(T^{1,0}\CC)$. As $g_\CC$ is the bilinear extension of the metric on $T\CC$, it follows that $g_\CC(\nabla_Y JW, X) = 0$. Hence, using compatibility with $J$:
	\begin{equation*}
		(\del(r^2 \eta^{(0,1)}))(X, Y) = g_\CC(\nabla_X JW, Y) = -g_\CC(\nabla_X W, JY).
	\end{equation*}
	Decomposing $W$ into real components, we find:
	\begin{align*}
		(\del(r^2 \eta)^{(0,1)})(X, Y) &= - g_\CC(\nabla_X r \del_r, JY) + i g_\CC(\nabla_X r \del_r, IJ Y) \\
		&= - \frac{1}{2}g_\CC(\nabla_X \nabla r^2, JY) + \frac{i}{2}g_\CC(\nabla_X \nabla r^2, K Y) \\
		&= - \frac{1}{2}\Hess(r^2)(X, JY) + \frac{i}{2}\Hess(r^2)(X, KY).
	\end{align*}
	The Hessian satisfies:
	\begin{equation}\label{HessianIsMetric}
		\Hess(r^2)(U, V) = 2g_\CC(U,V)
	\end{equation}
	for any $U,V \in\Gamma ( T\CC)$. Decomposing $X$ and $Y$ into real and imaginary components, using \eqref{HessianIsMetric}, and substituting back, we obtain:
	\begin{align*}
		(\del(r^2 \eta)^{(0,1)})(X, Y) &= - g_\CC(X, JY) + i g_\CC(X, KY).
	\end{align*}
	However, since $X, JY, KY \in \Gamma(T^{1,0} \CC)$ and $g_\CC$ is the bilinear extension of the metric on $T\CC$, then:
	\begin{equation*}
		(\del(r^2 \eta)^{(0,1)})(X, Y) = - g_\CC(X, JY) + i g_\CC(X, KY) = 0.
	\end{equation*}
\end{proof}

\begin{prop}\label{PropGrowthRateToHomogeneousForm}
	Pick $\beta \leq 0 < \alpha$ such that $\beta \in (-1,0]$ and $\alpha>0$ is small. Given a Ricci-flat Kähler cone $(\CC,\omega_\CC)$ of complex dimension $m\geq 2$, if $\eta\in \Omega^2({\CC})$ is a closed $(1,1)$-form satisfying $\tr_{\omega_{{\CC}}} \eta = A$ for some constant $A$, and $|\eta|\leq C\theta_{\beta,\alpha}$, then
	\begin{equation}\label{ExpansionTwoForm}
		\eta = i\del\delbar \left(\frac{A}{2m}r^2 + r^2 \phi \right),
	\end{equation}
	where $r^2\phi$ is harmonic and $\phi$ is an eigenfunction on $L$.
\end{prop}

\begin{proof}
	Pick $\alpha$ so small such that $(2,2+\alpha]\cap \E = \emptyset$. If $\beta \in E$, decrease $\beta$ such that no new elements of $\E$ enter $[\beta,\alpha]$. The first step is to obtain estimates on the derivatives of $\eta$. Pick a point $p\in {\CC}$ and let $R = \dist_{\omega_{\CC}}(o,p)$. Let $\tilde p \coloneqq \Phi^*(p)$ with $\dist_{\omega_{\CC}}(0,\tilde p) = 1$. Pulling $\eta$ back by $\Phi_R$, then by \cite[Lemma 2.1]{chenalpha2015}, for $\epsilon$ small enough, there exists a function $v\in C^{2,\alpha}(B_\epsilon(\tilde p))$ such that
	\begin{equation*}
		i\del\delbar v = \eta, \quad \Norm{v}{0, B_{\frac{\epsilon}{2}}(\tilde p)} \leq C \Norm{\Phi_R^* \eta}{0, B_{\epsilon}(\tilde p)},
	\end{equation*}
	where the constant $C$ only depends on $\epsilon$. Taking the trace, we find $\Delta_{\omega_{\CC}} v = A$, and bootstrapping the Schauder estimates shows that
	\begin{equation*}
		\Norm{v}{k+2, \alpha, B_{\frac{\epsilon}{4}}(\tilde p)} \leq C_k \Norm{\Phi_R^* \eta}{0, B_{\epsilon}(\tilde p)},
	\end{equation*}
	for any $k\in \N_0$. Hence,
	\begin{equation*}
		\Norm{\Phi_R^*\eta}{k, B_{\frac{\epsilon}{4}}(\tilde p)} \leq C_k \Norm{\Phi_R^* \eta}{0, B_{\epsilon}(\tilde p)}.
	\end{equation*}
	Pull back by $\Phi_{R^{-1}}$ to conclude that
	\begin{equation*}
		|\nabla_{\omega_{\CC}}^k \eta| \leq C_kr^{-k}|\eta|\leq Cr^{-k}\theta_{\beta,\alpha}(r)
	\end{equation*}
	for some constant $C>0$ and all $k \in \N_0$.\\
	
	Given the estimate on $\nabla_{\omega_{\CC}}\eta$, by Lemma \ref{PoincareLemma}, there exists a 1-form $\zeta = \zeta^{1,0} + \zeta^{(0,1)}\in \Omega^{1,0}({\CC}) \oplus \Omega^{0,1}({\CC})$ such that $d\zeta = \eta$ and $|\nabla_{\omega_{\CC}}^k \zeta| \leq C_k r^{1-k}\theta_{\beta,\alpha}(r)$. By Lemma \ref{PotentialNiceChoice} there exists a function $\upsilon$ such that $\Delta_{\omega_{\CC}} \upsilon = \delbar^*\zeta^{0,1}$ satisfying $|\nabla_{\omega_{\CC}}^k\upsilon| \leq C_kr^{2-k} \theta_{\beta,\alpha}(r)$. Consider
	\begin{equation*}
		\xi \coloneqq \delbar \upsilon - \zeta^{(0,1)}.
	\end{equation*}
	$\delbar \xi = \delbar^* \xi = 0$ and so $\Delta_{\omega_{\CC}}\xi = 0$ by the Kähler identities. By decomposing the Laplacian acting on harmonic 1-forms as done in the proofs of Lemma \ref{HajoLemmaB.1Modified} and Lemma \ref{DecompositionHoloFields} (for more details, see \cite[Lemma B.1]{heinCalabiYau2017}, more specifically p. 122) and picking $\alpha>0$ small enough, there exists a function $\chi$ such that $\xi = d\chi + r^2\hat \xi$ and $|\nabla_{\omega_{\CC}}^k \chi| \leq C_k r^{2-k}\theta_{\beta,\alpha}(r)$ for $k\in \N_0$. $r^2\hat \xi$ satisfies the properties of Lemma \ref{Lemma1Form}. If $\beta$ was very negative, or $\alpha>0$ was not assumed to be very small, terms of the form $r^{\tau}\eta'$ for $\eta'$ co-closed and $\tau\neq 0$ could appear. These terms do not necessarily satisfy Lemma \ref{Lemma1Form}.	Ricci-flatness of $\omega_\CC$ ensures that the term satisfying $\tau=2$ satisfies Lemma \ref{Lemma1Form}, i.e.
	\begin{equation*}
		\del (r^2 \hat \xi)^{(0,1)} = i\del\delbar q,
	\end{equation*}
	for some 2-homogeneous function $q$ with constant Laplacian.\\
	
	Define:
	\begin{equation*}
		\psi \coloneqq -i( \upsilon - \chi - iq),
	\end{equation*}
	then
	\begin{align*}
		i\del\delbar \psi &= \del\left(\delbar \upsilon - \delbar\upsilon + \zeta^{(0,1)} + (r^2\hat \xi)^{(0,1)} - i\delbar q \right)  \\
		&=\del \zeta^{(0,1)} + i\del\delbar q - i\del\delbar q\\
		&= \del \zeta^{(0,1)}.
	\end{align*}
	We conclude that
	\begin{equation*}
		i\del\delbar\, 2\re(\psi) = i\del\delbar (\psi + \overline \psi) = \del \zeta^{(0,1)} + \delbar \zeta^{1,0} = d\zeta = \eta,
	\end{equation*}
	 so $2\re(\psi)$ is a potential of $\eta$ satisfying $|\nabla_{\omega_\CC}^k 2\re(\psi)| \leq C_k'r^{-k}\theta_{\beta,\alpha}(r)$ for all $k\in \N_0$.\\
	
	To show that $\eta$ is actually 2-homogeneous, note that by assumption:
	\begin{equation*}
			\tr_{\omega_{{\CC}}}\eta = \tr_{\omega_{{\CC}}} i\del\delbar\, 2\re(\psi) = \Delta_{\omega_{{\CC}}} 2\re(\psi) = A,
	\end{equation*}
	for some constant $A\in \R$. Define $\tilde \psi = 2\re(\psi) - A\frac{r^2}{2m}$, then $ \Delta_{\omega_{{\CC}}} \tilde\psi = 0$. We conclude that $\tilde \psi$ has as expansion as given in Proposition \ref{PropGrowthRateToHomogeneousFunction}. If $c_j^\pm \in (1,2)$, then $r^{c_j^\pm}\phi_j$ is a pluriharmonic function by \cite[Theorem 2.14]{heinCalabiYau2017} and as $(\CC,\omega_\CC)$ is Ricci-flat.
\end{proof}
\vspace{10mm}

\renewcommand*{\bibfont}{\small}
\printbibliography

\end{document}